\documentclass[10pt,reqno]{amsart}
\usepackage{latexsym,amsmath,amssymb,amscd}
\def\today{\ifcase \month \or
   January \or February \or March \or April \or
   May \or June \or July \or August \or
   September \or October \or November \or December \fi
   \space\number\day , \number\year}
   
\makeatletter
  \newcommand\@dotsep{4.5}
  \def\@tocline#1#2#3#4#5#6#7{\relax
     \ifnum #1>\c@tocdepth 
     \else
     \par \addpenalty\@secpenalty\addvspace{#2}%
     \begingroup \hyphenpenalty\@M
     \@ifempty{#4}{%
     \@tempdima\csname r@tocindent\number#1\endcsname\relax
        }{%
         \@tempdima#4\relax
           }%
      \parindent\z@ \leftskip#3\relax \advance\leftskip\@tempdima\relax
      \rightskip\@pnumwidth plus1em \parfillskip-\@pnumwidth
       #5\leavevmode\hskip-\@tempdima #6\relax
       \leaders\hbox{$\m@th
       \mkern \@dotsep mu\hbox{.}\mkern \@dotsep mu$}\hfill
       \hbox to\@pnumwidth{\@tocpagenum{#7}}\par
       \nobreak
        \endgroup
         \fi}
\makeatother 

\begin{document}

\makeatletter
\@addtoreset{figure}{section}
\def\thefigure{\thesection.\@arabic\c@figure}
\def\fps@figure{h,t}
\@addtoreset{table}{bsection}

\def\thetable{\thesection.\@arabic\c@table}
\def\fps@table{h, t}
\@addtoreset{equation}{section}
\def\theequation{
\arabic{equation}}
\makeatother

\newcommand{\bfi}{\bfseries\itshape}

\newtheorem{theorem}{Theorem}
\newtheorem{acknowledgment}[theorem]{Acknowledgment}
\newtheorem{algorithm}[theorem]{Algorithm}
\newtheorem{axiom}[theorem]{Axiom}
\newtheorem{case}[theorem]{Case}
\newtheorem{claim}[theorem]{Claim}
\newtheorem{conclusion}[theorem]{Conclusion}
\newtheorem{condition}[theorem]{Condition}
\newtheorem{conjecture}[theorem]{Conjecture}
\newtheorem{construction}[theorem]{Construction}
\newtheorem{corollary}[theorem]{Corollary}
\newtheorem{criterion}[theorem]{Criterion}
\newtheorem{data}[theorem]{Data}
\newtheorem{definition}[theorem]{Definition}
\newtheorem{example}[theorem]{Example}
\newtheorem{lemma}[theorem]{Lemma}
\newtheorem{notation}[theorem]{Notation}
\newtheorem{problem}[theorem]{Problem}
\newtheorem{proposition}[theorem]{Proposition}
\newtheorem{question}[theorem]{Question}
\newtheorem{remark}[theorem]{Remark}
\newtheorem{setting}[theorem]{Setting}
\numberwithin{theorem}{section}
\numberwithin{equation}{section}

\newcommand{\todo}[1]{\vspace{5 mm}\par \noindent
\framebox{\begin{minipage}[c]{0.85 \textwidth}
\tt #1 \end{minipage}}\vspace{5 mm}\par}

\newcommand{\1}{{\bf 1}}

\newcommand{\hotimes}{\widehat\otimes}

\newcommand{\Ad}{{\rm Ad}}
\newcommand{\Alt}{{\rm Alt}\,}
\newcommand{\Ci}{{\mathcal C}^\infty}
\newcommand{\comp}{\circ}
\newcommand{\C}{\text{\bf C}}
\newcommand{\D}{\text{\bf D}}
\newcommand{\ph}{\text{\bf P}}
\newcommand{\de}{{\rm d}}
\newcommand{\ev}{{\rm ev}}
\newcommand{\fimes}{\mathop{\times}\limits}
\newcommand{\id}{{\rm id}}
\newcommand{\ie}{{\rm i}}
\newcommand{\Cp}{\textbf {CPos}\,}
\newcommand{\End}{{\rm End}\,}
\newcommand{\Gr}{{\rm Gr}}
\newcommand{\GL}{{\rm GL}}
\newcommand{\Hilb}{{\bf Hilb}\,}
\newcommand{\Hom}{{\rm Hom}}
\newcommand{\Ker}{{\rm Ker}\,}
\newcommand{\GLH}{\textbf{GrLHer}}
\newcommand{\HLH}{\textbf{HomogLHer}}
\newcommand{\LH}{\textbf{LHer}}
\newcommand{\Kern}{\textbf {Kern}}
\newcommand{\Lie}{\textbf{L}}
\newcommand{\lf}{{\rm l}}
\newcommand{\pr}{{\rm pr}}
\newcommand{\Ran}{{\rm Ran}\,}

\newcommand{\RK}{{\mathcal P}{\mathcal K}^{-*}}
\newcommand{\spann}{{\rm span}}
\newcommand{\SLH}{\textbf {StLHer}}
\newcommand{\Rg}{\textbf {RepGLH}}
\newcommand{\Rep}{\textbf {sRep}}

\newcommand{\Tr}{{\rm Tr}\,}
\newcommand{\Tran}{\textbf{Trans}}

\newcommand{\G}{{\rm G}}
\newcommand{\U}{{\rm U}}
\newcommand{\VB}{{\rm VB}}

\newcommand{\Bc}{{\mathcal B}}
\newcommand{\Cc}{{\mathcal C}}
\newcommand{\Dc}{{\mathcal D}}
\newcommand{\Ec}{{\mathcal E}}
\newcommand{\Gc}{{\mathcal G}}
\newcommand{\Hc}{{\mathcal H}}
\newcommand{\Kc}{{\mathcal K}}
\newcommand{\Oc}{{\mathcal O}}
\newcommand{\Pc}{{\mathcal P}}
\newcommand{\Qc}{{\mathcal Q}}
\newcommand{\Rc}{{\mathcal R}}
\newcommand{\Sc}{{\mathcal S}}
\newcommand{\Tc}{{\mathcal T}}
\newcommand{\Uc}{{\mathcal U}}
\newcommand{\Vc}{{\mathcal V}}
\newcommand{\Xc}{{\mathcal X}}
\newcommand{\Yc}{{\mathcal Y}}
\newcommand{\Zc}{{\mathcal Z}}
\newcommand{\Ag}{{\mathfrak A}}
\newcommand{\hg}{{\mathfrak h}}
\newcommand{\mg}{{\mathfrak m}}
\newcommand{\nng}{{\mathfrak n}}
\newcommand{\pg}{{\mathfrak p}}
\newcommand{\Sg}{{\mathfrak S}}

\markboth{}{}


\makeatletter
\title[Universal objects in categories of reproducing kernels]{Universal objects in categories of reproducing kernels}
\author{Daniel Belti\c t\u a and  Jos\'e E. Gal\'e}
\address{Institute of Mathematics ``Simion
Stoilow'' of the Romanian Academy, 
P.O. Box 1-764, Bucharest, Romania}
\email{Daniel.Beltita@imar.ro}
\address{Departamento de matem\'aticas and I.U.M.A., 
Universidad de Zaragoza, 50009 Zaragoza, Spain}
\email{gale@unizar.es}
\thanks{This research has been partly  supported by Project MTM2007-61446, DGI-FEDER, of the MCYT, Spain; the second author has also been supported by Project E-64, D.G. Arag\'on, Spain, while the first author acknowledges  
partial financial support from the CNCSIS grant PNII - Programme ``Idei'' (code 1194).}

\date{December 1, 2009}
\makeatother

\begin{abstract}
We continue our earlier investigation on generalized reproducing kernels, 
in connection with the complex geometry of $C^*$- algebra representations, 
by looking at them as the objects of an appropriate category.  
Thus the correspondence between reproducing $(-*)$-kernels 
and the associated Hilbert spaces of sections of vector bundles is made into a functor. 
We construct reproducing $(-*)$-kernels with universality properties 
with respect to the operation of pull-back. 
We show how completely positive maps can be regarded as  pull-backs of universal ones  
linked to the tautological bundle over the Grassmann manifold 
of the Hilbert space $\ell^2({\mathbb N})$.
\end{abstract}

\maketitle

\tableofcontents


\section*{Introduction}

The present work belongs to a line of research which has been initiated in 
the papers \cite{BR06}, \cite{BG07},  \cite{BG08},
and concerns representation theory for infinite-dimensional Lie groups on the one hand, 
and an approach to operator algebras by means of techniques from differential geometry 
on the other hand.
In \cite{BR06}, geometric realizations of GNS representations on 
groups of unitaries in $C^*$-algebras are obtained, in the spirit of the Bott-Borel-Weil theorem. 
Hilbertian section spaces of these realizations, in that infinite-dimensional setting, 
are constructed using the tool of reproducing kernels on vector bundles. 
As regards holomorphy, there are a few cases of unitary groups 
(or homogeneous spaces on which those groups act) 
where such a structural property appears automatically; 
see \cite{BR06}, and also \cite{BG07} for (homogeneous) orbits of Grassmannians. 
In order to extend the theory of \cite{BR06} to general holomorphic actions of complex Lie groups on complexified vector bundles, and to incorporate the unitary case to the wider setting of complex geometry, 
a new kind of generalized reproducing kernels on complex vector bundles has been introduced in \cite{BG08}, that takes into account prescribed involutions of the bundle bases. (In passing, the theory enables us to present a holomorphic geometric view of Stinespring dilations of completely positive maps.) 

In this way, we find the notion of (generalized) reproducing kernels on vector bundles in the core of that part of (infinite-dimensional) representation theory which is suitable to be treated using methods of complex geometry. So it seems to be a demanding task to study generalized reproducing kernels in themselves, mainly looking at those aspects of them which are more directly related to the aforementioned subjects. 
One of the relevant facts emerging naturally in the subject is that  
Grassmannian manifolds, over the Hilbert spaces defined by the kernels, seem to be the proper place where all the elements of the geometric theory are settled in a canonical manner; see \cite{BG07}. 
Since Grassmannians arise naturally in geometry as a universal notion,  we wonder if they can be also considered as the universal model for the geometrical implications of reproducing kernels. 
This question is further supported by the feeling that ---loosely speaking--- reproducing kernels are very much alike (so that, in particular for the classical ones, their mutual differences are mainly confined to function theory). 

In order to make this idea precise, we continue here our earlier investigation on generalized reproducing kernels 
by organizing them as a category and looking for the corresponding universal objects.  
Thus the correspondence between reproducing $(-*)$-kernels (see the definition below) 
and the associated Hilbert spaces will be looked at as a functor. 
By following this categorial approach we find that there exist 
reproducing $(-*)$-kernels enjoying universality properties 
with respect to the operation of pull-back. 
In particular it turns out that all of the reproducing kernels 
(in the classical sense, that is, for scalar valued functions) 
arise as pull-backs of a universal reproducing kernel 
that lives on the tautological bundle over the Grassmann manifold 
of the Hilbert space $\ell^2({\mathbb N})$
(see Theorem~\ref{6.1} and the comment preceding it). 

In a purely mathematical direction, investigations on operations on reproducing kernels have been 
previously carried out, with motivation coming from 
complex analysis, function theory, learning theory, and other areas; 
some recent references are \cite{AL} \cite{Kr08a}, \cite{Kr08b}, and \cite{CVTU08}. 
 The theory of kernels in the scalar case was systematically developed in \cite{Ar}.
An extension to vector valued functions has been worked out in \cite{Ku}, 
as well as a systematic account of this theory is included in \cite{Ne00}. 
Reproducing kernels for Hilbert spaces of sections generalize the above classes. 
There are already present in \cite{Ko}, but the first framework where to consider kernels for Hilbert spaces of sections of (finite-dimensional) complex bundles seems to have been suggested in~\cite{BH}. 
The corresponding extension
to the setting of vector bundles with infinite-dimensional fibers 
of infinite-dimensional manifolds is done in \cite{BR06}. The article \cite{BG08} extends in turn the results of \cite{BR06} to reproducing kernels associated with involutive diffeomorphisms in the bundle base spaces.

Reproducing kernels on line bundles over finite-dimensional complex manifolds 
have been on the other hand considered in quantum mechanics as well, related to quantization of classical states; 
see \cite{Od88}, \cite{Od92} and \cite{MP97}, as well as some references therein. 
Roughly speaking, in the model proposed in \cite{Od88} and \cite{Od92} 
a mechanical system is identified with a triple $(Z,\Hc, \zeta\colon Z\to\C\ph(\Hc))$ where $Z$ is a complex manifold seen as the phase space of classical states of the system, $\Hc$ is a Hilbert space and $\C\ph(\Hc)$ is its corresponding complex projective space, which is to play the role of the phase space of pure quantum states. 
Then the (geometric) quantization of the classical phase space $Z$ is supplied by the mapping $\zeta$. 
It turns out that the transition probability amplitudes linked to the phase space $Z$ 
can be formalized in a way that gives rise to a reproducing kernel $K$ defined on 
a complex line bundle with base~$Z$. 
Then the quantization map $\zeta=\zeta_K$ is constructed out of the kernel $K$, 
through the  reproducing kernel Hilbert space $\Hc_K$ associated with~$K$, 
see \cite{Od92} for details. 
The mapping $\zeta$ is very important in this approach since under certain conditions 
it replaces the calculation of the Feynman integral which expresses the transition amplitudes between states. 
Moreover, it is possible to recover the reproducing kernel defining the transition probabilities directly from $\zeta$.

In fact, the relationship between $K$ and $\zeta$ admits an interpretation in a categorial framework such as it has been developed in the paper~\cite{Od92}, with a view toward establishing a canonical relation between 
the classical observables and the quantum ones (see also \cite{Od88}). In this setting, natural ways of viewing states of the matter take the form of three different categories (in two of them the objects are complex line bundles, with a distinguished kernel $K$ in one case and a distinguished Hilbert space $\Hc$ in the other; in the third one the objects are the mappings $\zeta$) and it is possible to show the mutual functorial equivalence between them, see \cite{Od92}, p. 387. Pull-backs of the tautological universal bundle with base $\C\ph(\Hc)$, as well as of its canonical kernel, are used as part of the arguments to prove such equivalences.
 
Apparently, the above physical interpretation and treatment of reproducing kernels has remained unknown for the pure mathematical line of research referred to in former lines. Indeed, we have been aware of the papers \cite{Od88} and \cite{Od92} only after having written the most substantial part of the present article. Since we consider here vector bundles with infinite-dimensional fibers of infinite-dimensional manifolds, our work can be seen partly as a (non-trivial) extension to that setting of the categorial results of \cite{Od92}. However, it must be noticed that, whereas one of the aims of \cite{Od92} is to establish equivalence of categories thought of equivalent ways to viewing states of matter, we are rather interested in pointing out the properties of (generalized) reproducing kernels in themselves, so we present them as {\it objects} of a category and look for their canonical description in terms of their Grassmannian universal  representatives.    

The general plan of the article can be seen prior to this introduction. 
In Section~\ref{sect1} we give some elements for a categorial theory of notions and results 
related to the main items of papers \cite{BR06}, and  \cite{BG08}. 
Consequently in Section~\ref{sect2} subcategories of objects of a {\it positive} character are singled out, 
laying particular emphasis on the concept of reproducing $(-*)$-kernels. 
In this setting, completely positive mappings on Hilbert spaces appear as {\it objects} of a suitable category, 
instead of being considered (as usually in the literature) as morphisms. 
Section~\ref{sect3} collects significant results involving pull-backs of vector bundles and kernels, 
in the setting dealt with in the paper. 
Several items of previous sections are established in Section~\ref{sect4} for Grassmann manifolds, 
as a preparation for the main section of the article, namely Section~\ref{sect5}. 
Here, we establish the universality theorems making (generalized) reproducing kernels $K$ associated with a given (like-Hermitian) vector bundle 
$\Pi\colon D\to Z$
to appear as pull-backs of canonical ones living on Grassmannians. The essential tool to do that is a mapping $\zeta_K\colon Z\to\Gr(\Hc^K)$ like that one of above, but this time taking values in all of the Grassmann manifold $\Gr(\Hc^K)$ formed by all closed subspaces of $\Hc^K$. Here $\Hc^K$ is the reproducing kernel Hilbert space defined by $K$. 
Such a mapping  $\zeta_K$ is of course an extension of the corresponding mappings considered in \cite{Od92}, and indeed was suggested by an intermediate extension introduced in \cite{MP97}, in a finite-dimensional context. 
By mimicking names originated from physical considerations, we are tempted to call the mapping $\zeta_K$ the {\it quantization} map of the bundle $\Pi$, and every element $\zeta_K(s)$, $s\in Z$, a {\it coherent state} for $\Pi$.  
 
Finally, in Section~\ref{sect6} we give some applications or examples. 
We find particularly interesting the fact that {\it any} completely positive mapping taking values in the algebra of bounded operators on a Hilbert space, in analogy with what happens for kernels, can be obtained as the pull-back of a canonical (and universal) completely positive map associated with the tautological bundle of a Grassmann manifold.

\section{A categorial  framework for like-Hermitian structures}\label{sect1}

\subsection{Like-Hermitian structures}\label{subsect1.1}

We briefly review here the like-Hermitian vector bundles introduced in \cite{BG08} 
in order to study geometric models for representations of Banach-Lie groups and $C^*$-algebras. 

\begin{definition}\label{like}
\normalfont
Let $Z$ be a real Banach manifold with an involutive diffeomorphism
$z\mapsto z^{-*}$, $Z\to Z$,
that is, $(z^{-*})^{-*}=z$ for all $z\in Z$.
A {\it like-Hermitian structure} on a smooth vector bundle $\Pi\colon D\to Z$ 
(whose typical fiber is a complex Banach space $\Ec$) is a family
$\{(\cdot\mid\cdot)_{z,z^{-*}}\}_{z\in Z}$
with the following properties:
\begin{itemize}
\item[{\rm(a)}]
For every $z\in Z$,
$(\cdot\mid\cdot)_{z,z^{-*}}\colon D_z\times D_{z^{-*}}\to{\mathbb C}$
is  a sesquilinear strong duality pairing.
\item[{\rm(b)}]
For all $z\in Z$, $\xi\in D_z$, and $\eta\in D_{z^{-*}}$
we have $\overline{(\xi\mid\eta)}_{z,z^{-*}}=(\eta\mid\xi)_{z^{-*},z}$.
\item[{\rm(c)}]
If $V$ is an arbitrary open subset of $Z$,
and $\Psi_V\colon V\times\Ec\to\Pi^{-1}(V)$
and $\Psi_{V^{-*}}\colon V^{-*}\times\Ec\to\Pi^{-1}(V^{-*})$
are trivializations of the vector bundle $\Pi$
over $V$ and $V^{-*}$ ($:=\{z^{-*}\mid z\in V\}$),
respectively,
then the function
$(z,x,y)\mapsto(\Psi_V(z,x)\mid\Psi_{V^{-*}}(z^{-*},y))_{z,z^{-*}}$,
$V\times \Ec\times\Ec\to{\mathbb C}$
is smooth.
\end{itemize}
\qed
\end{definition}

\begin{remark}\label{sesqui}
\normalfont
Condition~(a) in Definition~\ref{like} means that
$(\cdot\mid\cdot)_{z,z^{-*}}\colon D_z\times D_{z^{-*}}\to{\mathbb C}$
is continuous, is linear in the first variable
and antilinear in the second variable, and both the mappings
$$
\xi\mapsto(\xi\mid\cdot)_{z,z^{-*}},\quad D_z\to(\overline{D}_{z^{-*}})^*,
\quad\text{ and }\quad
\eta\mapsto(\cdot\mid\eta)_{z,z^{-*}},\quad \overline{D}_{z^{-*}}\to D_z^*,
$$
are
(not necessarily isometric) isomorphisms of complex Banach spaces.
Here we denote, for any complex Banach space $\Zc$,
by $\Zc^*$ its dual (complex) Banach space
and by $\overline{\Zc}$ the complex-conjugate Banach space.
That is, the {\it real} Banach spaces underlying
$\Zc$ and $\overline{\Zc}$ coincide, and
for any $z$ in the corresponding real Banach space
and $\lambda\in{\mathbb C}$
we have
$\lambda\cdot z\;\; \text{(in $\overline{\Zc}$)}
=\overline{\lambda}\cdot z\;\; \text{(in ${\Zc}$)}$.
\qed
\end{remark}

\begin{remark}\label{sesqui1}
\normalfont
For later use we now record the following fact:
Let $\Xc$ and $\Yc$ be complex  Banach spaces with a sesquilinear strong duality pairing 
$(\cdot\mid\cdot)\colon\Xc\times\Yc\to{\mathbb C}$.
If $\Hc$ is a complex Hilbert space and
$T\colon\Hc\to\Xc$ is a bounded linear operator, 
then there exists a unique operator
$S\colon\Yc\to\Hc$ such that
\begin{equation}\label{adj}
(\forall h\in\Hc,y\in\Yc)\quad
(Th\mid y)=(h\mid Sy)_{\Hc}.
\end{equation}
Conversely,  for every
bounded linear operator $S\colon\Yc\to\Hc$ there exists
a unique bounded linear operator $T\colon\Hc\to\Xc$ satisfying~\eqref{adj},
and we denote $S^{-*}:=T$ and $T^{-*}:=S$.
\qed
\end{remark}

Like-Hermitian bundles admit a natural notion of morphism from one into another.

\begin{definition}\label{morph}
\normalfont
Let $\widetilde{\Pi}\colon \widetilde{D}\to\widetilde{Z}$ and 
$\Pi\colon D\to Z$ be 
like-Hermitian vector bundles, and assume that 
each of the manifolds $\widetilde{Z}$ and $Z$ is endowed with an involutive 
diffeomorphism denoted by $z\mapsto z^{-*}$ for both manifolds. 
A {\it morphism} (respectively, an {\it antimorphism}) 
of $\widetilde{\Pi}$ into $\Pi$ is a pair 
 $\Theta=(\delta,\zeta)$ 
such that 
$\delta\colon\widetilde{D}\to D$ and 
$\zeta\colon\widetilde{Z}\to Z$ 
are smooth mappings satisfying the following conditions: 
\begin{itemize}
\item[{\rm(i)}] 
The diagram 
$$
\begin{CD}
\widetilde{D} @>{\delta}>> D \\
@V{\widetilde{\Pi}}VV @VV{\Pi}V \\
\widetilde{Z} @>{\zeta}>> Z
\end{CD}
$$
is commutative. 
\item[{\rm(ii)}] 
The mapping 
$\delta_z:=\delta|_{\widetilde{D}_z}\colon\widetilde{D}_z\to D_{\zeta(z)}$ 
is a bounded linear operator whenever $z\in\widetilde{Z}$ 
(respectively, a bounded antilinear operator whenever $z\in\widetilde{Z}$). 
\item[{\rm(iii)}] 
For all $z\in\widetilde{Z}$ we have $\zeta(z^{-*})=\zeta(z)^{-*}$. 
\end{itemize}
Because of the assumption that the sesquilinear functionals 
$(\cdot\mid\cdot)_{z,z^{-*}}$ are strong duality pairings, 
it follows that the morphism $\Theta=(\delta,\zeta)$ is 
{\it quasi-adjointable} in the sense that for every $z\in\widetilde{Z}$ 
there exists a (unique) bounded linear operator 
$(\delta_z)^{-*} \colon D_{\zeta(z)^{-*}}\to\widetilde{D}_{z^{-*}}$ 
such that 
\begin{equation}\label{quasi-adjointable}
(\forall\xi\in\widetilde{D}_z,\eta\in D_{\zeta(z)^{-*}})\quad 
(\delta_z\xi\mid\eta)_{\zeta(z),\zeta(z)^{-*}}
=(\xi\mid(\delta_z)^{-*}\eta)_{z,z^{-*}}\ \ .
\end{equation}
 If $\Theta=(\delta,\zeta)$ is an antimorphism, 
then it is again {\it quasi-adjointable}  
since for every $z\in\widetilde{Z}$ 
there exists a (unique) continuous antilinear operator 
$(\delta_z)^{-*}\colon D_{\zeta(z)^{-*}}\to\widetilde{D}_{z^{-*}}$ 
such that 
\begin{equation}\label{quasi-adjointable_anti}
(\forall\xi\in\widetilde{D}_z,\eta\in D_{\zeta(z)^{-*}})\quad 
(\delta_z\xi\mid\eta)_{\zeta(z),\zeta(z)^{-*}}
=\overline{(\xi\mid(\delta_z)^{-*}\eta)}_{z,z^{-*}}\ \ .
\end{equation}
Now assume that $\zeta\colon\widetilde{Z}\to Z$  
is a diffeomorphism. 
We say that the morphism $\Theta=(\delta,\zeta)$ 
is {\it adjointable} if there exists 
a morphism $\Theta^{-*}=(\delta^{-*},\zeta^{-1})$ 
from $\Pi$ into $\widetilde{\Pi}$ 
such that \eqref{quasi-adjointable} is satisfied 
with $(\delta_z)^{-*}:=\delta^{-*}|_{D_{\zeta(z^{-*})}}
\colon D_{\zeta(z^{-*})}\to\widetilde{D}_{z^{-*}}$. 
If this is the case, then any morphism $\Theta^{-*}$ 
with this property is said to be {\it adjoint to} $\Theta$. 
The {\it adjointable antimorphisms} are defined in a similar manner 
and an adjoint of an antimorphism is by definition an antimorphism. 
\qed
\end{definition}

\begin{definition}\label{isometric}
\normalfont
Let $\widetilde{\Pi}\colon \widetilde{D}\to\widetilde{Z}$ and 
$\Pi\colon D\to Z$ be two 
smooth like-Hermitian vector bundles. 
Assume that $\Theta=(\delta,\zeta)$ is a morphism 
from $\widetilde{\Pi}$ into $\Pi$. 
We say that $\Theta$ is an \emph{isometry} if 
it satisfies the condition 
$(\xi\mid\eta)_{z,z^{-*}}=
(\delta(\xi)\mid\delta(\eta))_{\zeta(z),\zeta(z)^{-*}} $
whenever $z\in\widetilde{Z}$, $\xi\in\widetilde{D}_z$, 
and $\eta\in\widetilde{D}_{z^{-*}}$. 
Similarly, if $\Theta=(\delta,\zeta)$ is an antimorphism, 
then $\Theta$ is an \emph{isometry} if and only if we have 
$(\xi\mid\eta)_{z,z^{-*}}=
\overline{(\delta(\xi)\mid\delta(\eta))}_{\zeta(z),\zeta(z)^{-*}} $
whenever $z\in\widetilde{Z}$, $\xi\in\widetilde{D}_z$, 
and $\eta\in\widetilde{D}_{z^{-*}}$. 
\qed
\end{definition}

\begin{remark}\label{morph2}
\normalfont
In Definition~\ref{isometric}, 
if $\Theta=(\delta,\zeta)$ is an isometry  
and the mapping
$\delta_z:=\delta|_{\widetilde{D}_z}\colon\widetilde{D}_z\to D_{\zeta(z)}$ 
is bijective for all $z\in\widetilde{Z}$, 
then $\Theta$ is quasi-adjointable and the appropriate
condition \eqref{quasi-adjointable} or \eqref{quasi-adjointable_anti}
is satisfied  if we take 
$(\delta_z)^{-*}:=(\delta_{z^{-*}})^{-1}\colon 
D_{\zeta(z)^{-*}}\to\widetilde{D}_{z^{-*}}$
for each $z\in\widetilde{Z}$. 
This also shows that if $\zeta\colon\widetilde{Z}\to Z$ 
is a diffeomorphism and $\Theta=(\delta,\zeta)$ is an 
isometric morphism (respectively, antimorphism) 
that is fiberwise bijective
from $\widetilde{\Pi}$ to $\Pi$, then $\Theta$ is adjointable 
and its inverse $\Theta^{-1}=(\delta^{-1},\zeta^{-1})$ 
is an adjoint to~$\Theta$. 
\qed
\end{remark}

It is plainly seen that like-Hermitian vector bundles, as the objects, and morphisms in between form a category which will be denoted here by $\LH$. 
There are several important examples of the above bundles which we think is worthwhile to present under the categorial language. In the next subsection we consider like-Hermitian vector bundles acted on by Banach-Lie groups.

\subsection{Group actions on vector bundles and representations}\label{subsect1.2}

Recall that an {\it involutive} Banach-Lie group is a (real or complex)
Banach-Lie group $G$ equipped with a diffeomorphism $u\mapsto u^*$ satisfying
$(uv)^*=v^*u^*$ and $(u^*)^*=u$ for all $u,v\in G$.
In this case we denote
$$
 u^{-*}:=(u^{-1})^* \quad\  (\forall u\in G)\ \ 
\hbox{ and } \ \  G^{+}:=\{u^*u\mid u\in G\},
$$
and the elements of $G^{+}$ are called the {\it positive} elements of $G$.
If $H$ is a Banach-Lie subgroup of $G$,
then we say that $H$ is an {\it involutive} Banach-Lie subgroup if in addition $u^*\in H$
whenever $u\in H$.
If $G$ is an involutive Banach-Lie group then for every $u\in G$
we have $(u^{-1})^*=(u^*)^{-1}$ and moreover $\1^*=\1$.

\begin{definition}\label{relationship}
\normalfont
Assume that we have a complex involutive Banach-Lie group~$G$,  
a holomorphic like-Hermitian vector bundle 
$\Pi\colon D\to Z$, and two holomorphic actions $\mu$ and $\nu$  of the group $G$ on $D$ and $Z$ respectively, 
such that  
\begin{itemize}
\item[{\rm(i)}] there exists the commutative diagram
$$
\begin{CD}
G\times D @>{\mu}>> D \\
@V{\id_G\times\Pi}VV @VV{\Pi}V \\
G\times Z @>{\nu}>> Z
\end{CD}
$$
and for all $u\in G$ and $z\in Z$ the mapping 
$\mu(u,\cdot)|_{D_z}\colon D_z\to D_{\nu(u,z)}$
is a bounded linear operator,
\item[{\rm(ii)}] for every  $u\in G$ and $z\in Z$  we have
$\nu(u^{-*},z^{-*})=\nu(u,z)^{-*}$,
\item[{\rm(iii)}] for every $z\in Z$, $u\in G$; $\xi\in D_z$, $\eta\in D_{\nu(u^{-*},z^{-*})}$ we have
$$
(\mu(u^*,\eta)\mid\xi)_{z^{-*},z}=(\eta\mid\mu(u,\xi))_{\nu(u^{-*},z^{-*}),\nu(u,z)}.
$$
\end{itemize}
With the notation $(\Pi,G)$ we refer to a like-Hermitian vector bundle $\Pi$ 
which is acted on by the Banach-Lie group $G$ 
in the sense defined by the preceding properties.

Let $(\widetilde\Pi,\widetilde G)$ and $(\Pi,G)$ be two such elements. 
We say that a pair $(\Theta,\alpha)$ is a {\it morphism} from 
$(\widetilde\Pi,\widetilde G)$ into $(\Pi,G)$ if $\Theta:=(\delta, \zeta)$ is a morphism between the bundles $\widetilde\Pi$ and $\Pi$, 
$\alpha\colon\widetilde G\to G$ is a holomorphic group homomorphism which preserves involutions, and the following commutativity of diagrams holds:
\begin{itemize}  
\item[{\rm(a)}] $(\id_G\times\Pi)\circ(\alpha\times\delta)=(\alpha\times\zeta)\circ(\id_{\widetilde G}\times\widetilde\Pi)$, 
that is, 
$$\begin{CD}
{\widetilde G}\times{\widetilde D} @>{\alpha\times\delta}>> G\times D \\
@V{\id_{\widetilde G}\times{\widetilde\Pi}}VV @VV{\id_G\times\Pi}V \\
{\widetilde G}\times{\widetilde Z} @>{\alpha\times\zeta}>> G\times Z   
  \end{CD}
 $$
\item[{\rm(b)}] $\delta\circ\tilde\mu=\mu\circ(\alpha\times\delta)$ and $\zeta\circ\tilde\nu=\nu\circ(\alpha\times\zeta)$, 
that is, 
$$\begin{CD}
{\widetilde G}\times{\widetilde D} @>{\alpha\times\delta}>> G\times D \\
@V{\widetilde\mu}VV @VV{\mu}V \\
{\widetilde D} @>{\delta}>> D
  \end{CD}
\qquad\text{and}\qquad
\begin{CD}
 {\widetilde G}\times{\widetilde Z} @>{\alpha\times\zeta}>> G\times Z \\
@V{\widetilde\nu}VV @VV{\nu}V \\
{\widetilde Z} @>{\zeta}>> Z
\end{CD}
$$
\end{itemize}
It is then straightforward to check that we have got a category where pairs
$(\Pi,G)$ are the objects and pairs $(\Theta,\alpha)$ are the morphisms. 
Let $\GLH$ denote this category.
\qed
\end{definition}

Important examples of objects in $\GLH$ are suitably related to 
representations of groups acting in that category. 
Specifically, we shall make the following definition which 
claims its origins in the idea of \emph{basic mapping} that shows up in Example~2.5 
in \cite{BG08}.

\begin{definition}\label{transrel} 
\normalfont 
In the setting of Definition \ref{relationship}, let further assume that we have given 
a holomorphic $*$-representation $\pi\colon G\to\Bc(\Hc)$.

We say that the mapping $\Rc\colon D\to\Hc$ \textit{relates} $\Pi$ to $\pi$ 
if it has the following properties: 
\begin{itemize}
\item[{\rm(i)}] $\Rc$ is holomorphic; 
\item[{\rm(ii)}] for each $z\in Z$ the mapping 
$\Rc_z:=\Rc|_{D_z}\colon D_z\to\Hc$ 
is an injective bounded linear operator and 
in addition we have  
$(\xi\mid\eta)_{z,z^{-*}}=(\Rc(\xi)\mid\Rc(\eta))_{\Hc}$
whenever $\xi\in D_z$ and $\eta\in D_{z^{-*}}$; 
\item[{\rm(iii)}] for every $u\in G$ and $z\in Z$ the diagram 
\end{itemize}
$$\begin{CD}
D_z @>{\mu(u,\cdot)|_{D_z}}>> D_{\nu(u,z)} \\
@V{\Rc_z}VV @VV{\Rc_{\nu(u,z)}}V \\
\Hc @>{\pi(u)}>> \Hc 
\end{CD}
$$
is commutative. 
We alternatively call $\Rc$ the {\it transfer} mapping between $\Pi$ and $\pi$ 
(or $D$ and $\Hc$).
\qed
\end{definition}

In the case when an object $(\Pi,G)$ in the category $\GLH$ is associated with a representation $\pi$ and a 
transfer mapping $\Rc$ as in the above definition we shall refer to $(\Pi,G)$ by writing $(\Pi,G;\pi)$. 
Given two of these objects  
$(\widetilde\Pi,\widetilde G;\tilde\pi)$, $(\Pi,G;\pi)$ 
we say that a triple $(\Theta,\alpha,L)$ is a {\it morphism} from 
$(\widetilde\Pi,\widetilde G;\tilde\pi)$ into $(\Pi,G;\pi)$ 
whenever  $(\Theta,\alpha)$ is a morphism between 
$(\widetilde\Pi,\widetilde G)$ and $(\Pi,G)$ in the sense of Definition~\ref{relationship}, 
and $L\colon\widetilde\Hc\to\Hc$ is a bounded linear mapping such that:
\begin{itemize}
\item[{\rm(c)}]  For all $u\in\widetilde G$, $\pi(\alpha(u))\circ L=L\circ\tilde\pi(u)$.
\item[{\rm(d)}] $\Rc\circ\delta=L\circ\widetilde\Rc$.
\end{itemize}
Then triples $(\Pi,G;\pi)$ are the objects of a category, which we choose to call $\Rg$, 
whose morphisms are the triples $(\Theta,\alpha,L)$.
Of course properties (c) and (d) of $L$ are to be added to previous (a) and (b) for $(\Theta,\alpha)$.

There is an important subcategory of $\GLH$ (and of $\Rg$) 
where the base spaces of the corresponding vector bundles are homogeneous manifolds. 
These bundles have been considered in \cite{BG08}, in connection with 
holomorphic realizations of representations. 
We recall their definition in the next subsection.

\subsection{Homogeneous like-Hermitian vector bundles}\label{subsect1.3}

Assume that we have the following data:
\begin{itemize}
\item[$\bullet$] $G_A$ is an involutive real
(respectively, complex) Banach-Lie group and $G_B$
is an involutive real (respectively, complex) Banach-Lie subgroup of $G_A$.
\item[$\bullet$] For $X=A$ or $X=B$, assume $\Hc_X$ is a complex Hilbert space
with $\Hc_B$ closed subspace in $\Hc_A$, and
$\pi_X\colon G_X\to \Bc(\Hc_X)$ is a uniformly continuous
(respectively, holomorphic)
$*$-representation such that
$\pi_B(u)=\pi_A(u)|_{\Hc_B}$ for all $u\in G_B$.
By $*$-representation we mean that $\pi_A(u^*)=\pi_A(u)^*$
for all $u\in G_A$.
\item[$\bullet$] We denote by $P\colon\Hc_A\to\Hc_B$
the orthogonal projection.
\end{itemize}
Define an equivalence relation on
$G_A\times\Hc_B$ by
$(u,f)\sim(u',f')$ whenever there exists $w\in G_B$ such that $u'=uw$ and  $f'=\pi_B(w^{-1})f$.
For every pair $(u,f)\in G_A\times\Hc_B$ its equivalence class is denoted by
$[(u,f)]$ and the set of all equivalence classes is denoted by $D=G_A\times_{G_B}\Hc_B$.
Clearly, there exists a natural onto map
$$
\Pi\colon\,[(u,f)]\mapsto s:=u\ G_B,\quad D\to G_A/G_B.
$$
For $s\in G_A/G_B$, let $D_s:=\Pi^{-1}(s)$ denote the {\it fiber} on $s$.
Note that $(u,f)\sim(u',f')$ implies that $\pi_A(u)f=\pi_A(u')f'$
so that the correspondence
$[(u,f)]\mapsto\pi_A(u)f$, $D_s\to\pi_A(u)\Hc_B$,
gives rise to a complex linear structure on $D_s$, and a Hilbertian norm on $D_s$
defined by
$\Vert[(u,f)]\Vert_{D_s}:=\Vert\pi_A(u)f\Vert_{\Hc_A}$
where $[(u,f)]\in D_s$. 
Moreover, the formula 
\begin{equation}\label{strong}
\bigl([(u,f)]\mid[(u^{-*},g)]\bigr)_{s,s^{-*}}
=\bigl(\pi_A(u)f\mid\pi_A(u^{-*})g\bigr)_{\Hc_A}
=\bigl(\pi_A(u^{-1})\pi_A(u)f\mid g\bigr)_{\Hc_A}
=\bigl(f\mid g\bigr)_{\Hc_B},
\end{equation}
for $[(u,f)]\in D_s$ and [$(u^{-*},g)]\in D_{s^{-*}}$, 
defines a like-Hermitian structure 
$\bigl(\cdot\mid\cdot\bigr)_{s,s^{-*}}$
on the vector bundle~$\Pi$. 
See more details in~\cite{BG08}. 

So, in the case where $G_A$, $G_B$ are complex and $\pi_A$, $\pi_B$ are holomorphic the bundle $\Pi$ is an example of 
Definition \ref{relationship} with respect to the representation 
$\pi_A\colon G_A\to\Bc(\Hc_A)$ and the transfer mapping given by
$$
\Rc\colon\,[(u,f)]\mapsto\pi_A(u)f,\quad D\to\Hc_A.
$$
As in \cite{BG08}, we shall say that $\Pi\colon G_A\times_{G_B}\Hc_B\to G_A/G_B$ is the 
{\it homogeneous like-Hermitian vector bundle}
associated with the data~$(\pi_A,\pi_B,P)$.

We can consider the data $(\pi_A,\pi_B,P)$ as objects of a category 
$\HLH$ where the morphisms are given by pairs $(\alpha,L)$ from
$(\tilde\pi_A,\tilde\pi_B,\tilde P)$ into $(\pi_A,\pi_B,P)$ such that:
\begin{itemize}
\item[(i)] $\alpha\colon\widetilde G_A\rightarrow G_A$ is a continuous 
(or holomorphic when $G_A$ and $G_B$ are complex) group homomorphism preserving involutions, with $\alpha(\widetilde G_B)\subseteq G_B$.
\item[(ii)] $L\colon\widetilde\Hc_A\rightarrow\Hc_A$ is a bounded linear mapping such that $L\circ\tilde\pi_A(u)=\pi_A(\alpha(u))\circ L$ 
for every $u\in\widetilde G_A$ and 
$P\circ L\mid_{\widetilde\Hc_B}=L\circ\widetilde P$ or, equivalently, 
$L(\widetilde\Hc_B)\subseteq\Hc_B$. 
\end{itemize}
We observe that the conditions on $(\alpha,L)$ allow us in particular to get the well-defined mapping 
$$
[\alpha\times L]\colon[(u,f)]\mapsto[(\alpha(u),L(f))],\  
\widetilde G_A\times_{\widetilde G_B}\widetilde\Hc_B
\rightarrow G_A\times_{G_B}\Hc_B. 
$$
The morphism from $\widetilde G_A\times_{\widetilde G_B}\widetilde\Hc_B
\rightarrow \widetilde G_A/\widetilde G_B$ into $ G_A\times_{G_B}\Hc_B
\rightarrow G_A/ G_B$ is $\Theta=(\delta,\zeta)$ where 
$\delta=[\alpha\times L]$ and 
$\zeta=\alpha_q\colon u\widetilde G_B\mapsto\alpha(u)G_B,\ \widetilde G_A/\widetilde G_B\rightarrow G_A/G_B$. Moreover, it is an exercise of some patience to check that, for the transfer mapping defined just above and the natural multiplications, the morphism properties (a), (b), (c), (d) given in the definition of the category $\Rg$ hold for $\Theta=([\alpha\times L],\alpha_q)$.  

In the following subsection we consider a remarkable class of objects in the category $\HLH$.

\subsection{Vector bundles arising from completely positive maps on $C^*$-algebras}\label{subsect1.4}

In this subsection we deal with completely positive maps on $C^*$-algebras valued in algebras of 
bounded operators on Hilbert spaces. 
Let us briefly review the Stinespring construction of dilations of completely positive maps. 
For details and references, see for instance \cite{BG08}.

For every complex vector space $X$ let $M_n(X)$ denote the space formed by all matrices 
$n\times n$ with entries in $X$. 
Then, for every linear map $\Phi\colon X\to Y$ between two vector spaces $X$ and $Y$ and 
every integer $n\ge1$, put 
$\Phi_n=\Phi\otimes\id_{M_n({\mathbb C})}\colon M_n(X)\to M_n(Y)$, 
that is, $\Phi_n((x_{ij})_{1\le i,j\le n})=(\Phi(x_{ij}))_{1\le i,j\le n}$ 
for every matrix $(x_{ij})_{1\le i,j\le n}\in M_n(X)$.

Let $A$ be a unital $C^*$-algebra and let $\Hc_0$ be a complex Hilbert space. A linear map 
$\Phi\colon A\to\Bc(\Hc_0)$ 
is said to be a {\it completely positive} map if 
for every integer $n\ge1$ the map $\Phi_n\colon M_n(A)\to M_n(\Bc(\Hc_0))$ is positive 
in the sense that it takes positive 
elements in the $C^*$-algebra 
$M_n(A)$ to positive ones in $M_n(\Bc(\Hc_0))$. 
If moreover $\Phi(\1)=\1$ then we say that $\Phi$ is unital and in this case 
we have $\Vert\Phi_n\Vert=1$ for every $n\ge1$.
Define a nonnegative sesquilinear form on $A\otimes\Hc_0$ by the formula 
$$
\Bigl(\sum_{j=1}^nb_j\otimes\eta_j\mid\sum_{i=1}^n a_i\otimes\xi_i\Bigr)
=\sum_{i,j=1}^n(\Phi(a_i^*b_j)\eta_j\mid\xi_i)
$$
for $a_1,\dots,a_n,b_1,\dots,b_n\in A$, 
$\xi_1,\dots,\xi_n,\eta_1,\dots,\eta_n\in\Hc_0$ 
and $n\ge1$. Set $N_A=\{x\in A\otimes\Hc_0\mid (x\mid x)=0\}$ 
and denote by $\Kc_0$ the Hilbert space obtained as 
the completion of $(A\otimes\Hc_0)/N_A$ 
with respect to the scalar product defined by $(\cdot\mid\cdot)$ 
on this quotient space. 
One can define a representation $\widetilde{\pi}$ of $A$ 
by linear maps 
on $A\otimes\Hc_0$ given by 
$$
(\forall a,b\in A)(\forall \eta\in\Hc_0)\qquad
\widetilde{\pi}(a)(b\otimes\eta)=ab\otimes\eta.
$$
Then every linear map 
$\widetilde{\pi}(a)\colon A\otimes\Hc_0\to A\otimes\Hc_0$ 
induces a continuous map $(A\otimes\Hc_0)/N_A\to(A\otimes\Hc_0)/N_A$, 
whose extension by continuity will be denoted by $\pi_\Phi(a)\in\Bc(\Kc_0)$. 
In this way we obtain a unital $*$-representation 
$\pi_\Phi\colon A\to\Bc(\Kc_0)$
which is called the {\it Stinespring representation associated with $\Phi$}. 

Additionally, denote by $V\colon\Hc_0\to\Kc_0$ 
the bounded linear map obtained as 
the composition 
$$
V\colon\Hc_0 \to A\otimes\Hc_0\to(A\otimes\Hc_0)/N_A\hookrightarrow\Kc_0,
$$
where the first map is defined by $\Hc_0\ni h\mapsto\1\otimes h\in A\otimes\Hc_0$ 
and the second map is the natural quotient map. 
Then $V\colon\Hc_0\to\Kc_0$ is an isometry satisfying 
$\Phi(a)=V^*\pi_\Phi(a)V$ for all $a\in A$. In this sense $\pi_\Phi$ is called a {\it Stinespring dilation} of $\Phi$. 

Note that if $\dim\Hc_0=1$, that is, 
$\Phi$ is a state of $A$, 
the Stinespring representation associated with $\Phi$ coincides with 
the corresponding Gelfand-Naimark-Segal (GNS) representation. 

Particular cases of completely positive maps are the {\it conditional expectations}:
Let $\1\in B\subseteq A$ be two $C^*$-algebras. 
A conditional expectation $E\colon A\to B$ is a linear map satisfying $E^2=E$, $\Vert E\Vert=1$ and $\Ran E=B$. 
Then it follows by the theorem of Tomiyama that $E(\1)=\1$ and 
for every $a\in A$, $b_1,b_2\in B$ we have $E(a^*)=E(a)^*$, $0\le E(a)^*E(a)\le E(a^*a)$, and 
$E(b_1ab_2)=b_1E(a)b_2$. 
We have the following result 
(Lemma~6.7 in \cite{BG08}).

\begin{lemma}\label{two_squares}
Let $B\subseteq A$ be two unital $C^*$-algebras  
with a conditional expectation $E\colon A\to B$ 
and a unital completely positive map $\Phi\colon A\to\Bc(\Hc_0)$
satisfying $\Phi\circ E=\Phi$, where $\Hc_0$ is a complex Hilbert space.

Now let $\pi_A\colon A\to\Bc(\Hc_A)$ and $\pi_B\colon B\to\Bc(\Hc_B)$ 
be the Stinespring representations associated with 
the unital completely positive maps 
$\Phi$ and $\Phi|_B$, respectively. 
Then $\Hc_B\subseteq\Hc_A$, and for every
$h_0\in\Hc_0$ and 
$b\in B$ we have the commutative diagrams 
$$
\begin{CD}
A @>{\iota_{h_0}}>> \Hc_A @>{\pi_A(b)}>> \Hc_A \\
@V{E}VV      @V{P}VV              @VV{P}V \\
B @>{\iota_{h_0}}>> \Hc_B @>{\pi_B(b)}>> \Hc_B
\end{CD}
$$
where $P\colon\Hc_A\to\Hc_B$ is the orthogonal projection, and  
$\iota_{h_0}\colon A\to\Hc_A$ is the map induced by $a\mapsto a\otimes h_0$. 
\end{lemma}

\begin{remark}\label{two_sq} 
\normalfont
The projection $P$ in Lemma~\ref{two_squares} 
is indeed obtained from $E$ as  the limit operator of expressions of the type $P((\sum a_i\otimes y_i) +N_A):=\sum (E(a_i)\otimes y_i) +N_B$, for every $\sum a_i\otimes y_i\in A\otimes\Hc_A$.
\qed
\end{remark}

Let $\G_X$ be the group of invertible elements of $X$ for $X=A$ or $X=B$. 
Then $\G_X$ has a natural structure of involutive complex Banach-Lie group defined by the involution of $X$. 
Moreover, $\G_B$ is a complex Banach-Lie subgroup of $\G_A$. We call the triple $(\pi_A|_{\G_A},\pi_B|_{\G_B},P)$ 
the {\it Stinespring data} associated with $E$ and $\Phi$.  It is clear that in this way we have obtained an object of 
the category $\HLH$.
In fact, they form a subcategory of $\HLH$ together with the morphisms defined as follows.

Let $(\pi_{\tilde A},\pi_{\tilde B},\tilde P)$ and $(\pi_A,\pi_B,P)$ be two Stinespring data. A morphism from 
$(\pi_{\tilde A},\pi_{\tilde B},\tilde P)$ into $(\pi_A,\pi_B,P)$ is any pair $(\alpha, T)$ such that 
$\alpha\colon\tilde A\rightarrow A$ 
is a unital $*$-homomorphism and $T\colon \tilde\Hc_0\rightarrow\Hc_0$ is a bounded linear mapping satisfying  
\begin{itemize}
\item[(a)] $\alpha\circ\widetilde E\mid_{\widetilde B}
=E\circ\alpha\mid_{\widetilde B}$,
\item[(b)] $\Phi(\alpha(b))\circ T=T\circ\widetilde\Phi(b)$ for every $b\in\tilde B$,
\end{itemize}
where the symbols involved in the above equalities have an obvious meaning.
This category will be denoted as $\SLH$. 
We rewrite the objects 
$(\pi_A,\pi_B,P)$ as $(A,B,E;\Phi)$.

To conclude this section, we remark that the categories introduced above are related through the following  functors:
\begin{enumerate}
\item The functor $\GLH\rightarrow\LH$ just given by the inclusion.
\item The functor $\HLH\rightarrow\Rg$ defined by
$$
(\pi_A,\pi_B,P)\mapsto(\Pi,\pi;\Rc)\quad
\hbox{ and }\quad (\alpha,L)\mapsto(\Theta,\alpha,L)
$$
where for a data $\pi_A\colon G_A\to\Bc(\Hc_A)$, 
$\pi_B\colon G_B\to\Bc(\Hc_B)$, $P\colon\Hc_A\to \Hc_B$, we take
$$ 
\Pi\colon[(u,f)]\mapsto u\G_B,\ G_A\times_{\G_B}\Hc_B\rightarrow G_A/G_B;\ \ 
\pi=\pi_A;  \ \Rc([(u,f)])=\pi_A(u)f,
$$
and for a morphism $(\alpha,L)$ in $\HLH$ we take 
$\Theta=([\alpha\times L],\alpha_q)$. 
\end{enumerate}
The functor $\HLH\rightarrow\Rg$ entails a deep structural relationship between the corresponding categories, 
as it will be shown in Section~\ref{sect3}, see Theorem~\ref{compare}. 

In an analogue way to what has been done before, one could be tempted to affirm that the correspondence 
$(A,B,E;\Phi)\mapsto(\pi_A,\pi_B;P)$ defined by the Stinespring construction 
(described prior to Lemma~\ref{two_squares}) gives rise to a functor $\SLH\to\HLH$, 
so that a morphism $(\alpha,T)$ in $\SLH$ would be sent into a (possible) morphism in $\HLH$ 
to be induced by the mapping
\begin{equation}\label{tensor}
\alpha\otimes T\colon\sum_j a_j\otimes h_j\mapsto\sum_j\alpha(a_j)\otimes T(h_j),
\ \widetilde A\otimes\widetilde\Hc_0\rightarrow A\otimes\Hc_0.
\end{equation}  
Nevertheless, the extension of $\alpha\otimes T$ to $\Hc_A$ need not be well defined. 
An extra condition involving positivity on $(\alpha,T)$, 
which is sufficient for the purpose of such an extension and 
then for such a functor $\SLH\to\HLH$, is proposed in the following section.

\section{Positivity. The category of reproducing $(-*)$-kernels}\label{sect2}

There are elements making up the objects of the different categories of like-Hermitian vector bundles which enjoy a neat character of positivity. 
So in the category $\Rg$, where a transfer mapping $\Rc$ is positive (or positive-definite) 
in the sense that for all $n\ge1$ and every $s_j\in Z$, $\xi_j\in D_{s_j}$ 
($j=1,\dots, n$),
$$
\sum_{j,l=1}^n(\Rc(\xi_j)\mid\Rc(\xi_l))_{\Hc}
=(\sum_{j=1}^n\Rc(\xi_j)\mid\sum_{l=1}^n\Rc(\xi_l))
\ge0.
$$
Similarly, a $*$-representation $\pi_A$ taking part in an object of the category 
$\HLH$ satisfies for all $n\ge1$ and every $u_j\in G_A$, $h_j\in\Hc_A$ 
($j=1,\dots, n$),
$$
\sum_{j,l=1}^n(\pi_A(u_j)h_j\mid\pi_A(u_l)h_l)_{\Hc_A}
=(\sum_{j=1}^n\pi_A(u_j)h_j\mid\sum_{l=1}^n\pi_A(u_l)h_l)_{\Hc_A}
\ge0.
$$
Also, as regards the category $\SLH$, 
a (unital, in our case) completely positive mapping $\Phi\colon A\to\Bc(\Hc_0)$ 
satisfies by definition for all $n\ge1$, 
every positive matrix $(a_{j,l})_{j,l}\subset M_n(A)$ and $h_1,\dots,h_n\in\Hc_0$,
$$
\sum_{j,l=1}^n(\Phi(a_{j,l})h_j\mid h_l)_{\Hc_0}\ge0.
$$
As we are going to see below, the above self-imposed, by definition,  
properties  are particular cases of a general notion of positivity linked to like-Hermitian vector bundles.

\begin{definition}\label{kernel}
\normalfont 
Let $\Pi\colon D\to Z$ be a like-Hermitian bundle. 
A {\it reproducing $(-*)$-kernel} on $\Pi$ is a section  
$K\in\Gamma(Z\times Z,\Hom(p_2^*\Pi,p_1^*\Pi))$
(whence $K(s,t)\colon D_t\to D_s$ for all $s,t\in Z$) 
which is ($-*$)-positive definite in the following sense: 
For every $n\ge1$ and $t_j\in Z$,
$\eta_j\in D_{t_j^{-*}}$ ($j= 1,\dots, n$),
$$
\sum_{j,l=1}^n
\bigl(K(t_l,t_j^{-*})\eta_j\mid \eta_l\bigr)_{t_l,t_l^{-*}}\ge0.
$$ 
Here $p_1,p_2\colon Z\times Z\to Z$ are the natural projection mappings. 
If in addition $\Pi\colon D\to Z$ is 
a holomorphic like-Hermitian vector bundle, $K$ is continuous and 
$K(\cdot,s)\xi\in\Oc (Z,D)$ for all $\xi\in D_s$ and 
$s\in Z$, 
then we say that $K$ is a {\it holomorphic} reproducing $(-*)$-kernel. 
We shall say that a reproducing $(-*)$-kernel $K$ on a holomorphic like-Hermitian vector bundle 
$\Pi\colon D\to Z$
is {\it holomorphic of the second kind} 
if it is continuous and 
the mapping $\xi\mapsto K(t,\Pi(\xi))\xi$, $D\to D$ is holomorphic for every $t\in Z$. Examples of kernels enjoying both types of holomorphy are the kernels $K^\pi$, defined through representations $\pi$, associated to the categories $\HLH$ and $\Rep$; see Remark \ref{kapi} (2) below.
\qed
\end{definition}

\begin{remark}\label{obs}
\normalfont
Let $K(s,t)^{-*}\colon D_{s^{-*}}\to D_{t^{-*}}$ be the 
\lq\lq adjoint" 
operator of the mapping $K(s,t)$ with respect to the family $(\cdot\mid\cdot)_{r,r^{-*}}$, $r\in Z$; i.e.,
$(K(s,t)^{-*}\xi^{-*}\mid\eta)_{t^{-*},t}:=(\xi^{-*}\mid K(s,t)\eta)_{s^{-*},s}$
for every $s,t\in Z$ and $\xi^{-*}\in D_{s^{-*}}$, $\eta\in D_t$.

From the definition of reproducing kernel we have that $K(s,s^{-*})\ge0$ where we understand for this that 
$\bigl(K(s,s^{-*})\xi\mid\xi\bigr)_{s,s^{-*}}\ge0$ 
for all $\xi\in D_{s^{-*}}$. Moreover, by applying the $(-*)$-positiveness of $K$ to $s,t\in Z$, $\xi\in D_{s^{-*}}$ and $\eta\in D_{t^{-*}}$, we have 
$$
0\leq\bigl(K(s,s^{-*})\xi\mid\xi\bigr)_{s,s^{-*}}
+ \bigl(K(s,t^{-*})\eta\mid\xi\bigr)_{s,s^{-*}}
+\bigl(K(t,s^{-*})\xi\mid\eta\bigr)_{t,t^{-*}}
\bigl(K(t,t^{-*})\eta\mid\eta\bigr)_{t,t^{-*}}
$$
whence we obtain that the sum
$\bigl(K(s,t^{-*})\eta\mid\xi\bigr)_{s,s^{-*}}
+\bigl(K(t,s^{-*})\xi\mid\eta\bigr)_{t,t^{-*}}$ is a real number. So
$$
\Im \bigl(K(s,t^{-*})\eta\mid\xi\bigr)_{s,s^{-*}}=-\Im\bigl(K(t,s^{-*})\xi\mid\eta\bigr)_{t,t^{-*}}\ .
$$
By substituting $\xi$ with $i\xi$ we derive that 
$\Im i\bigl(K(s,t^{-*})\eta\mid\xi\bigr)_{s,s^{-*}}=\Im i\bigl(K(t,s^{-*})\xi\mid\eta\bigr)_{t,t^{-*}}$.
Therefore
\begin{equation}\label{exchange}
\bigl(K(s,t^{-*})\eta\mid\xi\bigr)_{s,s^{-*}}=\overline{\bigl(K(t,s^{-*})\xi\mid\eta\bigr)_{t,t^{-*}}}
=\bigl(\eta\mid K(s^{-*},t)\xi\bigr)_{s,s^{-*}}\ . 
\end{equation}
This implies that $K(a,b)^{-*}=K(b^{-*},a^{-*})$ for every $a,b\in Z$.
\qed
\end{remark}

Reproducing $(-*)$-kernels on  like-Hermitian bundles have been introduced in \cite{BG08} 
where they have been used to show that certain $*$-representations of involutive Banach-Lie groups 
(in particular Banach-Lie groups of invertibles in $C^*$-algebras) 
admit realizations on Hilbert spaces $\Hc^K$ of holomorphic sections on like-Hermitian vector bundles.

Such spaces $\Hc^K$ are constructed in the following way. 
For all $s\in Z$ and $\xi\in D_s$ denote 
$K_\xi=K(\cdot,s)\xi\in\Gamma(Z,D)$. (In fact $K_\xi$ is a holomorphic section, 
i.e., $K_\xi\in\Oc(Z,D)$, provided that the bundle $\Pi\colon D\to Z$ is holomorphic.) 
Then denote 
$\Hc^K_0:=\spann\{K_\xi\mid\xi\in D\}\subseteq\Gamma(Z,D)$.  
The Hilbert space $\Hc^K$ is defined as the completion of $\Hc^K_0$ with respect to the scalar product given by
\begin{equation}\label{prods}
(K_\eta\mid K_\xi)_{\Hc}=(K(s^{-*},t)\eta\mid\xi)_{s^{-*},s} 
\end{equation}
whenever $s,t\in Z$, $\xi\in D_{s}$, and $\eta\in D_t$. 
See additional details in \cite{BG08}.

Each one of the categories $\GLH$, $\Rg$, $\HLH$, $\SLH$ entails its own canonical type of $(-*)$-kernels. 
To begin with, let us consider the category $\GLH$. 
In this case, for a bundle $\Pi$ being under the action of a group $G$, it sounds sensible, 
for a section $K$ to be a reproducing kernel on $\Pi$, to demand some condition relating $K$ and $G$. 
The suitability of the 
following one will emerge in Section~\ref{sect5} (Theorem~\ref{U10}).

\begin{definition}\label{reprogroup}
\normalfont
Let $(\Pi,G)$ be an object in the category $\GLH$. Let $K$ be a reproducing $(-*)$-kernel on $\Pi$. 
We say that $K$ is a 
reproducing $(-*)$-kernel on $(\Pi,G)$ if $K$ satisfies
\begin{equation}\label{kgroup}  
K(t,\nu(u,s))\mu(u,\xi)=\mu(u,K(\nu(u^{-1},t),s)\xi) 
\end{equation}
for every $s,t\in Z$, $\xi\in D_s$, and $u\in G$.
\qed
\end{definition}

We next describe the $(-*)$-kernels of the remaining categories in a categorial framework, 
showing in passing that the elements of positive character can be in turn 
regarded as objects of suitable categories.

\begin{definition}\label{cat1}
\normalfont
Let $\widetilde{\Pi}\colon\widetilde{D}\to\widetilde{Z}$ 
and $\Pi\colon D\to Z$ be two like-Hermitian bundles, 
and let $\widetilde{K}$ and $K$ be reproducing $(-*)$-kernels 
on the bundles $\widetilde{\Pi}$ and $\Pi$, respectively. 
A \textit{morphism of reproducing $(-*)$-kernels} 
from $\widetilde{K}$ to $K$ 
(respectively, an \textit{antimorphism of reproducing $(-*)$-kernels} 
from $\widetilde{K}$ to $K$)
is a morphism $\Theta=(\delta,\zeta)$ of $\widetilde{\Pi}$ into $\Pi$ 
(see Definition~\ref{morph}) 
(respectively, an antimorphism) 
with the property that there exists a constant $M>0$ 
such that for every integer $n\ge1$, 
every $s_1,\dots,s_n\in \widetilde{Z}$, 
and all $\xi_1\in\widetilde{D}_{s_1},\dots, 
\xi_n\in\widetilde{D}_{Œs_n}$ 
we have 
\begin{equation}\label{hom}
\sum_{l,j=1}^n\bigl(K(\zeta(s_l)^{-*},\zeta(s_j))
\delta(\xi_j)\mid\delta(\xi_l)\bigr)_{\zeta(s_l)^{-*},\zeta(s_l)} 
\le M\sum_{l,j=1}^n 
\bigl(\widetilde{K}(s_l^{-*},s_j)\xi_j\mid\xi_l\bigr)_{s_l^{-*},s_l}. 
\end{equation}
We denote by $\Hom(\widetilde{K},K)$ 
the set of all morphisms from $\widetilde{K}$ to $K$.

It is clear that the composition of two morphisms of 
reproducing $(-*)$-kernels is again a morphism of reproducing $(-*)$-kernels. 
Consequently we get a category $\Kern$ whose 
objects are the reproducing $(-*)$-kernels on like-Hermitian bundles, and whose morphisms are defined as above.
\qed 
\end{definition}

\begin{remark}\label{cat2}
\normalfont
In the setting of Definition~\ref{cat1} 
it is clear that the mapping 
$$
\Hc_0^{\widetilde{K}}\to\Hc_0^{K},\quad 
\sum_{j=1}^n\widetilde{K}_{\tilde{\xi}_j}\mapsto 
\sum_{j=1}^n K_{\delta(\tilde{\xi}_j)}
$$
is well defined and continuous, 
and it extends to a bounded linear operator 
$\Hc^\Theta\colon\Hc^{\widetilde{K}} 
\to\Hc^{K}$ 
(respectively, an antilinear operator if $\Theta$ is an antimorphism) 
satisfying $\Vert\Hc^\Theta\Vert\le M$. 
(See Section~3 in \cite{BG08}.) 

If $\Theta_1$ and $\Theta_2$ are morphisms or antimorphisms 
of reproducing $(-*)$-kernels 
whose composition $\Theta_1\circ\Theta_2$ makes sense, 
then we have 
\begin{equation}\label{funct}
\Hc^{\Theta_1\circ\Theta_2}=
\Hc^{\Theta_1}\circ\Hc^{\Theta_2}.
\end{equation}
\qed
\end{remark}

\begin{definition}\label{cat3}
\normalfont 
Let $\Hilb$ be the category whose 
objects are the Hilbert spaces over ${\mathbb C}$ 
and whose morphisms are the bounded linear operators. 
The \textit{reproducing kernel Hilbert space functor} 
is the functor 
\begin{equation}\label{rkhs}
\Hc\colon\Kern\to\Hilb
\end{equation}
that takes every reproducing $(-*)$-kernel $K$ to 
the associated reproducing kernel Hilbert space $\Hc^K$ 
(see Definition~3.4 in \cite{BG08}) 
and every morphism $\Theta\in\Hom(\widetilde{K},K)$ 
to the bounded linear operator 
$\Hc^\Theta\colon\Hc^{\widetilde{K}} 
\to\Hc^{K}$.
The fact that \eqref{rkhs} is indeed a functor 
follows by \eqref{funct}. 
\qed
\end{definition}

\begin{remark}\label{SevCat}
\normalfont 
The functor $\Hc$ was considered in \cite{Od92} in the case of line vector bundles of finite dimension. 
Here $\Hc$ is also the main correspondence between categories that we consider in this paper, 
since it allows us to relate general kernels or fiber vector bundles with Grassmannian manifolds. 
In Proposition~\ref{posifunct} below we establish other functors between the following categories:
\begin{itemize}
\item[$\bullet$] $\Tran$ ---its objects are the transfer mappings $\Rc$ associated with the objects of the category $\Rg$, 
and where the  morphisms are defined as the linear mappings $L$, 
appearing in the morphisms $(\Theta,\alpha, L)$ of $\Rg$, for which there exists a constant $M>0$ 
such that for every integer $n\ge1$, every $s_1,\dots, s_n\in \widetilde{Z}$, 
and all $\xi_1\in\widetilde{D}_{s_1},\dots, \xi_n\in\widetilde{D}_{s_n}$ 
we have 
\begin{equation}\label{trans}
\sum_{l,j=1}^n\bigl(\Rc(L(\xi_l))\mid\Rc(L(\xi_j))\bigr)_{\Hc} 
\le M\sum_{l,j=1}^n \bigl(\widetilde{\Rc}(\xi_j)\mid\widetilde{\Rc}(\xi_l)\bigr)_{\widetilde{\Hc}}. 
\end{equation}

\item[$\bullet$] $\Rep$, whose objects are the $*$-representations associated with the objects of $\HLH$ 
and the morphisms are the pairs $(\alpha, L)$, appearing as the morphisms of the category $\HLH$, 
for which there exists a constant $M>0$ 
such that for every integer $n\ge1$, every $u_1,\dots, u_n\in \widetilde{G}_A$, 
and all $h_1,\dots, h_n\in\widetilde{\Hc}_A$ 
we have 
\begin{equation}\label{repr}
\sum_{l,j=1}^n\bigl(\pi_A(\alpha(u_l))(L(h_l))
\mid\pi_A(\alpha(u_j))(L(h_j))\bigr)_{\Hc_A} 
\le M\sum_{l,j=1}^n \bigl(\tilde{\pi}_A(u_l)h_l
\mid
\tilde{\pi}_A(u_j)h_j\bigr)_{\widetilde{\Hc}_A}. 
\end{equation}  

\item[$\bullet$] $\Cp$, whose objects are the completely posive maps associated with the objects of  the category $\SLH$ 
and the morphisms are the pairs $(\alpha, T)$ appearing in the morphisms of the category $\SLH$,
for which there exists a constant $M>0$ 
such that for every integer $n\ge1$, every positive 
$(a_{lj})\subset M_n(\widetilde A)$, 
and all $h_1,\dots, h_n\in\widetilde{\Hc}_0$  
we have 
\begin{equation}\label{comPos}
\sum_{l,j=1}^n
\bigl(\Phi(\alpha(a_{l,j}))T(h_l)
\mid T(h_j)\bigr)_{\Hc_0} 
\le M\sum_{l,j=1}^n \bigl(\widetilde{\Phi}(a_{lj})h_l
\mid h_j\bigr)_{\widetilde{\Hc}_0}. 
\end{equation}  
\end{itemize}
\qed
\end{remark}

Completely positive mappings have been considered as (relevant) morphisms of suitable categories 
of matrix ordered spaces or operator systems; 
see for instance \cite{CE77}, \cite {ER00}, \cite{Pa02}, or \cite{Zu97}. 
It is maybe worth noticing that such mappings are viewed in the context developed here 
as {\it objects} of a category, rather than morphisms.   


\begin{proposition}\label{posifunct} 
The following correspondences define functors between the respective categories.
\begin{itemize}
\item[(i)] The dilation functor $\Cp\rightarrow\Rep$ given by
$$
\Phi\mapsto\pi_A\quad \hbox{ and } \quad (\alpha,T)\mapsto(\alpha,\alpha\otimes T)\ ,
$$
where $\pi_A$ is the Stinespring dilation associated with 
the completely positive mapping $\Phi$ and $\alpha\otimes T$ is 
the continuous extension to $\widetilde\Hc_A$ of the map defined in \eqref{tensor}.
\item[(ii)] $\Rep\rightarrow\Tran$ given by
$$
\pi\mapsto\Rc^{\pi}\quad \hbox{ and } \quad (\alpha,L)\mapsto(\Theta;\alpha,L),
$$
where $\Rc^{\pi}([(u,f)])=\pi(u)f$, for every $u\in G_A$, $f\in\Hc_B$.
\item[(iii)] $\Tran\rightarrow\Kern$ given by
$$
\Rc\mapsto K^{\Rc}\quad \hbox{ and } \quad (\Theta;\alpha,L)\mapsto\Theta,
$$
where $K^{\Rc}(s,t)=(\Rc_{s^{-*}})^{-*}\circ\Rc_t$, for every $s,t\in Z$. 
(Here, the expression 
$(\Rc_{s^{-*}})^{-*}$ has the meaning explained in Remark~\ref{sesqui1}.)
\end{itemize}
\end{proposition}

\begin{proof} We shall show that the pair $(\alpha,\alpha\otimes T)$ is 
well defined as a morphism in the category $\Rep$. 
The functorial property follows then readily.

So let $\widetilde\Phi\colon\widetilde A\to\Bc(\widetilde\Hc_0)$,
$\Phi\colon A\to\Bc(\Hc_0)$ objects in the category $\Cp$, with associated triples 
$(\widetilde A,\widetilde B, \widetilde E)$, $(A,B,E)$ respectively. 
Take a morphism 
$(\alpha,T)\colon\widetilde\Phi\to\Phi$. 
This means that $\alpha\colon\widetilde A\to A$ is a unital algebra $*$-homomorphism 
with $\alpha(\widetilde B)\subseteq B$ and that 
$T\colon\widetilde\Hc_0\to\Hc_0$ is a bounded linear operator.

Take $x=\sum_{i=1}^na_i\otimes y_i$ in $N_{\widetilde A}
\subseteq\widetilde A\otimes\widetilde\Hc_0$. Then 
$\sum_{i,j=1}^n(\widetilde\Phi(a_i^{*}a_j)y_j\mid y_i)_{\widetilde\Hc_0}=0$ and therefore
$\sum_{i,j=1}^n(\Phi(\alpha(a_i)^{*}\alpha(a_j) Ty_j\mid Ty_i)_{\Hc_0}
\le M\sum_{i,j=1}^n(\widetilde\Phi(a_i^{*}a_j) y_j\mid y_i)_{\widetilde\Hc_0}=0$ from which $(\alpha\otimes T)(x)\in N_A$. 
Thus the quotient mapping 
$(\alpha\otimes T)_q\colon\widetilde A\otimes\widetilde\Hc_0/N_{\widetilde A}
\rightarrow A\otimes\Hc_0/N_{A}$ is well defined. 
Moreover, for 
$x\in\widetilde A\otimes\widetilde\Hc_0$ as above, we have 
$$
\begin{aligned}
\Vert\alpha\otimes T(x)\Vert_{\Hc_A}^2
&=\Vert \sum_{i} \alpha(a_i)\otimes Ty_i\Vert_{\Hc_A}^2
=\sum_{i,j}(\Phi(\alpha(a_i^{*}a_j))Ty_j\mid Ty_i)_{\Hc_0} \\
&\le M\sum_{i,j}(\widetilde\Phi(a_i^{*}a_j)y_j\mid y_i)_{\widetilde\Hc_0}
=M\Vert x\Vert_{\widetilde\Hc_A}^2
\end{aligned}
$$
so that $(\alpha\otimes T)_q$ extends by continuity from $\widetilde\Hc_A$ to $\Hc_A$. 
We keep the notation $\alpha\otimes T$ for such an extension.

Note that, in an analogous way, $\alpha\otimes T$ also extends continuously from 
$\widetilde\Hc_B$ to $\Hc_B$.
Let us now see that $(\alpha,\alpha\otimes T)$ is a morphism from 
$(\pi_{\widetilde A},\pi_{\widetilde B},\widetilde P)$ into $(\pi_{A},\pi_{B},P)$ in $\HLH$.
Since $\alpha\circ\widetilde E\mid_{\widetilde B}=E\circ\alpha_{\widetilde B}$ it follows that 
$\alpha(\G_{\widetilde B})\subseteq\G_{B}$ and, also, 
$$
\begin{aligned}
P\circ(\alpha\otimes T)(\sum_i b_i\otimes y_i+N_{\widetilde B})
&=P(\sum_i\alpha(b_i)\otimes Ty_i+N_{\widetilde B}) 
=\sum_i E(\alpha(b_i))\otimes Ty_i+N_{\widetilde B} \\
&=\sum_i\alpha(\widetilde E(b_i))\otimes Ty_i+N_{\widetilde B}
=(\alpha\otimes T)\widetilde P(\sum_i b_i\otimes y_i+N_{\widetilde B}).
\end{aligned}
$$
for each $\sum_i b_i\otimes y_i+N_{\widetilde B}\in\widetilde B\otimes\widetilde\Hc_0$. By density it follows that 
$P\circ(\alpha\otimes T)=(\alpha\otimes T)\widetilde P$ on $\widetilde\Hc_B$.

Now for every $u\in\G_{\widetilde A}$ and $\sum_i a_i\otimes y_i+N_{\widetilde A}
\in\widetilde A\otimes\widetilde\Hc_0$,
$$
\begin{aligned}
(\alpha\otimes T)\pi_{\widetilde A}(u)&(\sum_i a_i\otimes y_i+N_{\widetilde A})
=\sum_i\alpha(ua_i)\otimes Ty_i+N_{\widetilde A} \\
&=\pi_{A}(\alpha(u))(\sum_i\alpha(a_i)\otimes Ty_i+N_{\widetilde A})=\pi_{A}(\alpha(u))(\alpha\otimes T)(\sum_i a_i\otimes y_i+N_{\widetilde A}).
\end{aligned}
$$
Again by density we obtain that $(\alpha\otimes T)\pi_{\widetilde A}(u)=\pi_{A}(\alpha(u))(\alpha\otimes T)$ for all 
$u\in\G_{\widetilde A}$. Note that in particular, from the above equality we deduce that the mapping 
$$
(\alpha,\alpha\otimes T)\colon[(u,f)]\mapsto[(\alpha(u),(\alpha\otimes T)f)], 
\ \G_{\widetilde A}\times_{\G_{\widetilde B}}\Hc_{\widetilde B}
\rightarrow\G_{A}\times_{\G_{B}}\Hc_{ B}
$$
is well defined. 

It is clear that 
$\alpha_q\circ(\1_{\G_{\widetilde A}}\times\widetilde\Pi)
= (\1_{\G_A}\times\Pi)\circ(\alpha,\alpha\otimes T)$, where $\alpha_q$ is the quotient mapping
defined by  $\alpha_q(u \G_{\widetilde B})=\alpha(u)\G_{B}$  for $u\in\G_{\widetilde A}$. 
Moreover, 
$\alpha_q(u \G_{\widetilde B})^{-*}=\alpha_q((u \G_{\widetilde B})^{-*})$ for all 
$u\in\G_{\widetilde A}$.

Since all the above mappings are clearly holomorphic, we have shown that 
$((\alpha,\alpha\otimes T),\alpha_q)$ is a morphism in the category $\HLH$, as required in a first step.

Finally, for $n\ge1$, $u_i\in\G_{\widetilde A}$, 
$f_i=\sum_{k=1}^m a_k^{i}\otimes y_k^{i}$ with $a_k^{i}\in\widetilde A$ and 
$y_k^{i}\in\widetilde\Hc_0$ ($i=1,\dots,n$), we get
$$
\begin{aligned}
\sum_{i,j=1}^n (\pi_A(\alpha(u_i)) & (\alpha\otimes T)(f_i)
\mid\pi_A(\alpha(u_j))(\alpha\otimes T)(f_j))_{\Hc_A}  \\
&=\sum_{i,j=1}^n\sum_{k=1}^m(\pi_A(\alpha(u_i))(\alpha(a_k^{i})\otimes Ty_k^{i})
\mid\pi_A(\alpha(u_j))(\alpha(a_k^{j})\otimes Ty_k^{j}))_{\Hc_A} \\
&=\sum_{k=1}^m\sum_{i,j=1}^n(\alpha(u_ia_k^{i})\otimes Ty_k^{i}
\mid\alpha(u_ja_k^{j})\otimes Ty_k^{j})_{\Hc_A} \\
&=\sum_{k=1}^m\sum_{i,j=1}^n
(\Phi(\alpha((u_ja_k^{j})^{*}u_ia_k^{i}))Ty_k^{i}
\mid Ty_k^{j} )_{\Hc_0} \\
&\le\sum_{k=1}^m M\sum_{i,j=1}^n (\widetilde\Phi((u_ja_k^{j})^{*}u_ia_k^{i})Ty_k^{i}
\mid y_k^{j} )_{\widetilde\Hc_0}
=M\sum_{i,j=1}^n(\pi_{\widetilde A}(u_i)f_i
\mid\pi_{\widetilde A}(u_j)f_j)_{\Hc_{\widetilde A}},
\end{aligned}
$$
where in the inequality we have used that $(\alpha,T)$ is a morphism between 
$\widetilde\Phi$ and $\Phi$.

For arbitrary $(f_i)\subseteq\Hc_{\widetilde A}$ it is enough to approximate each $f_i$ by elements in 
$\widetilde A\otimes\widetilde\Hc_0/N_{\widetilde A}$ and apply the continuity of $\alpha\otimes T$ 
to the just proved inequality, to obtain that 
$(\alpha,\alpha\otimes T)$ is a morphism in the category~$\Rep$. 

(ii) This part is quite obvious.

(iii) The statement is a simple consequence of the fact that for every $s,t\in Z$, 
$\xi\in D_{s^{-*}}$ and $\eta\in D_t$ one has
$(K^{\Rc}(s,t)\eta\mid\xi)_{s,s^{-*}}=((\Rc_{s^{-*}})^{-*}(\Rc_t\eta)\mid\xi)_{s,s^{-*}}
=(\Rc_t\eta\mid\Rc_{s^{-*}}\xi)_{\Hc}$. 
\end{proof}

\begin{remark}\label{kapi}
\normalfont
(1) The kernel $K^{\Rc}$ defined in part (iii) of the above proposition 
is a reproducing kernel on $(\Pi,G)$ in $\Rg$. 
In fact, take $s,t\in Z$, $\xi\in D_s$ and $u\in G$. 
For simplicity, we put just in this remark $u\cdot s:=\nu(u,s)$ and $u\cdot\xi:=\mu(u,\xi)$. 
Then, for every $\eta\in D_{t^{-*}}$,
$$  
\begin{aligned}
(K^{\Rc}(t,u\cdot s)(u\cdot\xi)\mid\eta)_{t,t^{-*}}
&=(\Rc_{u\cdot s}(u\cdot\xi)\mid\Rc_{t^{-*}}\eta)_\Hc 
=((\pi(u)\Rc_s)(\xi)\mid\Rc_{t^{-*}}\eta)_\Hc \\
&=(\Rc_s\xi\mid\pi(u)^*\Rc_{t^{-*}}\eta)_\Hc 
=(\Rc_s\xi\mid\pi(u^*)\Rc_{t^{-*}}\eta)_\Hc \\ 
&=(\Rc_s\xi\mid\Rc_{u^*\cdot t^{-*}}(u^*\cdot\eta))_\Hc 
=((\Rc_{u^*\cdot t^{-*}})^{-*}(\Rc_s\xi)\mid u^*\cdot\eta) 
_{(u^{-1}\cdot t), (a^*\cdot t^{-*})} \\
&=(u\cdot[ (\Rc_{u^*\cdot t^{-*}})^{-*}(\Rc_s\xi)]\mid\eta)_{t,t^{-*}}
=(u\cdot[ K^{\Rc}(u^{-1}\cdot t,s)\xi]\mid\eta)_{t,t^{-*}}
\end{aligned}
$$
where we have used (iii) of Definition~\ref{transrel} twice, 
and (ii) and (iii) of Definition~\ref{relationship}. 
Since the sesquilinear form
$(\cdot\mid\cdot)_{t,t^{-*}}$ is a strong duality pairing, by letting $\eta$ running over 
$D_{t^{-*}}$ one obtains 
$K^{\Rc}(t,u\cdot s)(u\cdot\xi)=u\cdot K^{\Rc}(u^{-1}\cdot t,s)\xi$ as we wanted to show.

It is worth noticing that, for {\it any} $(-*)$-kernel $K$ associated with a like-Hermitian bundle 
$\Pi\colon D\to Z$, there always exists a Hilbert space $\Hc$ and a mapping 
$\Rc_K\colon D\to\Hc$ enjoying part of the properties of a transfer mapping, such that 
$K=K^{\Rc_K}$. In fact, one can choose $\Hc=\Hc^K$ as 
in the construction pointed out prior to Definition \ref{reprogroup}. 
Then by Proposition~3.7 c) of \cite{BG08} there are 
the \emph{evaluation} maps $\ev_s^\iota\colon\Hc^K\to D$ ($s\in Z$), so that 
$K(s,t)=\ev_s^\iota\circ(\ev_{t^{-*}}^\iota)^{-*}$ for all $s,t\in Z$. 
Thus one can define 
$\Rc_K(\xi):=(\ev_{s^{-*}}^\iota)^{-*}(\xi)$ for every $s\in Z$ and $\xi\in D_s$. 
It is routine to check that 
$K=K^{\Rc_K}$. 
However, one cannot assert that $\Rc_K$ is an isometry since 
$(\Rc_K(\xi)\mid\Rc_K(\eta))_{\Hc^K}=(K(s,s)\xi\mid\eta)_{s,s^{-*}}$ for $s\in Z$ and $\xi\in D_s$, 
$\eta\in D_{s^{-*}}$.

(2) A kernel $K^{\pi}$ can be associated to a given representation $\pi\in\Rep$ by the composition of functors $\Rep\to\Tran\to\Kern$. 
Such a kernel is therefore given as 
$K^{\pi}(s,t)=(\Rc_{s^{-*}})^{-*}\circ\Rc_t$ for $s,t\in D=G_A\times_{G_B}\Hc_B$, where 
$\Rc([(w,h)]):=\pi_A(w)h\in\Hc_A$ for every $[(w,h)]\in D$. 
Indeed, kernels of type $K^{\pi}$ admit alternatively a description as
\begin{equation}\label{jones}
K^{\pi}(s,t)\eta=[(u,P(\pi_A(u^{-1})\pi_A(v)f))]\in D_s 
\end{equation}
provided that $u,v\in G_A$, $s=uG_B$, $t=vG_B$, $\eta=[(v,f)]\in D_t$, 
and $P\colon\Hc_A\to\Hc_B$ is the orthogonal projection.
In fact, for each $\xi=[(u^{-*},g)]\in D_{s^{-*}}$, we have
$$
\begin{aligned}
(K^{\pi}(s,t)\eta\mid\xi)_{s,s^{-*}}
&=((\Rc_{s^{-*}})^{-*}(\Rc_t\eta)\mid\xi)_{s,s^{-*}} 
=(\Rc_t\eta\mid\Rc_{s^{-*}}\xi)_{\Hc_A}
=(\pi_A(v)f\mid\pi_A(u^{-*})g)_{\Hc_A} \\
&=(\pi_A(u^{-1})\pi_A(v)f\mid g)_{\Hc_B}=(P(\pi_A(u^{-1})\pi_A(v)f)\mid g)_{\Hc_B} \\
&=([(u,P(\pi_A(u^{-1})\pi_A(v)f))]\mid[(u^{-*},g)])_{s,s^{-*}}
\end{aligned}
$$
where the last equality follows by \eqref{strong}. 
So \eqref{jones} follows.

In turn, the composition of functors $\Cp\to\Rep\to\Kern$ gives us what we call here the {\it Stinespring kernel functor} $\Cp\to\Kern$ defined by 
$$
\Phi\mapsto K^\Phi, \quad (\alpha,T)\mapsto(\Theta,\alpha,L)
$$
where $K^\Phi$ is given as in \eqref{jones}, and where $\pi_A$ the Stinespring dilation 
of $\Phi$, $\Theta$ is the vector bundle morphism $\Theta=(\delta,\zeta)$ with 
$\delta=[\alpha\times L]$, $\zeta=\alpha_q$, and $L=\alpha\otimes T$ as defined formerly.

The expression of $K^\pi$ given in \eqref{jones} is the form under which the canonical kernels associated to homogeneous like-Hermitian bundles have been defined in \cite{BG08}, 
see also \cite{BR06}. 
Such kernels are used in  \cite{BG08} to construct Hilbert spaces of holomorphic sections on those bundles.
\qed
\end{remark}

\section{Operations on reproducing $(-*)$-kernels. Pull-backs}\label{sect3}

In this section we discuss the operation of pull-back on reproducing
kernels
in the general setting of like-Hermitian bundles.
A few special instances of this operation had been previously
in the literature in the case of trivial bundles;
see e.g., \cite{Ne00} or subsection~3.2 in \cite{CVTU08}.
The pull-backs will play a crucial role in connection with the
universality theorems
that we are going to establlish in Section~\ref{sect5}.

\begin{definition}\label{cone}
\normalfont
Let $\Pi\colon D\to Z$ be a like-Hermitian vector bundle, 
and denote by $p_1,p_2\colon Z\times Z\to Z$ the projection mappings. 
It is clear from Definition~3.1 in \cite{BG08} that the set of 
all reproducing $(-*)$-kernels on $\Pi$ 
is closed under addition and under multiplication by positive scalars. 
Thus, that set is a convex cone in the complex vector space 
$\Gamma(Z\times Z,\Hom(p_2^*\Pi,p_1^*\Pi))$; 
we shall denote this cone by 
$\RK(\Pi)$ 
and we shall call it 
the {\it cone of reproducing $(-*)$-kernels on $\Pi$}.
\qed
\end{definition}

\begin{definition}\label{pull}
\normalfont
Let $\widetilde{\Pi}\colon \widetilde{D}\to\widetilde{Z}$ and 
$\Pi\colon D\to Z$ be 
like-Hermitian vector bundles, and assume that 
each of the manifolds $\widetilde{Z}$ and $Z$ is endowed with an involutive 
diffeomorphism denoted by $z\mapsto z^{-*}$ for both manifolds. 
Denote by $p_1,p_2\colon Z\times Z\to Z$ and 
$\widetilde{p}_1,\widetilde{p}_2\colon\widetilde{Z}\times\widetilde{Z}
\to\widetilde{Z}$ 
the natural projections. 
Assume that $\Theta=(\delta,\zeta)$ is a morphism (or antimorphism) of 
$\widetilde{\Pi}$ into $\Pi$ 
(see Definition~\ref{morph}).

Now let $K\in\Gamma(Z\times Z,\Hom(p_2^*\Pi,p_1^*\Pi))$ 
be a reproducing kernel on $\Pi$, 
where $p_1,p_2\colon Z\times Z\to Z$ are the natural projections. 
Then the {\it pull-back} of $K$ by $\Theta$ is 
the reproducing kernel 
$$
\Theta^*K:=\widetilde{K}\in
\Gamma(\widetilde{Z}\times\widetilde{Z},
\Hom(\widetilde{p}_2^*\widetilde{\Pi},\widetilde{p}_1^*\widetilde{\Pi}))
$$
defined by 
\begin{equation}\label{pullK}
\widetilde{K}(s,t)=(\delta_{s^{-*}})^{-*}\circ K(\zeta(s),\zeta(t))\circ \delta_t
\end{equation}
for all $s,t\in Z$. 
Here $(\delta_{s^{-*}})^{-*}\colon D_{\zeta(s)}\to\widetilde D_s$ 
is the operator defined by \eqref{quasi-adjointable} 
(or \eqref{quasi-adjointable_anti}).
\qed
\end{definition}

It is easy to see that 
$\widetilde{K}$ is indeed a reproducing kernel on~$\widetilde{\Pi}$. 
Note that formula \eqref{pullK} means that the following diagram  
$$
\begin{CD}
D_{\zeta(t)} @>{K(\zeta(s),\zeta(t))}>> D_{\zeta(s)} \\
@A{\delta_t}AA @VV{(\delta_{s^{-*}})^{-*}}V \\
\widetilde{D}_t @>{\widetilde{K}(s,t)}>>  \widetilde{D}_s
\end{CD}
$$
is commutative for all $s,t\in Z$.

Let us further observe that formula \eqref{pullK} gives us the equality in \eqref{hom} of Definition~\ref{cat1}, for $M=1$. In this sense we can interpret that the pull-back of a given reproducing kernel $K$, with respect to a vector bundle $\Theta$, is another reproducing kernel $\Theta^*K$ for which $\Theta$ is {\it extremal}, say, for  \eqref{hom}, or, alternatively, $M$-extremal for \eqref{hom}, with $M=1$.     

\begin{example}\label{triv}
\normalfont
Let us consider the special case when both bundles $\widetilde{\Pi}$ and $\Pi$ 
are trivial and $\zeta$ is the identity map 
of some manifold $Z$. 
Thus, let $\widetilde{\Hc}$ and $\Hc$ be complex Hilbert spaces, and let 
$\kappa\colon Z\times Z\to\Bc(\Hc)$ be a reproducing kernel 
in the classical sense 
(see e.g., \cite{Ne00}), 
which can be identified with a reproducing kernel $K$ on 
the trivial bundle $Z\times\Hc\to Z$ by $K(s,t)=(s,t,\kappa(s,t))$ 
whenever $s,t\in Z$. 

Next let $F\colon Z \to\Bc(\widetilde{\Hc},\Hc)$ be any mapping, 
which can be identified with a homomorphism $\delta$ from the trivial bundle 
$Z\times \widetilde{\Hc}\to Z$ into the trivial bundle $Z\times\Hc\to Z$ 
defined by $\delta(z,\widetilde{f})=(z,F(z)\widetilde{f})$ 
for all $z\in Z$ 
and $\widetilde{f}\in\widetilde{\Hc}$ 
(see for instance \cite{La01}). 
Now let $\Theta=(\delta,\id_Z)$ (compare Definition~\ref{pull}). 

Then the pull-back of $K$ by $\Theta$ can be identified as above with 
a reproducing kernel (in the classical sense) defined by 
$Z\times Z\to\Bc(\widetilde{\Hc})$, 
$(s,t)\mapsto F(s)^*\circ\kappa(s,t)\circ F(t)$.

Now recall the setting of subsection 1.4. Let $A$ be a unital $C^*$-algebra and let $\Hc_0$, $\Hc_A$ be two complex Hilbert spaces. For a unital completely positive mapping $\Phi\colon A\to\Bc(\Hc_0)$, let 
$\pi_A\colon A\to\Bc(\Hc_A)$ be a Stinespring dilation of $\Phi$. Thus
\begin{equation}\label{dilcom}
\Phi(a)=V^*\circ\pi_A(a)\circ V,\ \qquad (a\in A)
\end{equation}
for some isometry $V\colon\Hc_0\to\Hc_A$. In the above example about trivial bundles, make the choice 
$Z=A$, $F(s)=V$ for all $s\in A$ so that $\delta=\id_A\times V$, and $\kappa(s,t)=\pi(s)$ for every 
$(s,t)\in A\times A$. Then the pull-back reproducing kernel of $\kappa$ by $\Theta:=(\delta,\id_A)$ is the mapping $\Phi\colon(s,t)\mapsto\Phi(s),\ A\times A\to\Bc(\Hc_0)$.

In other words, the fundamental relation \eqref{dilcom} 
can be interpreted as a particular example of the pull-back operation for reproducing kernels. 
 \qed
\end{example}

\begin{remark}\label{hol}
\normalfont
If both $\widetilde{\Pi}$ and $\Pi$ 
are holomorphic bundles, the pair $(\delta,\zeta)$ is a 
homomorphism of 
$\widetilde{\Pi}$ into $\Pi$, 
the mapping $\zeta$ is holomorphic, 
and the mapping $\delta$ is {\it anti}-holomorphic, 
then the reproducing kernel $\Theta^*K$ 
is holomorphic provided $K$ is holomorphic. 
\qed
\end{remark}

\begin{remark}\label{isometry}
\normalfont
In the setting of Definition~\ref{pull}, 
assume that $\Theta=(\delta,\zeta)$
is an isometry (see Definition~\ref{isometric}) and in addition 
$\delta_z\colon\widetilde{D}_z\to D_{\zeta(z)}$ 
is a bijective map 
for all $z\in\widetilde{Z}$. 
In this case it follows by Remark~\ref{morph2} that 
for any reproducing kernel $K$ on $\Pi$ we have 
$$
(\forall s,t\in\widetilde{Z})\quad 
\Theta^*K(s,t)
=(\delta_s)^{-1}\circ K(\zeta(s),\zeta(t))\circ \delta_t.
$$
An important instance when the above situation occurs 
is when the bundle $\widetilde{\Pi}$ is the pull-back of $\Pi$ by $\zeta$,   
and $\delta$ is the associated map. 
(See \cite{La01} for details on pull-backs of vector bundles.)  
\qed
\end{remark}

The next result provides us in particular with a characterization of 
pull-backs of reproducing $-*$-kernels in terms of the corresponding reproducing kernel Hilbert spaces.

\begin{proposition}\label{cat5}
Let $\widetilde{\Pi}\colon\widetilde{D}\to\widetilde{Z}$ 
and $\Pi\colon D\to Z$ be two like-Hermitian bundles, 
and let $\widetilde{K}$ and $K$ be reproducing $(-*)$-kernels 
on $\widetilde{\Pi}$ and $\Pi$, respectively. 
In addition, let $\Theta=(\delta,\zeta)$ 
be a morphism (respectively, antimorphism) 
of $\widetilde{\Pi}$ into $\Pi$. 
Then the following assertions hold: 
\begin{itemize}
\item[{\rm(i)}] We have $\widetilde{K}=\Theta^*K$ 
if and only if $\Theta\in\Hom(\widetilde{K},K)$ 
and $\Hc^\Theta\colon\Hc^{\widetilde{K}}
\to\Hc^{K}$ is an isometry. 
\item[{\rm(ii)}] Assume that 
$\Theta$ is an isometry of like-Hermitian structures 
and $\delta(\widetilde{D}_{\tilde{t}})$ is dense in 
$D_{\zeta(\tilde{t})}$ for all $\tilde{t}\in\widetilde{Z}$. 
If $\widetilde{K}=\Theta^*K$ and 
both $\widetilde{K}$ and $K$ are reproducing $(-*)$-kernels, 
then the diagram 
$$
\begin{CD}
\widetilde{D}_{\tilde{t}} @>{\delta|_{\widetilde{D}_{\tilde{t}}}}>> 
  D_{\zeta(\tilde{t})} \\
@A{\ev_{\tilde{t}}}AA @AA{\ev_{\zeta(\tilde{t})}}A \\
\Hc^{\widetilde{K}}  @>{\Hc^\Theta}>> 
  \Hc^{K}
\end{CD}
 $$
is commutative for all $\tilde{t}\in\widetilde{Z}$.
\end{itemize}
\end{proposition}

\begin{proof}
(i) We have $\widetilde{K}=\Theta^*K$  
if and only if 
for all $s,t\in\widetilde{Z}$, 
\begin{equation}\label{pre_star}
(\forall s,t\in\widetilde{Z})\quad 
\widetilde{K}(s,t)
=(\delta_{s^{-*}})^{-*}\circ K(\zeta(s),\zeta(t))
\circ\delta_{t}.
\end{equation}
Now assume that $\Theta$ is a morphism. 
Since both $\widetilde{\Pi}$ and $\Pi$ are like-Hermitian bundles, 
condition~\eqref{pre_star} is further equivalent to 
\begin{equation}\label{star}
(\forall s,t\in\widetilde{Z})
(\forall\xi^{-*}\in\widetilde{D}_{s^{-*}},
\eta\in\widetilde{D}_{t})\quad 
\bigl(\widetilde{K}(s,t)\eta 
   \mid\xi^{-*}\bigr)_{s,s^{-*}}
=\bigl(K(\zeta(s),\zeta(t))\delta(\eta)
  \mid\delta(\xi^{-*})\bigr)_{\zeta(s),\zeta(s)^{-*}},
\end{equation}
and this is equivalent to the fact that 
$\Theta\in\Hom(\widetilde{K},K)$ 
and $\Hc^\Theta\colon\Hc^{\widetilde{K}}
\to\Hc^{K}$ is an isometry. 
In fact, as noticed after Definition \ref{pull}, in the setting of Definition~\ref{cat1} the latter condition 
means that in \eqref{hom} we should always have equality and $M=1$.

On the other hand if $\Theta$ is an antimorphism, 
then condition~\eqref{pre_star} is equivalent to 
\begin{equation}\label{star_anti}
(\forall\tilde{s},\tilde{t}\in\widetilde{Z})
(\forall\tilde{\xi}^{-*}\in\widetilde{D}_{\tilde{s}^{-*}},
\tilde{\eta}\in\widetilde{D}_{\tilde{t}})\quad 
\bigl(\widetilde{K}(\tilde{s},\tilde{t})\tilde{\eta} 
   \mid\tilde{\xi}^{-*}\bigr)_{\tilde{s},\tilde{s}^{-*}}
=\overline{\bigl(K(\zeta(\tilde{s}),\zeta(\tilde{t}))\delta(\tilde{\eta})
  \mid\delta(\tilde{\xi}^{-*})\bigr)}_{\zeta(\tilde{s}),\zeta(\tilde{s})^{-*}},
\end{equation}
which is in turn equivalent to the fact that 
$\Theta\in\Hom(\widetilde{K},K)$ 
and $\Hc^\Theta\colon\Hc^{\widetilde{K}}
\to\Hc^{K}$ is an isometry. 
(Note that both sides of \eqref{hom} are real numbers 
hence are equal to their own complex-conjugates.) 

(ii) Since $\spann\{\widetilde{K}_{\tilde{\xi}}\mid 
\tilde{\xi}\in\widetilde{D}\}$ is dense in $\Hc^{\widetilde{K}}$, 
it will be enough to check that 
\begin{equation}\label{sstar}
\bigl(
\Hc^\Theta(\widetilde{K}_{\tilde{\xi}})\bigr)(\zeta(\tilde{t}))
=\delta(\widetilde{K}_{\tilde{\xi}}(\tilde{t})).
\end{equation}
Now let us assume that $\Theta$ is a morphism. 
We have 
$\bigl(
\Hc^\Theta(\widetilde{K}_{\tilde{\xi}})\bigr)(\zeta(\tilde{t}))
=K_{\delta(\tilde{\xi})}(\zeta(\tilde{t}))
=K(\zeta(\tilde{t}),\zeta(\tilde{s}))\delta(\tilde{\xi})$, 
hence for arbitrary 
$\tilde{\eta}^{-*}\in\widetilde{D}_{\tilde{t}^{-*}}$ 
we get 
$$\begin{aligned}
\bigl(\bigl(
\Hc^\Theta(\widetilde{K}_{\tilde{\xi}})\bigr)(\zeta(\tilde{t})) 
\mid\delta(\tilde{\eta}^{-*})\bigr)_{\zeta(\tilde{t}),\zeta(\tilde{t})^{-*}} 
 &=
 \bigl(
K(\zeta(\tilde{t}),\zeta(\tilde{s}))\delta(\tilde{\xi})
\mid\delta(\tilde{\eta}^{-*})\bigr)_{\zeta(\tilde{t}),\zeta(\tilde{t})^{-*}} \\
 &\stackrel{\scriptstyle\eqref{star}}{=} 
 \bigl(
\widetilde{K}(\tilde{t},\tilde{s})\tilde{\xi}
\mid\tilde{\eta}^{-*}\bigr)_{\tilde{t},\tilde{t}^{-*}} \\
 &= 
 \bigl(
\widetilde{K}_{\tilde{\xi}}(\tilde{t})
\mid\tilde{\eta}^{-*}\bigr)_{\tilde{t},\tilde{t}^{-*}} \\
 &= 
 \bigl(
\delta(\widetilde{K}_{\tilde{\xi}}(\tilde{t}))
\mid\delta(\tilde{\eta}^{-*})\bigr)_{\zeta(\tilde{t}),\zeta(\tilde{t})^{-*}} 
\end{aligned}
$$
where the latter equality follows by the hypothesis that 
$\Theta$ is an isometry of like-Hermitian structures 
(see Definition~\ref{isometric}).  
Now we get equation~\eqref{sstar} since 
$\{\delta(\tilde{\eta}^{-*})\mid
\tilde{\eta}^{-*}\in\widetilde{D}_{\tilde{t}^{-*}}\}$ is dense in 
$D_{\zeta(\tilde{t})^{-*}}$ by hypothesis,  
while $\Pi$ is a like-Hermitian bundle. 

A similar reasoning works in the case when $\Theta$ is an antimorphism. 
\end{proof}

The following statement is a version of Proposition~III.3.3 in \cite{Ne00} 
in our framework. 
(See also the proof of Theorem~3 in \cite{FT99}.) It supplies conditions on a vector bundle $\Pi$ for an associated kernel $K$ to be the pullback of itself.

\begin{corollary}\label{cat6}
Assume that $Z$ is a Banach manifold with an involutive 
diffeomorphism $Z\to Z$, $z\mapsto z^{-*}$, 
$\Pi\colon D\to Z$ is a like-Hermitian vector bundle, 
and $K$ is a reproducing $(-*)$-kernel on the bundle $\Pi$. 

Now let $\tau\colon D\to D$ be a smooth map such that $\tau^2=\id_D$, and 
for all $z\in Z$ we have $\tau(D_z)\subseteq D_{z^{-*}}$, 
$\tau|_{D_z}\colon D_z\to D_{z^{-*}}$ is antilinear,   
and 
\begin{equation}\label{anti}
(\forall s\in Z)(\forall\xi\in D_s,\xi^{-*}\in D_{s^{-*}})\qquad 
(\tau(\xi)\mid\tau(\xi^{-*}))_{s^{-*},s}
=(\xi^{-*}\mid\xi)_{s^{-*},s}.
\end{equation}
Then we have 
\begin{equation}\label{symm}
(\forall s,t\in Z)\quad K(s,t)
=\tau^{-1}\circ K(s^{-*},t^{-*})\circ\tau|_{D_t},
\end{equation}
if and only if there exists an involutive antilinear isometry 
$\overline{\tau}\colon\Hc^K\to\Hc^K$ 
such that for all $\xi\in D$ we have
$\overline{\tau}(K_\xi)=K_{\tau(\xi)}$. 
If this is the case, then  
$$
(\forall F\in\Hc^K)(\forall t\in Z)\quad 
\bigl(\overline{\tau}(F)\bigr)(t)
=\tau(F(t^{-*})).
$$ 
\end{corollary}

\begin{proof}
First note that \eqref{symm}  
is equivalent to the fact that $\Theta^*K=K$, where 
$\Theta$ is the antimorphism of the bundle $\Pi$ into itself 
defined by the pair of mappings $\tau$ and $z\mapsto z^{-*}$. 
Now the first assertion follows by Proposition~\ref{cat5}(i) 
and Remark~\ref{cat2} with $\overline{\tau}:=\Hc^\Theta$.  

To prove the second assertion we can use Proposition~\ref{cat5}(ii). 
\end{proof}

The following theorem shows how objects of the category $\HLH$ 
can be obtained as bundle pull-backs of objects of the category $\Rg$. 
Further,  this theorem points out the structural role of 
the transfer mappings $\Rc$ within these categories.

\begin{theorem}\label{compare}
Assume the setting of {\rm Definition~\ref{relationship}}. 
Let $z_0\in Z$ such that $z_0^{-*}=z_0$. 
Assume that the isotropy group 
$G_0:=\{u\in G\mid\nu(u,z_0)=z_0\}$
is a Banach-Lie subgroup of $G$ and in addition assume that the orbit 
$\Oc_{z_0}=\{\nu(u,z_0)\mid u\in G\}$ of $z_0$
is a submanifold of $Z$, 
and denote by $i_0\colon\Oc_{z_0}\hookrightarrow Z$ the corresponding 
embedding map. 
Let $i_0^*(D)$ be the pull-back manifold of $\Pi\colon D\to Z$ through $i_0$, that is, 
$i_0^*(D):=\{(\xi,t)\in D\times\Oc_{z_0}\mid\pi(\xi)=t\}$.
Then there exists a closed subspace $\Hc_0$ of $\Hc$ 
such that the following assertions hold: 
\begin{itemize}
\item[{\rm(i)}] For every $u\in G_0$ we have $\pi(u)\Hc_0\subseteq\Hc_0$. 
\item[{\rm(ii)}] Denote by $\pi\colon G_0\to\Bc(\Hc_0)$, 
$u\mapsto\pi(u)|_{\Hc_0}$, 
the corresponding representation of $G_0$ on $\Hc_0$, 
by $\Pi_0\colon D_0\to G/G_0$ the like-Hermitian vector bundle 
associated with the data $(\pi,\pi_0,P_{\Hc_0})$, 
and by $\Rc_0\colon D_0\to\Hc$ the transfer mapping associated with 
the data $(\pi,\pi_0,P_{\Hc_0})$ (as in subsection~\ref{subsect1.3}). 
Then there exists a biholmorphic bijective $G$-equivariant map 
$\theta\colon D_0\to i_0^*(D)$ such that 
$\theta$ sets up an isometric isomorphism of like-Hermitian 
vector bundles over $G/G_0\simeq\Oc_{z_0}$ and the diagram 
$$
\begin{CD}
D_0 @>{\theta}>> i_0^*(D) \\
@V{\Rc_0}VV @VV{\Rc|_{i_0^*(D)}}V \\
\Hc @>{\id_{\Hc}}>> \Hc
\end{CD}$$
is commutative.
\end{itemize} 
\end{theorem}

\begin{proof}
We shall take $\Hc_0:=\Ran(\Rc_{z_0})\subseteq\Hc$. 

For arbitrary $u\in G_0$ we have $\nu(u,z_0)=z_0$. 
Then property~(iii) in Definition~\ref{relationship} shows that 
we have a commutative diagram 
$$
\begin{CD}
D_{z_0} @>{\mu(u,\cdot)|_{D_{z_0}}}>> D_{z_0} \\
@V{\Rc_{z_0}}VV @VV{\Rc_{z_0}}V \\
\Hc @>{\pi(u)}>> \Hc
\end{CD}
$$
whence $\pi(u)(\Rc_{z_0}(D_{z_0}))\subseteq \Rc_{z_0}(D_{z_0})$, 
that is, $\pi(u)\Hc_0\subseteq\Hc_0$. 
Thus $\Hc_0$ has the desired property~(i). 

To prove~(ii) we first note that, since $G_0$ is a Banach-Lie subgroup of $Z$, 
it follows that the $G$-orbit $\Oc_{z_0}\simeq G/G_0$ 
has a natural structure of Banach homogeneous space of $G$ 
(in the sense of \cite{Rae77}) such that the inclusion map 
$i_0\colon\Oc_{z_0}\hookrightarrow Z$ is an embedding. 

Next define 
\begin{equation}\label{tilde}
\widetilde{\theta}\colon G\times\Hc_0\to D,\quad 
\widetilde{\theta}(u,f):=
\mu(u,\Rc_{z_0}^{-1}(f))=
\Phi_{\nu(u,z_0)}^{-1}(\pi(u)f)
\in D_{\nu(u,z_0)}\subseteq D,
\end{equation}
where the equality follows by property~(iii) in Definition~\ref{relationship}. 
Then for all $u\in G$, $u_0\in G_0$, and $f\in\Hc_0$ 
we have 
$\nu(uu_0^{-1},z_0)=\nu(u,z_0)$ and 
$$
\widetilde{\theta}(uu_0^{-1},\pi(u_0)f)
=\Rc_{\nu(uu_0^{-1},z_0)}^{-1}(\pi(uu_0^{-1})\pi(u_0)f) 
=\Rc_{\nu(u,z_0)}^{-1}(\pi(u)f)
=\widetilde{\theta}(u,f).
$$
In particular there exists a well defined map 
$$
\theta\colon G\times_{G_0}\Hc_0\to D,\quad 
[(u,f)]\mapsto\Rc_{\nu(u,z_0)}^{-1}(\pi(u)f).
$$
This mapping is $G$-equivariant with respect to 
the actions of $G$ on $G\times_{G_0}\Hc$ and on $D$ 
since so is $\widetilde{\theta}$:
for all $u,v\in G$ and $f\in\Hc_0$ 
we have 
$$
\widetilde{\theta}(uv,f)
=\Rc_{\nu(uv,z_0)}^{-1}(\pi(uv)f)  
=\Rc_{\nu(u,\nu(v,z_0))}^{-1}(\pi(u)\pi(v)f) 
=\mu\bigl(u,\Rc_{\nu(v,z_0)}^{-1}(\pi(v)f)\bigr) 
=\mu(u,\widetilde{\theta}(v,f)),
$$
where the second equality follows since $\nu\colon G\times Z\to Z$ is a group action, 
while the third equality follows by property~(iii) in Definition~\ref{relationship}. 
Besides, it is clear that $\theta$ is a bijection onto $i_0^*(D)$ and 
a fiberwise isomorphism. 
Also it is clear from the above construction of $\theta$ 
and from the definition 
of the transfer mapping $\Rc_0\colon D_0\to\Hc$ associated with 
the data $(\pi,\pi_0,P_{\Hc_0})$ that $\Rc\circ\theta=\Rc_0$, 
that is, the diagram in the statement is indeed commutative. 
In addition, since both mappings $\Rc$ and $\Rc_0$ 
are fiberwise ``isometric'' 
(see property~(ii) in Definition~\ref{relationship}), 
it follows by $\Rc\circ\theta=\Rc_0$ that 
$\theta$ gives an isometric morphism of like-$(-*)$-Hermitian 
bundles over $G/G_0\simeq\Oc_{z_0}$. 

Now we still have to prove that the map 
$\theta\colon D_0=G\times_{G_0}\Hc_0\to i_0^*(D)\subseteq D$ is biholomorphic. 
We first show that it is holomorphic. 
Since $\Oc_{z_0}$ is a submanifold of $Z$, it 
follows that $i_0^*(D)$ is a submanifold of $D$ 
(see for instance the comments after Proposition~1.4 
in Chapter~III of \cite{La01}). 

Thus it will be enough to show that $\theta\colon G\times_{G_0}\Hc_0\to D$ 
is holomorphic. 
And this property is equivalent (by Corollary~8.3(ii) in \cite{Up85}) 
to the fact that 
the mapping $\widetilde{\theta}\colon G\times\Hc_0\to D$ 
is holomorphic, 
since the natural projection $G\times\Hc_0\to G\times_{G_0}\Hc_0$ 
is a holomorphic submersion. 
Now the fact that $\widetilde{\theta}\colon G\times\Hc_0\to D$ 
is a holomorphic map follows by the first formula 
in its definition~\eqref{tilde}, since the group action 
$\alpha\colon G\times D\to D$ is holomorphic. 

Consequently the mapping $\theta\colon G\times_{G_0}\Hc_0\to i_0^*(D)$ 
is holomorphic. 
Then the fact that the inverse 
$\theta^{-1}\colon i_0^*(D)\to G\times_{G_0}\Hc_0$ 
is also holomorphic follows by 
general arguments in view of the following facts
(the first and the second of them have been already established, 
and the third one is well-known): 
Both $G\times_{G_0}\Hc_0$ and 
$i_0^*(D)$ are locally trivial holomorphic vector bundles; 
we have a commutative diagram 
$$
\begin{CD}
G\times_{G_0}\Hc_0 @>{\theta}>> i_0^*(D) \\
@VVV @VVV \\
G/G_0 @>>> \Oc_{z_0}
\end{CD}
$$ 
where the bottom arrow is the biholomorphic map 
$G/G_0\simeq\Oc_{z_0}$ induced by the action $\nu\colon G\times Z\to Z$, 
and the vertical arrows are the projections of the corresponding holomorphic 
like-$(-*)$-Hermitian vector bundles; 
the inversion mapping is holomorphic on the open set of invertible 
operators on a complex Hilbert space. 
Now the proof is finished.   
\end{proof}

The above result can be prolonged, to include morphisms of the respective categories, as follows.
Let $(\widetilde\Pi,\widetilde\pi;\widetilde\Rc)$,  $(\Pi,\pi;\Rc)$ be two objects in the category $\Rg$, and let $(\Theta;\alpha,L)$ be a morphism between them, with $\Theta=(\delta,\zeta)$. 
Suppose that both objects satisfy the assumptions of the previous theorem for $z_0\in\widetilde Z$ and 
$\zeta(z_0)\in Z$, with $\widetilde G_0$, $G_0$ and $\Oc_{z_0}$, $\Oc_{\zeta(z_0)}$ 
the isotropy subgroups and orbits, respectively.

Let $(\tilde\pi,\tilde\pi_0,P_{\widetilde\Hc_0})$, $(\pi,\pi_0,P_{\Hc_0})$ denote 
the data objects in $\HLH$ constructed out of 
$(\widetilde\Pi,\widetilde\pi;\widetilde\Rc)$,  $(\Pi,\pi;\Rc)$, respectively, 
with bundle maps $(\tilde\theta,\tilde\iota)$, $(\theta,\iota)$, as in part (ii) of Theorem~\ref{compare}. 
Put 
$\widetilde P_0=P_{\widetilde\Hc_0}$, $P_0=P_{\Hc_0}$.

\begin{proposition}\label{pullmorphi}
In the above setting, one has
$$ 
\alpha(\widetilde G_0)\subseteq G_0 \quad\text{and}\quad L(\widetilde\Hc_0)\subseteq\Hc_0.
$$
Hence $(\alpha_0,L_0):=(\alpha\mid_{\widetilde G_0},L\mid_{\widetilde\Hc_0})$ is a morphism between $(\tilde\pi,\tilde\pi_0,\widetilde P_0)$ and $(\pi,\pi_0,P_0)$. Moreover, 
\begin{itemize}
\item[{\rm(1)}] $\theta\circ[\alpha_0\times L_0]=\delta\circ\tilde\theta$.
\item[{\rm(2)}] $\Pi\circ[\alpha_0\times L_0]=\alpha_q\circ\widetilde\Pi_0$.
\item[{\rm(3)}] $\iota\circ\alpha_q=\zeta\circ\tilde\iota$.
\end{itemize}
\end{proposition}

\begin{proof}  The inclusion $\alpha(\widetilde G_0)\subseteq G_0$ is equivalent to have 
$\nu(\alpha(u),\zeta(z_0))=\zeta(z_0)$ whenever $u\in\widetilde G$ satisfies 
$\tilde\nu(u,z_0)=z_0$, and this is an easy consequence of the fact that 
$\zeta\circ\tilde\nu=\nu\circ(\alpha\times\zeta)$, see assumption (b) prior to subsection~\ref{subsect1.3}.

To see the second inclusion recall that $\widetilde\Hc_0=\widetilde\Rc_{z_0}(\widetilde D_{z_0})$ and 
$\Hc_0=\Rc_{\zeta(z_0)}(D_{\zeta(z_0)})$ (note that $\zeta(z_0)=\zeta(z_0^{-*})=\zeta(z_0)^{-*}$). 
By assumption (d) prior to subsection~\ref{subsect1.3}, we have 
$L_0\circ\widetilde\Rc_{z_0}=\Rc_{\zeta(z_0)}\circ\delta\mid_{\widetilde D_{z_0}}$, so the inclusion 
$L(\widetilde\Hc_0)\subseteq\Hc_0$ holds.

Then we are going to prove (1), (2) and (3) of the statement.

(1) Take $[(u,f)]\in\widetilde G\times_{\widetilde G_0}\widetilde\Hc_0$. Then 
$\pi(\alpha_0(u))L_0(f)=L\tilde\pi(u)f$ by (c) prior to subsection~\ref{subsect1.3}.

Thus
$$
\begin{aligned}
(\theta\circ[\alpha_0\times L_0])[(u,f)]&=\theta[(\alpha_0(u),L_0(f))] \\
&=\Rc_{\nu(\alpha(u),\zeta(z_0))}^{-1}(\pi(\alpha(u))L_0(f))
=\Rc_{\nu(\alpha(u),\zeta(z_0))}^{-1}(L\tilde\pi(u)f) \\ 
&=\Rc_{\nu(\alpha(u),\zeta(z_0))}^{-1}
((L\circ\widetilde\Rc_{\tilde\nu(u,z_0)})(\widetilde\Rc_{\tilde\nu(u,z_0)}^{-1}\tilde\pi(u)f) \\ 
&=\Rc_{\nu(\alpha(u),\zeta(z_0))}^{-1}
(\Rc_{\zeta(\tilde\nu(u,z_0))}\circ\delta)(\widetilde\Rc_{\tilde\nu(u,z_0)}^{-1}\tilde\pi(u)f) \\ 
&=(\delta\circ\widetilde\Rc_{\tilde\nu(u,z_0)}^{-1})(\tilde\pi(u)f)
=(\delta\circ\tilde\theta)[(u,f)]
\end{aligned}
$$
where we have used (d) prior to subsection~\ref{subsect1.3} in the fifth equality, 
and (b) prior to subsection~\ref{subsect1.3} in the next-to-last equality.

(2) This is immediate.

(3) This is again a consequence of (b) prior to subsection~\ref{subsect1.3}.
\end{proof}

As it has been seen before, there are canonical reproducing $(-*)$-kernels $K^{\pi}$ and $K^{\Rc}$ associated with the vector bundles $\Pi_0\colon D_0\to\Oc_{z_0}$, $\Pi\colon D\to\Oc$, respectively,  of 
Theorem \ref{compare}. 
Recall that they are defined by
$$
K^{\Rc}(s,t):=(\Rc_{s^{-*}})^{-*}\circ\Rc_t\qquad\hbox{ for }\ s,t\in Z
$$
and
$$
K^{\pi}(s,t):=(\Rc^0_{s^{-*}})^{-*}\circ\Rc_t^0
$$
or, equivalently (according to Remark \ref{kapi}),
$$
K^{\pi}(s,t):=[(u,P_0(\pi(u)^{-1}\Rc^0(\cdot)))]
$$
for every $s=uG_0, t=vG_0\in G/G_0$, where $\Rc^0$ is the transfer mapping associated with $\Pi_0$. 

\begin{corollary} 
The reproducing kernel $K^{\pi}$ is the pullback kernel of 
$K^{\Rc}$ corresponding to the bundle morphism $\Theta=(\theta,\iota)$ 
obtained in Theorem \ref{compare}, that is, 
$$ 
K^{\pi}(s,t)=(\theta_{s^{-*}})^{-*}\circ K^{\Rc}(\iota(s),\iota(t))\circ\theta_t
\qquad\hbox{ for every }\ s,t\in\Oc_{z_0}\equiv G/G_0.
$$
\end{corollary}

\begin{proof} 
It has been noticed in the proof of Theorem~\ref{compare} that 
$\Rc\circ\theta=\Rc^0$. 
Hence, for $s,t\in\Oc_{z_0}\subseteq Z$,  
$$
\begin{aligned}
(\theta_{s^{-*}})^{-*}\circ K^{\Rc}(\iota(s),\iota(t))\circ\theta_t
&=(\theta_{s^{-*}})^{-*}\circ(\Rc_{\iota(s)^{-*}})^{-*}\circ\Rc_{\iota(s)}\circ\theta_t \\ 
&=(\Rc_{s^{-*}}\circ\theta_{s^{-*}} )^{-*}\circ(\Rc_t\circ\theta_t)
=(\Rc^0_{s^{-*}})^{-*}\circ\Rc^0_t
\end{aligned}
$$
as we wanted to show.
\end{proof}

\begin{remark}\label{nonhomog}
\normalfont
The significance of Theorem~\ref{compare} is the following one: 
In the setting of Definition~\ref{relationship}, the special situation 
of subsection 1.3 is met precisely when the action 
$\nu\colon G\times Z\to Z$ is transitive, 
and in this case the transfer mapping  
is essentially the unique mapping that relates the bundle~$\Pi$ 
to the representation of the bigger group~$G$. 

On the other hand, by considering direct products of homogeneous Hermitian vector bundles, 
we can construct obvious examples of other maps relating bundles to representations 
as in Definition~\ref{relationship}. 
Moreover, the canonical examples in this paper of both classes of vector bundles 
are related with the tautological ones linked to Grassmannians on Hilbert spaces. 
See next Section~\ref{sect4}.  
\qed
\end{remark}

\section{Grassmannians and universal reproducing $(-*)$-kernels}\label{sect4}

Let $\Hc$ be a complex Hilbert space and let $\Bc(\Hc)$ be as above the
$C^*$-algebra of bounded linear operators on $\Hc$ with the involution $T\mapsto T^*$ 
defined by the adjoint mapping. 
Let $\GL(\Hc)$ be
the Banach-Lie group of all invertible elements of $\Bc(\Hc)$, and
$\U(\Hc)$ its Banach-Lie subgroup of all unitary operators on $\Hc$. 
For short, we put $\Gc=\GL(\Hc)$ and $\Uc=\U(\Hc)$.
We recall the following notation from Example~4.6 in \cite{BG07}.
\begin{itemize}
\item[$\bullet$]
  $\Gr(\Hc):=\{\Sc\mid\Sc\text{ closed linear subspace of }\Hc\}$;
\item[$\bullet$]
  $\Tc(\Hc):=\{(\Sc,x)\in\Gr(\Hc)\times\Hc\mid x\in\Sc\}
   \subseteq\Gr(\Hc)\times\Hc$;
\item[$\bullet$]
  $\Pi_{\Hc}\colon\,(\Sc,x)\mapsto\Sc$, $\Tc(\Hc)\to\Gr(\Hc)$;
\item[$\bullet$]
  for every $\Sc\in\Gr(\Hc)$ we denote by $p_{\Sc}\colon\Hc\to\Sc$
  the corresponding orthogonal projection.
\end{itemize}
These objects have the
following well known properties:
\begin{itemize}
\item[{\rm(a)}] Both $\Gr(\Hc)$ and $\Tc(\Hc)$ have structures of complex Banach manifolds ($\Gr(\Hc)$ is called the {\it Grassmannian manifold} of $\Hc$) and $\Gr(\Hc)$ carries
 a natural (non-transitive) action of $\Uc$.
(See Examples 3.11 and 6.20 in \cite{Up85}.) In fact, such an action extends to $\Gc$ as the inclusion $\Gc\hookrightarrow\Bc(\Hc)$. 
\item[{\rm(b)}] The mapping $\Pi_{\Hc}\colon\Tc(\Hc)\to\Gr(\Hc)$
is a holomorphic Hermitian vector bundle, and we call it the {\it
universal (tautological) vector bundle} associated with  the Hilbert
space $\Hc$.
\end{itemize}

\begin{definition}\label{U3}
\normalfont
Denote by $p_1,p_2\colon\Gr(\Hc)\times\Gr(\Hc)\to\Gr(\Hc)$ 
the natural projections and define 
$$
Q_{\Hc}\colon\Gr(\Hc)\times\Gr(\Hc)\to
 \Hom(p_2^*(\Pi_{\Hc}),p_1^*(\Pi_{\Hc}))
 $$
by 
$$
Q_{\Hc}(\Sc_1,\Sc_2)=(p_{\Sc_1})|_{\Sc_2}\colon\Sc_2\to\Sc_1
\text{ for all }\Sc_1,\Sc_2\in\Gr(\Hc). 
$$
The mapping $Q_{\Hc}$ is the {\it universal reproducing kernel} 
associated with the Hilbert space $\Hc$. 
\qed
\end{definition}

The vector bundle $\Pi_{\Hc}\colon\Tc(\Hc)\to\Gr(\Hc)$ being Hermitian, 
it is in particular like-Hermitian for the involution on $\Gr(\Hc)$ given by the identity map. 
Indeed, for the natural action of $\Gc$ on $\Pi_{\Hc}$, 
the natural representation $\Gc\to\Bc(\Hc)$, and the involution 
$u\mapsto u^{-1},\ \Gc\to\Gc$, 
one can associate the transfer mapping $\Rc_{\Hc}$ given by the projection in the second component
$\Rc_{\Hc}\colon(S,x)\mapsto x,\ \Tc(\Hc)\to\Hc$. 
Thus the tautological bundle $\Pi_{\Hc}\colon\Tc(\Hc)\to\Gr(\Hc)$, under the action of 
$\Gc=\GL(\Hc)$, induces an object of the category $\Rg$, 
and it is readily seen that the kernel $Q_{\Hc}$ is that one defined by the transfer mapping $\Rc_{\Hc}$. 
So the kernel $Q_{\Hc}$ defined above is an important example of object in the category $\Tran$. 
We look for more examples of like-Hermitian structures in $\Pi_{\Hc}$ 
with non-trivial involutions in the base space $\Gr(\Hc)$. 
For this purpose we introduce the following notion.

\begin{definition}\label{invol_morph}
\normalfont
An {\it involutive morphism} 
(respectively, an {\it involutive antimorphism}) 
on the tautological vector bundle $\Pi_{\Hc}$ 
is a morphism 
(respectively, an antimorphism) 
$\Omega=(\gamma,\omega)$ 
from $\Pi_{\Hc}$ into itself 
such that the following conditions are satisfied: 
\begin{itemize}
\item[{\rm(a)}] 
 The map $\omega\colon\Gr(\Hc)\to\Gr(\Hc)$ 
 is an involutive diffeomorphism which we denote by $\Sc\mapsto\Sc^{-*}$.
\item[{\rm(b)}] 
 For every $\Sc\in\Gr(\Hc)$ the mapping 
 $\gamma_{\Sc}\colon\Sc\to\Sc^{-*}$ 
 is an isometric linear (respectively, antilinear) bijective map and 
 $\gamma_{\Sc^{-*}}=(\gamma_{\Sc})^{-1}$.
\item[{\rm(c)}] 
 For all $\Sc_1,\Sc_2\in\Gr(\Hc)$ and $x_1\in\Sc_1$, $x_2\in\Sc_2$ 
 we have 
 $(x_1\mid x_2)_{\Hc}=(\gamma_{\Sc_1}(x_1)\mid \gamma_{\Sc_2}(x_2))_{\Hc}$ 
(respectively, 
  $(x_1\mid x_2)_{\Hc}=\overline{(\gamma_{\Sc_1}(x_1)\mid \gamma_{\Sc_2}(x_2))}_{\Hc}$).
\end{itemize}
Now assume that $\Omega$ is an involutive morphism 
on the tautological vector bundle $\Pi_{\Hc}$. 
Then the {\it like-Hermitian structure associated with $\Omega$} 
is the structure 
$\{(\cdot\mid\cdot)_{\Sc,\Sc^{-*}}\}_{\Sc\in\Gr(\Hc)}$ 
on $\Pi_{\Hc}$ defined by 
$$
(x\mid y)_{\Sc,\Sc^{-*}}=(x\mid\gamma_{\Sc^{-*}}(y))_{\Hc}
$$
whenever $x\in\Sc$, $y\in\Sc^{-*}$, and $\Sc\in\Gr(\Hc)$. 
It is easy to see that this is indeed a like-Hermitian structure. 
The vector bundle $\Pi_{\Hc}$ endowed with 
the like-Hermitian structure associated with $\Omega$ 
will be denoted by~$\Pi_{\Hc,\Omega}$. 
The {\it reproducing $(-*)$-kernel associated with $\Omega$} 
is the reproducing $(-*)$-kernel on $\Pi_{\Hc,\Omega}$, 
$$
Q_{\Hc,\Omega}\colon\Gr(\Hc)\times\Gr(\Hc)\to
 \Hom(p_2^*(\Pi_{\Hc,\Omega}),p_1^*(\Pi_{\Hc,\Omega}))
 $$
defined by 
$Q_{\Hc,\Omega}(\Sc_1,\Sc_2)=p_{\Sc_1}\circ\gamma_{\Sc_2}
 \colon\Sc_2\to\Sc_1$
whenever $\Sc_1,\Sc_2\in\Gr(\Hc)$. 
\qed
\end{definition}

Let us show that $Q_{\Hc,\Omega}$ is a reproducing $(-*)$-kernel on $\Pi_{\Hc,\Omega}$. 
For all 
$\Sc_1,\dots,\Sc_n\in\Gr(\Hc)$ and $x_j\in\Sc_j^{-*}$ for $j=1,\dots,n$ 
we have
$$
\begin{aligned}
\sum_{j,l=1}^n
 (Q_{\Hc,\Omega}(\Sc_l,\Sc_j^{-*})x_j\mid x_l)_{\Sc_l,\Sc_l^{-*}}
 &=\sum_{j,l=1}^n
 (p_{\Sc_l}(\gamma_{\Sc_j^{-*}}(x_j))\mid x_l)_{\Sc_l,\Sc_l^{-*}} 
  =\sum_{j,l=1}^n
 (p_{\Sc_l}(\gamma_{\Sc_j^{-*}}(x_j))\mid 
     \gamma_{\Sc_l^{-*}}(x_l))_{\Hc} \\
 &=\sum_{j,l=1}^n
 (\gamma_{\Sc_j^{-*}}(x_j)\mid 
     \gamma_{\Sc_l^{-*}}(x_l))_{\Hc} 
  =
 (\sum_{j=1}^n\gamma_{\Sc_j^{-*}}(x_j)\mid 
     \sum_{l=1}^n\gamma_{\Sc_l^{-*}}(x_l))_{\Hc}\ge 0
\end{aligned}
$$
and 
$$
\begin{aligned}
(x_j\mid Q_{\Hc,\Omega}(\Sc_j,\Sc_l^{-*})x_l)_{\Sc_j^{-*},\Sc_j} 
 &=(x_j\mid p_{\Sc_j}(\gamma_{\Sc_l^{-*}}(x_l)))_{\Sc_l,\Sc_l^{-*}}
  =(x_j\mid 
    \gamma_{\Sc_j}(p_{\Sc_j}(\gamma_{\Sc_l^{-*}}(x_l))))_{\Hc}\\
 &=(\gamma_{\Sc_j^{-*}}(x_j)\mid 
    p_{\Sc_j}(\gamma_{\Sc_l^{-*}}(x_l)))_{\Hc} 
  =(\gamma_{\Sc_j^{-*}}(x_j)\mid 
    \gamma_{\Sc_l^{-*}}(x_l))_{\Hc}\\
 &=(Q_{\Hc,\Omega}(\Sc_l,\Sc_j^{-*})x_j\mid x_l)_{\Sc_l,\Sc_l^{-*}},
\end{aligned}
$$
where the latter equality was obtained during the previous calculation. 

Hence we have that $Q_{\Hc,\Omega}$ is an object of the category $\Kern$. 
We would like to have $Q_{\Hc,\Omega}$ in $\Tran$ but there is no reason 
for the existence of a transfer mapping, associated with the bundle $\Pi_{\Hc}$, 
from which $Q_{\Hc,\Omega}$ could emerge.

\begin{remark}\label{U5.6}
\normalfont
Every involutive isometric linear operator $C\colon\Hc\to\Hc$ 
defines an involutive morphism $\Omega_C=(\gamma_C,\omega_C)$ on 
$\Pi_{\Hc}$ 
by the formulas 
$$
\omega_C\colon\Sc\mapsto C(\Sc),\quad \Gr(\Hc)\to\Gr(\Hc),$$
and $(\gamma_C)_{\Sc}:=C|_{\Sc}\colon\Sc\to C(\Sc)$ 
whenever $\Sc\in\Gr(\Hc)$. 
If $C\colon\Hc\to\Hc$ is an involutive isometric antilinear operator, 
then the same formulas define an involutive antimorphism 
$\Omega_C=(\gamma_C,\omega_C)$ on 
$\Pi_{\Hc}$.

We shall write $Q_{\Hc,C}$ to refer to $Q_{\Hc,\Omega}$ in the case when $\Omega=\Omega_{C}$. So we have 
$$
Q_{\Hc,C}(S_1,S_2)=p_{S_1}\circ C\colon S_2\to S_1\quad \text{ for }S_1,S_2\in \Gr(\Hc).
$$
Similarly, we put $\Pi_{\Hc,C}=\Pi_{\Hc,\Omega_C}$.

Let us take a closer look at this example. 
Firstly, we notice that if we take the involution in $\Gc$ given by $u\mapsto Cu^{-1}C$, $\Gc\to\Gc$, 
then the vector bundle $\Pi_{\Hc,C}$, satisfies, under the natural action of $ \Gc$, 
those properties stated in Definition~\ref{relationship} 
which are independent of a (possible) transfer mapping. 
Furthermore, let $\Rc_{\Hc}^C$ denote the map 
$\Rc_{\Hc}\colon(S,x)\mapsto x,\ \Tc(\Hc)\to\Hc$ considered after 
Definition \ref{U3}, where we assume that $\Pi_{\Hc}$ is endowed with 
the like-Hermitian pairings $(x\mid y)_{S,S^{-*}}:=(x\mid Cy)_{\Hc}$, for every 
$S\in\Gr(\Hc)$, $x\in S$, $y\in C(S)$. 
Then $\Rc_{\Hc,S}$ is the inclusion $\iota_S\colon S\hookrightarrow\Hc$, so that for every $y\in S^{-*}$ 
and $h\in\Hc$,
$$
\begin{aligned}
((\Rc_{\Hc,S^{-*}})^{-*}h\mid y)_{S,S^{-*}}&=(h\mid\Rc_{\Hc,S^{-*}}y)_{\Hc}
=(h\mid y)_{\Hc} 
=(Ch\mid Cy)_{\Hc} \\
&=(p_SC(h)\mid Cy)_{\Hc}=(p_SC(h)\mid Cy)_{S,S^{-*}}.
\end{aligned}
$$
Hence $(\Rc_{\Hc,S^{-*}})^{-*}=p_S\circ C$ whence we get 
$(\Rc_{\Hc,S_1}^{-*})^{-*}\circ\Rc_{\Hc,S_2}
=p_{S_1}\circ C\circ\iota_{S_2}=Q_{\Hc,C}(S_1,S_2)$.

Besides the above equality, properties (i) and (iii) in Definition~\ref{relationship} 
are clearly satisfied by $\Rc_{\Hc}^C$. 
Thus the map $\Rc_{\Hc}^C$ resembles a transfer mapping for $\Pi_{\Hc,C}$ 
as those of  the category $\Rg$. 
If $\Rc_{\Hc}^C$ were such a mapping then  $Q_{\Hc,C}$ would become a kernel in the category $\Tran$. 
But the mapping $\Rc_{\Hc}^C$ is not a transfer mapping because 
it does not satisfy property (ii) of Definition~\ref{relationship}. 
\qed
\end{remark}

By applying Theorem~\ref{compare} to the bundle $\Pi_{\Hc}$ 
(with the identity map taken as involution in $\Gr(\Hc)$) 
we obtain objects of the category  $\HLH$ in terms of Grassmannians. 
We next  examine this example in detail.

Let $\Sc_0\in\Gr(\Hc)$. 
The orbit of $\Sc_0$ in $\Gr(\Hc)$ under the collineation action by $\Gc$ 
will be denoted by $\Gr_{\Sc_0}(\Hc)$. 
It is well known that $\Gr_{\Sc_0}(\Hc)$ coincides with the unitary orbit of $\Sc_0$ 
and with the connected component of $\Sc_0$ in $\Gr(\Hc)$, and that it is given by
$$
\begin{aligned}
\Gr_{\Sc_0}(\Hc)
&:=\{u\Sc_0\mid u\in\GL(\Hc)\}=\{u\Sc_0\mid u\in\U(\Hc)\} \\
&=\{\Sc\in\Gr(\Hc)\mid\dim\Sc=\dim\Sc_0\text{ and
}\dim\Sc^\perp=\dim\Sc_0^\perp\} \\
&\simeq\U(\Hc)/(\U(\Sc_0)\times\U(\Sc_0^\perp)).
\end{aligned}
 $$
(See Proposition~23.1 in \cite{Up85} or Lemma~\ref{U3} in \cite{BG07}.)

Let $\Gc([p_{\Sc_0}])$ denote the isotropy subgroup of $\Sc_0$; that is, 
$\Gc([p_{\Sc_0}]):=\{u\in\Gc\mid u\Sc_0=\Sc_0\}$. 
This notation has been taken from \cite{BG07} 
where it was suggested for reasons which can be found in that paper. 
Then
$$
\Gc/\Gc([p_{\Sc_0}])\simeq\Gr_{\Sc_0}(\Hc),
$$
where the symbol \lq\lq\ $\simeq$ " means diffeomorphism between the
respective differentiable structures, and that the differentiable
structure of the quotient space is the one associated with the quotient topology.

Set
$\Tc_{\Sc_0}(\Hc):=\{(\Sc,x)\in\Tc(\Hc)\mid\Sc\in\Gr_{\Sc_0}(\Hc)\}$.
The vector bundle $\Pi_{\Hc,\Sc_0}\colon\Tc_{\Sc_0}(\Hc)\to\Gr_{\Sc_0}(\Hc)$ obtained by
restriction of $\Pi_{\Hc}$ to $\Tc_{\Sc_0}(\Hc)$ will be called here
the universal vector bundle {\it at $\Sc_0$}. 
It is also Hermitian and holomorphic. 
Moreover it is  biholomorphically diffeomorphic to the vector bundle  
$\Gc\times_{\Gc([p])}\Sc_0\to\Gc/\Gc([p])$ where we write $p:=p_{\Sc_0}$ for brevity. 
Such a diffeomorphism is induced by the map defined between the total spaces of the respective bundles as  
$([u,x])\mapsto(u\Sc_0,ux),\ \Gc\times_{\Gc([p])}\Sc_0\to\Tc_{\Sc_0}$.

Then, by restriction to $\Pi_{\Hc,\Sc_0}$ of the pertinent elements previously considered in this section, we obtain that $\Pi_{\Hc,\Sc_0}$ defines an object of the category $\HLH$. So in particular the kernel
$$
Q_{\Hc,\Sc_0}(u_1\Sc_0,u_2\Sc_0)(u_2x):=u_1p_{\Sc_0}u_1^{-1}u_2x,
\qquad (u_1,u_2\in\Gc,\,x\in\Sc_0)
$$
is an object coming from the category $\Rep$.

As in the case of $\Pi_{\Hc}$ we wish to find now involutions in $\Gr_{\Sc_0}(\Hc)$ different from the identity map. 
For this, one can take an involutive morphism satisfying an additional assumption. 
We illustrate this point by considering involutive isometries $C$ on $\Hc$.

Specifically, let $C\colon\Hc\to\Hc$ be a continuous involutive map 
which is either ${\mathbb C}$-linear or conjugate-linear, 
and let $\Sc_0$ be a closed linear subspace 
such that $C(\Sc_0)=\Sc_0$. 
It turns out that the orbit $\Gr_{\Sc_0}(\Hc)$ 
is invariant under the involution $\Sc\mapsto C(\Sc)$ 
of $\Gr(\Hc)$, and then by restriction we will get an involution 
on $\Gr_{\Sc_0}(\Hc)$. 
In fact, 
 \begin{equation}\label{G}
\Gr_{\Sc_0}(\Hc)=\{D(\Sc_0)\mid D\in\Gc\}.
\end{equation}
Now let $\Sc\in\Gr_{\Sc_0}(\Hc)$ arbitrary. 
Then there exists $D\in\G(A)$ with $\Sc=D(\Sc_0)$, 
so that $C(\Sc)=CD(\Sc_0)=CDC^{-1}(\Sc_0)$ since $C(\Sc_0)=\Sc_0$. 
Since $CDC^{-1}\in\Gc$ irrespective of whether $C$ is ${\mathbb C}$-linear or conjugate-linear, 
it then follows by \eqref{G} that $C(\Sc)\in\Gr_{\Sc_0}(\Hc)$. 
Thus the orbit $\Gr_{\Sc_0}(\Hc)$ is indeed preserved by the involution $\Sc\mapsto C(\Sc)$ of $\Gr(\Hc)$. 
This property 
corresponds to the fact that the quotient manifold $\Gc/\Gc([p])$, which is diffeomorphic to 
$\Gr_{\Sc_0}(\Hc)$, is invariant under the involution $u\Gc([p])\mapsto u^{-*}\Gc([p])$ where 
$u^{-*}:=CuC$ for every $u\in\Gc$.

Let us point out what is the form of the kernel $Q_{\Hc,C}$ restricted to the bundle $\Pi_{\Hc,\Sc_0}$. 
Suppose that $\Sc\in\Gr_{\Sc_0}(\Hc)$. Then there exists $u\in\U(\Hc)$ such that 
$u\Sc_0=\Sc$ and $u\Sc_0^\perp=\Sc^\perp$. 
Then $up_{\Sc_0}=p_{\Sc}u$, that is, 
$p_{\Sc}=up_{\Sc_0}u^{-1}$. 
Thus for all $u_1,u_2\in\U(\Hc)$ and $x_1,x_2\in\Sc_0$ 
we have 
$$
Q_{\Hc,C}(u_1\Sc_0,u_2\Sc_0)(Cu_2x_2)
=p_{u_1\Sc_0}(Cu_2x_2)=u_1p_{\Sc_0}(u_1^{-1}Cu_2x_2).
$$ 
In short, $Q_{\Hc,C}(\Sc_1,\Sc_2)=Q_{\Hc}(\Sc_1,C(\Sc_2))\circ C$.

\medskip
The so chosen involution in $Gr_{\Sc_0}(\Hc)$ depends obviously on a pre-fixed involutive isometry 
$C$ on $\Hc$.Thus it seems to be too much having the pretension to obtain universality results associated with the manifold $Gr_{\Sc_0}(\Hc)$ in this situation. 
Searching for an alternative way to overcome this obstacle we are led to consider 
a (canonical) complexification of $Gr_{\Sc_0}(\Hc)$.

Let $B:=\{p\}'$ be the commutant subalgebra of $p$ in $\Bc(\Hc)$. 
The algebra $B$ is formed by the operators $T$ such that $T(\Sc_0)\subset\Sc_0$, $T(\Sc_0^\perp)\subset\Sc_0^\perp$. 
Moreover, $B$ is self-adjoint so that it is in fact a $C^*$-subalgebra of $\Bc(\Hc)$. 
Let $\Gc(p)$ denote the group of invertibles in $B$. 
Clearly, $\Gc(p)$ is a (closed) Banach-Lie subgroup of $\Gc$, 
which moreover is stable under the operation of taking adjoints $u\mapsto u^*$.

Then, endowed with the quotient topology, the space $\Gc/\Gc(p)$ is 
a homogeneous Banach manifold enjoying the natural involution given by 
$$  
u\Gc(p)\mapsto u^{-*}\Gc(p),\ \Gc/\Gc(p)\to\Gc/\Gc(p)
$$
where $u^{-*}$ is this time the inverse of the adjoint operator of $u\in\Gc$.

Further, the quotient manifold $\Gc/\Gc(p)$ can be described as the set of pairs 
$\{(u\Sc_0,u^{-*}\Sc_0):u\in\Gc\}$ so that the former involution takes the form 
$$
(u\Sc_0,u^{-*}\Sc_0)\mapsto(u^{-*}\Sc_0,u\Sc_0).
$$ 
Note that the fixed-point set of such an involution is 
$$
\{(u\Sc_0,u\Sc_0)\mid u\in\Uc\}\equiv\Gr_{\Sc_0}(\Hc).
$$
The preceding remarks tell us in particular that the orbit of projections 
$\Gc/\Gc(p)\simeq\{upu^{-1}:u\in\Gc\}$ is a complexification of the (complex, in turn) manifold 
$\Gr_{\Sc_0}(\Hc)$. This fact can be extended up to the level of vector bundles, so that the bundle 
$\Gc\times_{\Gc(p)}\Sc_0\to\Gc/\Gc(p)$, which we shall denote by 
$\Pi_{\Hc}^{\C}$ 
in the sequel, becomes a complexification of the bundle  
$\Gc\times_{\Gc([p])}\Sc_0\to\Gr_{\Sc_0}(\Hc)$, in a sense which is made precise in \cite{BG08}, Theorem 4.4. 
See \cite{BG07} and \cite{BG08} for details about the above results.

For $\Hc_A:=\Hc$, $\Hc_B:=\Sc_0$, $p:=p_{\Sc_0}$, the inclusion mapping 
$\pi_A\colon\Gc\hookrightarrow\Bc(\Hc)$ and the restriction mapping 
$\pi_B:=\pi_A\mid_{\Gc(p)}$, it is clear that 
the vector bundle $\Pi_{\Hc,\Sc_0}$ gives rise to the object $(\pi_A,\pi_B;p)$ in 
(the notation for objects in) the category $\HLH$. 
In fact, $(\pi_A,\pi_B;p)$ can be obtained as an object in the category $\SLH$. 
More precisely, for every closed subspace $\Kc$ of $\Sc_0$ 
there is a completely positive mapping $\Phi_{\Kc}$ on $\Bc(\Hc)$ 
which by application of the Stinespring method (see subsection~\ref{subsect1.4}) 
yields the vector bundle $\Pi_{\Hc,\Sc_0}$.

\begin{proposition}\label{cpgras}
In the above setting, put 
$E_p\colon T\mapsto pTp+(1-p)T(1-p),\ \Bc(\Hc)\to\{p\}'$. 
Let $\Kc$ be a closed subspace of $\Sc_0$ and define the compression mapping
$$  
\Phi_{\Kc}\colon T\mapsto p_{\Kc}\circ T\circ\iota_{\Kc},\ \Bc(\Hc)\to\Bc(\Kc).
$$
Then $\Phi_{\Kc}$ is a unital completely positive mapping with 
$\Phi_{\Kc}\circ E_p=\Phi_{\Kc}$ and such that the vector bundle defined by
$(\Bc(\Hc), \{p\}',E_p;\Phi_{\Kc})$ coincides with $\Pi_{\Hc,\Sc_0}$.
\end{proposition}

\begin{proof} The mapping $\Phi_{\Kc}$ of the statement is clearly unital. 
It is also completely positive:
For an integer $n$, a positive matrix $(T_{i,j})_{1\leq i,j\leq n}$ in $M_n(\Bc(\Hc))$ and 
$x_1,\dots,x_n\in\Kc$,
$$
\sum_{i,j=1}^n(\Phi_{\Kc}(T_{i,j})x_j\mid x_i)_{\Kc}
=\sum_{i,j=1}^n(p_{\Kc}T_{i,j}x_j \mid x_i)_{\Kc}
=\sum_{i,j=1}^n(T_{i,j}x_j \mid x_i)_{\Hc}\geq0
$$
since $(T_{i,j})_{1\leq i,j\leq n}$ is positive.

Take $T$ in $\Bc(\Hc)$. Then 
$\Phi_{\Kc}(E_p(T))=p_{\Kc}\circ Ep(T)\circ\iota_{\Kc}
=p_{\Kc}(pTp+(1-p)T(1-p))\iota_{\Kc}$. Since $\Kc\subseteq\Sc_0$ we have 
$(1-p)\iota_{\Kc}=0$, whence 
 $\Phi_{\Kc}(E_p(T))=(p_{\Kc}\circ p)\circ T\circ\iota_{\Kc}$. 
 Also, for every 
 $x,y\in\Kc$ we have $(p_{\Kc}(1-p)x\mid y)=((1-p)x\mid p_{\Kc}y)=0$ since 
 $p_{\Kc}y\in\Kc\subseteq\Sc_0$ and $(1-p)x\in\Sc_0^{\perp}$. 
This means that $p_{\Kc}p=p_{\Kc}$. 
Thus $\Phi_{\Kc}E_p=\Phi_{\Kc}$.

Now we proceed with the Stinespring method. 
For every $b,a\in\Bc(\Hc)$ and $x,y\in\Kc$ we have 
$$
(\Phi_{\Kc}(a^*b)x\mid y)_{\Kc}=(p_{\Kc}(a^*bx)\mid y)_{\Kc}
=(a^*bx\mid y)_{\Hc}=(bx\mid ay)_{\Hc}.
$$
From this, it follows that the norm $\Vert\cdot\Vert_{\Phi_\Kc}$ on $\Bc(\Hc)\otimes\Kc$ 
defined by the mapping $\Phi_\Kc$ as in subsection~\ref{subsect1.4} is given by
$$
\Vert\sum_{j=1}^n b_j\otimes x_j\Vert_{\Phi_\Kc}^2
=\Vert\sum_{j=1}^n b_jx_j\Vert^2 
$$
for every $\sum_{j=1}^n b_j\otimes x_j\in\Bc(\Hc)\otimes\Kc$. 
Hence the associated null space $N$ is formed by all the above elements 
$\sum_{j=1}^n b_j\otimes x_j\in\Bc(\Hc)\otimes\Kc$ such that 
$\sum_{j=1}^n b_jx_j=0$.
 On the other hand, the map 
$\sum_{j=1}^n b_j\otimes x_j\mapsto\sum_{j=1}^n b_jx_j$, 
$\Bc(\Hc)\otimes\Kc\to\Hc$ is well defined 
(note that $\Hc$ is a left module on $\Bc(\Hc)$ for the natural action) and surjective, and so we get that 
$$
[\Bc(\Hc)\otimes\Kc]/N\simeq \Hc.
$$
Moreover, it is straightforward to derive from the above that 
the Stinespring dilation obtained from $\Phi_\Kc$ is the identity mapping 
$\iota\colon\Bc(\Hc)\to\Bc(\Hc)$.

In an analogue way one can show that $[\{p\}'\otimes\Kc]/N_0\simeq \Sc_0$ 
where $N_0$ is the null space associated with the previous norm restricted to $\{p\}'\otimes\Kc$. 
Here the only point to comment is perhaps the surjectivity of the mapping
$$ 
\sum_{j=1}^n b_j\otimes x_j\mapsto\sum_{j=1}^n b_jx_j, \ 
\{p\}'\otimes\Kc\to\Sc_0.
$$
This can be proved for instance as follows. 
Given $x\in\Sc_0$ one can choose $T$ in $\Bc(\Hc)$ and $k\in\Kc$ such that $T(k)=x$. 
Then $E_p(T)\in\{p\}'$ and $E_p(T)k=x$ since $\Kc\subseteq\Sc_0$. 
Hence the above map is onto. 

Putting all the above facts together the statement of the proposition follows readily.
\end{proof}





\section{The universality theorems}\label{sect5}

We state and prove in this section the fundamental results 
which enable us to recover reproducing $(-*)$-kernels as pull-backs of the canonical kernels 
associated with Grassmannian tautological bundles. The first theorem shows up the existence of such pull-backs in purely algebraic terms.

\begin{theorem}\label{U6}
Let $\Pi\colon D\to Z$ be a like-Hermitian vector bundle. 
Denote by $p_1,p_2\colon Z\times Z\to Z$ the natural projections 
and let $K\colon Z\times Z\to\Hom(p_2^*\Pi,p_1^*\Pi)$ 
be a reproducing $(-*)$-kernel. 
Then let $\Hc^K$ be 
the reproducing kernel Hilbert space associated with $K$, 
and define 
$$
\widehat{K}\colon D\to\Hc^K,\quad 
\widehat{K}(\xi)=K_\xi=K(\cdot,\Pi(\xi))\xi\colon Z\to D,
$$ 
and 
$$
\zeta_K\colon Z\to\Gr(\Hc^K),\quad 
\zeta_K(s):=\overline{\widehat{K}(D_s)}.
$$
Assume that we have an involutive morphism $\Omega=(\gamma,\omega)$ 
on $\Pi_{\Hc}$ 
such that the involutive diffeomorphism 
$\Sc\mapsto\Sc^{-*}:=\omega(\Sc)$ 
of $\Gr(\Hc^K)$ satisfies 
\begin{equation}\label{compatible}
(\forall s\in Z)\quad \zeta_K(s^{-*})=\zeta_K(s)^{-*}.
\end{equation}
Then there exists
a vector bundle morphism 
$\Delta_K=(\delta_K,\zeta_K)$ from $\Pi$ into 
the like-Hermitian vector bundle~$\Pi_{\Hc^K,\Omega}$ such that 
$K$ is equal to the pull-back of the reproducing $(-*)$-kernel 
$Q_{\Hc^K,\Omega}$ by $\Delta_K$, that is, 
$K=(\Delta_K)^*Q_{\Hc^K,\Omega}$.  
\end{theorem}

\begin{proof}
Denote 
$$
L_K:=\{(s,x)\in Z\times\Hc^K\mid x\in\zeta_K(s)\}\subseteq 
Z\times\Hc^K,
$$
and define 
$$
\Lambda_K\colon L_K\to Z,\quad (s,x)\mapsto s.
$$
It follows by \eqref{compatible} that 
$\Lambda_K\colon L_K\to Z$ is a like-Hermitian vector bundle, 
namely this is the pull-back of the tautological bundle 
$\Pi_{\Hc^K,\Omega}\colon\Tc(\Hc^K)\to\Gr(\Hc^K)$ 
by the mapping $\zeta_K\colon Z\to\Gr(\Hc^K)$.

On the other hand, define 
$$
\check{K}\colon D\to L_K,\quad 
\check{K}(\xi)=(\Pi(\xi),\widehat{K}(\xi))=(\Pi(\xi),K_\xi)
$$
and 
$$
\psi_K\colon L_K\to\Tc(\Hc^K),\quad\psi_K(s,x)=(\zeta_K(s),x).
$$
Then we get the commutative diagram 
$$\begin{CD}
D @>{\check{K}}>> L_K @>{\psi_K}>> \Tc(\Hc^K) \\
@V{\Pi}VV @VV{\Lambda_K}V @VV{\Pi_{\Hc^K}}V \\
Z @>{\id_Z}>> Z @>{\zeta_K}>> \Gr(\Hc^K)
\end{CD}$$
where the vertical arrows are like-Hermitian vector bundles, 
while the pairs of horizontal arrows $(\check{K},\id_Z)$ 
and $(\psi_K,\zeta_K)$ 
are homomorphisms between these bundles.

We now define $\delta_K=\psi_K\circ\check{K}$ and check that 
the vector bundle homomorphism $\Delta_K:=(\delta_K,\zeta_K)$ 
has the wished-for property $K=(\Delta_K)^*Q_{\Hc^K,\Omega}$. 
In fact, let $t,s\in Z$ and $\eta\in D_t$, $\xi\in D_s$ arbitrary. 
Then 
$$
\begin{aligned}
(((\Delta_K)^*Q_{\Hc^K,\Omega}(s^{-*},t)\eta\mid\xi)_{s^{-*},s} 
 &=((\delta_K|_{D_{s}})^{-*}\circ 
     Q_{\Hc^K,\Omega}(\zeta_K(s^{-*}),\zeta_K(t))
   \circ(\delta_K|_{D_t}))\eta\mid\xi)_{s^{-*},s} \\
 &=(Q_{\Hc^K,\Omega}(\zeta_K(s^{-*}),\zeta_K(t))\delta_K\eta\mid
    \delta_K\xi)_{\zeta_K(s)^{-*},\zeta_K(s)} \\
 &=(Q_{\Hc^K,\Omega}(\zeta_K(s^{-*}),\zeta_K(t))K_\eta\mid
    K_\xi)_{\zeta_K(s)^{-*},\zeta_K(s)} \\
 &=((p_{\zeta_K(s^{-*})}\circ\gamma_{\zeta_K(t)}(K_\eta)\mid 
     \gamma_{\zeta_K(s)}(K_\xi))_{\Hc^K} \\
 &=(\gamma_{\zeta_K(t)}(K_\eta)\mid
 \gamma_{\zeta_K(s)}(K_\xi))_{\Hc^K} \\
 &=(K_\eta\mid K_\xi)_{\Hc^K} 
 =(K(s^{-*},t)\eta\mid\xi)_{s^{-*},s}
\end{aligned}
$$
so that indeed $K=(\Delta_K)^*Q_{\Hc^K,\Omega}$.
\end{proof}

The construction of the mapping $\zeta_K\colon Z\to\Gr(\Hc^K)$ 
in Theorem~\ref{U6} is inspired by the mapping 
$\Zc$ defined in formula~(16) in \cite{MP97}, 
which in turn extends the corresponding maps for complex {\it line} bundles given in \cite{Od92}. 
In the latter reference it is shown that categories formed by objects like $\zeta_K$ (or $\Zc$) 
are equivalent to categories of vector bundles with distinguished kernels. 
Next, and under natural conditions, we extend that result by showing that 
the pull-back operation giving the  kernel $K$ in Theorem~\ref{U6} 
also enables us to recover the whole vector bundle $\Pi$.

\begin{corollary}\label{cor_invert}
Under the conditions of Theorem~\ref{U6}, if, in addition, we assume that 
for every $s\in Z$
\begin{equation}\label{invert}
K(s,s)\colon D_s\to D_s\text{ is invertible}, 
\end{equation}
then the vector bundle morphism $(\check{K},\id_Z)$ is an (algebraic) isomorphism 
between the bundles $\Pi\colon D\to Z$ and $\Lambda_K\colon L_K\to Z$, so $\Pi\colon D\to Z$
is isomorphic to the pull-back bundle 
$\zeta_K^*\Tc(\Hc^K)\to Z$ defined by the mapping $\zeta_K$.
\end{corollary}

\begin{proof} 
In view of Theorem~\ref{U6} it suffices to show that the map $\check{K}$ has an inverse. 
If $(s,x)\in L_K$ then $x\in\zeta_K(s)=\widehat K(D_s)$ with $\widehat K$ injective, so there exists a unique 
$\xi\in D_s$ with $x=\widehat K(\xi)$. 
Then we obtain the inverse of $\check K$ by defining 
$\check{K}^{-1}\colon L_K\to D$ by 
$\check{K}^{-1}(s,x):=(\widehat K|_{D_s})^{-1}(x)$ for all $s\in Z$ and $x\in\zeta_K(s)$.
\end{proof}

Now, of course, there is the question of the continuity and/or smoothness properties of 
the pull-back morphism $\Delta_K$ 
in the preceding theorem. Since 
$\Delta_K=(\delta_K,\zeta_K)$ with $\delta_K(\xi)=(\zeta_K(\Pi(\xi)),\widehat K(\xi))$ for all $\xi\in D$, 
the morphism $\Delta_K$ is continuous if and only if $\widehat K$ and $\zeta_K$ are continuous mappings. 
In order to prove the continuity of $\widehat K$ it seems natural to assume the continuity of $K$. 

For a trivialization $\Psi_U$ of $\Pi$ over the open set $U\subseteq Z$, let us denote 
$\Psi_U^{-1}(\xi)=(\Pi(\xi),\varphi_U(\xi))$, so that 
$\varphi_U\colon D_{\Pi(\xi)}\to\Ec$ is a linear topological isomorphism, for all $\xi\in\Pi^{-1}(U)$. 
We shall sometimes identify~$\xi$ with $(\Pi(\xi),\varphi_U(\xi))$ putting $\xi\equiv(\Pi(\xi),\varphi_U(\xi))$. 
Let $\Pi_{p_2^*p_1^*}$ denote the vector bundle 
$\Hom(p_2^*\Pi,p_1^*\Pi)\to Z\times Z$. 
If $U,V$ are open sets in $Z$ of corresponding trivializations 
$\Psi_U$, $\Psi_V$ one has 
$$
\Pi^{-1}_{p_2^*p_1^*}(U\times V)=\bigcup_{(s,t)\in U\times V} \Bc(D_t,D_s)$$
with associated isomorphisms
$$
\Psi^{-1}_{U\times V}\colon 
T\mapsto(s,t,T_\Ec),\ \Pi^{-1}_{p_2^*p_1^*}(U\times V)\to U\times V\times \Bc(\Ec,\Ec)
$$
given by 
$$
T_\Ec(x)=\varphi_U\left(T(\Psi_V(t,x))\right),\qquad (x\in\Ec),
$$
for $T\in\Bc(D_t,D_s)$, and
$$
\Psi_{U\times V}\colon
(s,t,T_\Ec)\mapsto T_{st}(\eta),\ U\times V\times \Bc(\Ec,\Ec)\to\Pi^{-1}_{p_2^*p_1^*}(U\times V)
$$
such that 
$$
T_{st}(\eta)=\Psi_U(s,T_\Ec(\varphi_U(\eta))),\qquad (\eta\in D_t).
$$
The basic neighborhoods of any $T\colon D_t\to D_s$ of $\Hom(p_2^*\Pi,p_1^*\Pi)$ have the form
$\Psi_{U\times V}(U_0\times V_0\times B_\rho(T_\Ec))$ 
where $U$, $V$ are trivialization open subsets of $Z$, $U_0\subseteq U$ and $V_0\subseteq V$ 
run over the family of basic neighborhoods of $t,s$, respectively, 
and $B_\rho:=\{L\in\Bc(\Ec,\Ec)\mid\Vert L\Vert_{\Ec\to\Ec}<\rho\}$ with 
$B_\rho(T_\Ec)=T_\Ec+B_\rho$, for $\rho>0$.

If $K$ is a reproducing $(-*)$- kernel and $U,V$ are as above, we put
$$
K^{U,V}_\Ec(s,t):=\Psi_{U\times V}^{-1}\circ K|_{U\times V}(s,t), \qquad (s,t)\in U\times V.
$$
We shall identify $K^{U,V}_\Ec(s,t)$ with $K(s,t)_\Ec$ for all $(s,t)\in U\times V$. 
Clearly, $K$ is continuous if and only if $K_\Ec$ is.

\begin{lemma}\label{Kcontin} 
If the reproducing kernel $K$ is continuous then the mapping
$$
\widehat{K}\colon D\to\Hc^K,\quad 
\widehat{K}(\xi)=K_\xi=K(\cdot,\Pi(\xi))\xi\colon Z\to D,
$$
is continuous.
\end{lemma}

\begin{proof} 
First, take $s_0\in Z$ and a trivialization isomorphism $\Psi_U$ with $s_0\in U$. 
For any $t\in U$, define the continuous sesquilinear form $(\cdot\mid\cdot)_{U,t}$ by
$$
(x\mid y)_{U,t}=(\Psi_U(t,x)\mid\Psi_{U^{-*}}(t^{-*},y))_{t,t^{-*}}\ ,\quad\text{for } x,y\in\Ec.
$$
For $x,y\in\Ec$ and $t$ near $s_0$ there is a constant $C_U(x,y)>0$ such that 
$\vert(x\mid y)_{U,t}\vert\leq C_U(x,y)$, by property~(c) in Definition~\ref{like}. 
Then by using the Banach-Steinhaus theorem we get a constant $C_U>0$ such that
\begin{equation}\label{bilinear}
\vert(x\mid y)_{U,t}\vert\leq C_U\Vert x\Vert\ \Vert y\Vert,\quad\text{for } x,y\in\Ec,\ t\in U', 
\end{equation}
for a suitable neighborhood $U'\subseteq U$ of $s_0$.

Now, given $\rho>0$, since $K$ is continuous there exists an open set $U_0\subseteq U'$ with $s_0\in U_0$, such that 
$$
K(U^{-*}_0, U_0)
\subset \Psi_{U^{-*}\times U}\left(U^{-*}\times U\times\Omega_\rho(K^{U^{-*},U}_\Ec(s^{-*}_0,s_0))\right).
$$
Put $K_\Ec=K^{U^{-*},U}_\Ec$. The previous inclusion means that 
$
\Vert K_\Ec(s^{-*},t)-K_\Ec(s^{-*}_0,s_0)\Vert\leq\rho$ for $s,t\in U_0$. 
Finally, let $\xi_0$ be in $D$ such that $s_0=\Pi(\xi_0)$. Put 
$M=C_U(\Vert K_\Ec(s^{-*}_0,s_0)\Vert + \Vert x_0\Vert + 2)(2\Vert x_0\Vert+1)$. 
If $\eta\in\Pi^{-1}(U_0)$, $\eta\equiv(t,x)$ such that 
$x\in B_\rho(x_0):=\{y:\Vert y-x_0\Vert\leq\rho\}$ with $0<\rho<1$, then
$$
\begin{aligned}
\Vert K_\eta-K_{\xi_0}\Vert^2
&=(K_\eta\mid K_\eta)-(K_{\xi_0}\mid K_\eta)
-(K_\eta\mid K_{\xi_0})+( K_{\xi_0}\mid K_{\xi_0}) \\
&=(K(t^{-*},t)\eta-K(t^{-*},s_0)\xi_0\mid\eta)_{t^{-*},t}
+(K(s^{-*}_0,s_0)\xi_0-K(s^{-*}_0,t)\eta\mid\xi_0)_{s^{-*}_0,s_0} \\
&=(K_\Ec(t^{-*},t)x-K_\Ec(t^{-*},s_0)x_0\mid x)_{U^{-*},t^{-*}}
+(K_\Ec(s^{-*}_0,s_0)x_0-K_\Ec(s^{-*}_0,t)x\mid x_0)_{U^{-*},s^{-*}_0} \\ 
&\leq C_U
\left(\Vert  K_\Ec(t^{-*},t)-K_\Ec(t^{-*},s_0)\Vert\, \Vert x\Vert 
+\Vert  K_\Ec(t^{-*},s_0)\Vert\, \Vert x-x_0\Vert\right) \Vert x\Vert  \\
&+ C_U
\left(\Vert K_\Ec(s^{-*}_0,s_0)\Vert\,  \Vert x-x_0\Vert
+\Vert  K_\Ec(s_0^{-*},s_0)-K_\Ec(s_0^{-*},t)\Vert\, \Vert x\Vert\right)\Vert x_0\Vert \\
&\leq C_U
\left(\rho(\rho+\Vert x_0\Vert)+(\rho+\Vert  K_\Ec(s_0^{-*},s_0)\Vert)\rho\right)
(\rho+\Vert x_0\Vert) \\
&+
C_U\left(\Vert  K_\Ec(s_0^{-*},s_0)\Vert\rho+\rho(\rho+\Vert x_0\Vert)\right)\Vert x_0\Vert \leq M\rho
\end{aligned}
$$
Since the above inequality holds for every $\eta$ in the neighborhood 
$\Psi_U(U_0\times B_\rho(x_0))$ we have proved that $\widehat K$ is continuous. 
\end{proof}

\begin{remark} 
\normalfont
The argument used in Lemma~\ref{Kcontin} gives us in fact that the mapping
$$
(\xi,\eta)\mapsto K(\Pi(\xi),\Pi(\eta))\eta,\ D\times D\to D
$$
is continuous provided that $K$ is continuous.
\qed
\end{remark}

In the statement of Theorem~\ref{continuzeta} we are going to use the following terminology: 
Two linear subspaces $\Vc_1$ and $\Vc_2$ of 
some Hilbert space $\Hc$ are said to be \emph{similar to each other in} $\Hc$
if there exists an invertible bounded linear operator $T\colon\Hc\to\Hc$ 
such that $T(\Vc_1)=\Vc_2$. 
For instance, it is easy to see that if $\Vc_1$ is finite-dimensional, then 
it is similar to $\Vc_2$ if and only if $\Vc_2$ is also finite-dimensional 
and $\dim\Vc_1=\dim\Vc_2$. 
Note however that if both $\Vc_1$ and $\Vc_2$ are closed infinite-dimensional 
subspaces of some separable Hilbert space $\Hc$ and $\dim\Vc_1^\perp<\infty=\dim\Vc_2^\perp$, 
then $\Vc_1$ and $\Vc_2$ are \emph{not} similar to each other although 
$\dim\Vc_1=\dim\Vc_2$ ($=\infty$). In the theorem we are dealing with the mapping $\zeta$, which is set-valued.

For a topological space ${\mathcal T}$, we say that 
a mapping $F\colon{\mathcal T}\to \Gr(\Hc)$ is {\it similar-valued} if 
$F(\tau)$ is similar to $F(\sigma)$ for every $\tau,\sigma\in{\mathcal T}$. 
Then $F$ will be called {\it locally} similar-valued on ${\mathcal T}$ if for any point $\tau$ in 
${\mathcal T}$ there is a neighborhood $V$ of $\tau$ such that 
$F|_V$ is similar-valued on $V$.

Since every orbit in $\Gr(\Hc)$ is an open set (see for example Corollary~8.4(iv) in \cite{Up85}), 
it is clear that a function $F$ as above  is locally similar whenever it is continuous. 
Conversely, we are going to see that the continuity of the $\Gr(\Hc^K)$-valued mapping $\zeta_K$ 
is equivalent to its local similarity, under the assumption that  $\widehat K$ is open on fibers.

\begin{theorem}\label{continuzeta} In the above setting, suppose that:
\begin{enumerate} 
\item\label{K} The kernel $K$ is continuous,
\item\label{open} the bounded linear operator $\widehat{K}|_{D_s}\colon  D_s\to\Hc^K$ 
is injective with closed range for all $s\in Z$.
\end{enumerate} 
Then $\zeta_K\colon Z\to\Gr(\Hc^K)$ is continuous if and only if it is locally similar.
\end{theorem}

\begin{proof} 
As noticed before, the continuity of $\zeta_K$ implies automatically that 
$\zeta_K$ is locally similar  $\Gr(\Hc)$-valued.
Conversely, assume that $\zeta_K$ is locally similar. Let $s_0\in Z$ arbitrary and pick any 
open neighborhood $V\subseteq Z$ of $s_0$ such that there exists a trivialization 
$\Psi_0\colon V\times\Ec\to\Pi^{-1}(V)$ of the bundle $\Pi$ over~$V$. Also denote $\Hc_0=\zeta_K(s_0)\subseteq\Hc^K$ 
and define the maping
$$
\widetilde{K}_V\colon V\to\Bc(\Ec,\Hc^K), \quad 
s\mapsto \widehat{K}(\Psi_0(s,\cdot)). 
$$
Using an argument as that one of the proof of Lemma~\ref{Kcontin} it can be seen that the mapping $\widetilde{K}_V$ is 
continuous with respect to the norm operator topology: For $s,t\in V$, $x\in\Ec$ and $\xi\equiv(s,x)$, $\eta\equiv(t,x)$, 
$$
\begin{aligned}
\Vert\widetilde K_V(t)x-&\widetilde K_V(s)x\Vert^2 \\
&=(K_\Ec(t^{-*},t)x-K_\Ec(t^{-*},s)x\mid x)_{V^{-*},t^{-*}}+(K_\Ec(s^{-*},s)x
-K_\Ec(s^{-*},t)x\mid x)_{V^{-*},s^{-*}} \\ 
&\leq C_V\left(\Vert  K_\Ec(t^{-*},t)-K_\Ec(t^{-*},s_0)\Vert\,+\Vert  K_\Ec(s_0^{-*},s_0)
-K_\Ec(s_0^{-*},t)\Vert\right)\Vert x\Vert^2
\end{aligned}
$$
from which $\lim_{t\to s}\Vert\widetilde K_V(t)-\widetilde K_V(s)\Vert_{\Ec\to\Hc^K}=0$.

Now, note that the mapping $\widehat{K}|_{D_s}$ being injective with closed range 
means that for every $s\in Z$ there is a constant $C(s)>0$  such that 
$\Vert \widehat K(\xi)\Vert^2\geq C(s)\Vert\xi\Vert_{D_s}^2$ for all $\xi\in D_s$. 
This implies that $\widehat{K}(D_s)$ is a 
{\it closed} linear subspace of $\Hc^K$, so that $\zeta_K(s)=\widehat{K}(D_s)$, and moreover, 
 for every $s\in V$ the operator $\widetilde{K}_V(s)\colon\Ec\to\Hc^K$ 
is injective and its range is the closed subspace $\zeta_K(s)\subseteq\Hc^K$.
For the sake of simplicity let us denote by 
$\widetilde{K}_V(s_0)^{-1}\colon\Hc_0\to\Ec$ 
the inverse of the bijective linear operator 
$x\mapsto\widetilde{K}_V(s_0)x,\ \Ec\to\Hc_0$.

Since $\Hc_0$ is a closed subspace of $\Hc^K$, 
it follows by the open mapping theorem that $\widetilde{K}_V(s_0)^{-1}$ 
is continuous, hence we can define 
$$
\widehat{K}_0\colon V\to\Bc(\Hc_0,\Hc^K), \quad 
s\mapsto \widetilde{K}_V(s)\circ\widetilde{K}_V(s_0)^{-1}.
$$
This mapping is continuous with respect to the norm operator topology 
since so is~$\widetilde{K}_V$. 
In addition, for every $s\in V$ we have $\Ran(\widetilde{K}_V(s))=\zeta_K(s)$, 
hence the hypothesis implies that the ranges of the values of $\widehat{K}_0$ 
are similar to each other in~$\Hc^K$. 
Since the mapping $T\mapsto\Ran T$ is continuous on the latter 
subset of $\Bc(\Hc_0,\Hc^K)$ by Proposition~3.6 in \cite{DEG98} 
(see also Example~6.1 in \cite{DG01} and the principal bundle~(4.5) in \cite{DG02}), 
it then follows that 
the mapping $\zeta_K$ is continuous on~$V$. 
Since $V$ is a neighborhood of the arbitrary point $s_0\in Z$, this completes the proof. 
\end{proof}

\begin{remark}\label{U7}
\normalfont
As regards the continuity of the mapping $T\mapsto\Ran T$ 
on certain sets of operators (see the final argument in the above proof), 
a related result belonging to this circle of ideas 
is provided by Proposition~1.7 in~\cite{MS97}. 
\qed
\end{remark}

\begin{remark}\label{U8}
\normalfont
According to condition \eqref{open} assumed in Theorem~\ref{continuzeta}, 
the operator $K|_{D_s}\colon D_s\to\Hc^K$, being injective and with closed range, provides a topological isomorphism of $D_s$ onto 
a closed subspace of~$\Hc^K$. Thus the topology of the fiber $D_s$ can be defined by 
a suitable scalar product which turns this fiber into 
a complex \emph{Hilbert} space.
In particular, this remark shows that the existence of 
a reproducing $(-*)$-kernel 
$K\colon Z\times Z\to\Hom(p_2^*\Pi,p_1^*\Pi)$ 
satisfying condition~\eqref{open} at some point in the base 
of a like-Hermitian vector bundle $\Pi\colon D\to Z$ 
imposes rather strong conditions on this bundle 
---at least in the sense that the fiber of the bundle 
has to allow for a scalar product which defines its topology. 
(That is, the fiber has to be a \emph{Hilbertable} vector space, 
in the terminology of Chapter~VII in \cite{La01}.) 
Thus condition~\eqref{open} in Theorem~\ref{continuzeta} could appear as being rather restrictive. 
However, the next theorem 
shows that the invertibility property~\eqref{invert} of the kernel $K$ assumed in Corollary~\ref{cor_invert} 
is sufficient to imply that condition. 
\qed
\end{remark}

\begin{theorem}\label{contpull} 
Let $K\colon Z\times Z\to\Hom(p_2^*\Pi,p_1^*\Pi)$ be a reproducing 
$(-*)$-kernel on a 
like-Hermitian vector bundle $\Pi\colon D\to Z$ as in Theorem~\ref{U6}, 
with the corresponding bundle morphism~$\Delta_K$. 
Assume also that:
\begin{enumerate} 
\item\label{Kcontinuando}
$K$ is continuous.
\item\label{invertibility}
For every $s\in Z$,
\begin{equation}\label{invert1}
K(s,s)\colon D_s\to D_s\text{ is invertible}.
\end{equation}
\item\label{localsim} 
The mapping $\zeta_K$ is locally similar in $\Gr(\Hc^K)$.
\end{enumerate}
Then the morphism $\Delta_K$ is continuous.
\end{theorem}

\begin{proof} 
Recall that $\Delta_K=(\delta_K,\zeta_K)$ and $\delta_K=(\zeta\circ\Pi,\widehat K)$. 
By Lemma \ref{Kcontin}, the continuity of $\widehat K$ follows by the fact that $K$ is assumed to be continuous. 
Then  
it is clear that the only thing we need to show is that point~\eqref{invertibility} 
of the statement implies point~\eqref{open} of Theorem~\ref{continuzeta}. 
In order to do so fix $s\in Z$. 
Now observe that since the pair 
$(\cdot\mid\cdot)_{s^{-*},s}$ is a strong duality pairing (see Remark~\ref{sesqui}) 
there exists a bounded linear operator 
$\theta_{s^{-*}}\colon\Hc^K\to D_s$ satisfying
$$
(\widehat K(\eta)\mid h)_{\Hc^K}=(\eta\mid\theta_{s^{-*}}h)_{s^{-*},s}\qquad \forall\eta\in D_s\ , \forall h\in\Hc^K,
$$
see Lemma~3.5 in \cite{BG08}. 
In fact, it turns out that the operator $\theta_{s^{-*}}$ restricted to $\Hc_0^K$ coincides with the evaluation mapping at $s$, see the proof of Proposition~3.7 in \cite{BG08}. 
Thus under the assumption~\eqref{invert1}
 we get, for every $\xi\in D_s$,
$$
\Vert\xi\Vert^2
\leq\Vert K(s,s)^{-1}\Vert^2\Vert K(s,s)\xi\Vert_{D_s}^2 
=\Vert K(s,s)^{-1}\Vert^2  \Vert \theta_{s^{-*}}(K_\xi)\Vert_{D_s}^2
\leq C(s)\Vert\widehat K(\xi)\Vert_{\Hc^K}^2.
$$
This implies that $\widehat{K}|_{D_s}\colon  D_s\to\Hc^K$ is injective with closed range 
for all $s\in Z$, and the theorem follows.
\end{proof}

\begin{remark}\label{U8-1}
\normalfont
Assumptions~\eqref{Kcontinuando} and~\eqref{invertibility} in the preceding theorem are automatically satisfied 
in many important situations: for instance, when $\Pi$ is a line bundle and $K(s,s)\ne 0$, or when 
$K=K^\Rc$ is the canonical kernel associated with a bundle in the category $\Rg$, and with a transfer mapping $\Rc$. 
In fact, in this last case, if $s\in Z$, $\xi\in D_s$ and $\eta\in D_{s^{-*}}$ then for all $s\in Z$ 
we have $K(s,s)=\id$ since 
$$
(K(s,s)\xi\mid\eta)_{s,s^{-*}}=((\Rc_{s^{-*}})^{-*}(\Rc_s\xi)\mid\eta)_{s,s^{-*}}
=(\Rc_s\xi\mid\Rc{s^{-*}}\eta)_\Hc=(\xi\mid\eta)_{s,s^{-*}}.
$$
In these cases the typical fiber of the corresponding bundle is Hilbertable, 
see Remark~\ref{U8}.
\qed
\end{remark}

\begin{remark}\label{inn}
\normalfont
The proof of Corollary~\ref{cor_invert} only requires the injectivity of $\widehat K$ on $D_s$, $s\in Z$. 
So the corollary remains true if one replaces~\eqref{invert} with the weaker condition~\eqref{open} 
of Theorem~\ref{continuzeta}.
\qed
\end{remark}

Assumption~\eqref{localsim} in Theorem~\ref{contpull} holds for objects $(\Pi,G)$ in the category $\GLH$ under a mild condition, 
as we are going to see:
Let $(\Pi,G)$ be an object of the category $\GLH$, with 
$\Pi\colon D\to Z$ under the action of $G$ through 
$\mu\colon G\times D\to D$ and $\nu\colon G\times Z\to Z$. 
Let $K$ be a $(-*)$-reproducing kernel on $(\Pi,G)$ in the sense of Definition~\ref{reprogroup}.

\begin{lemma}\label{zetaorbits}
 For $\zeta_K(s)=\overline{\widehat K(D_s)}$, $s\in Z$, we have 
$$
u\cdot\zeta_K(s)=\zeta_K(\nu(u,s))\qquad\forall u\in G, \forall s\in Z.
$$
\end{lemma}

\begin{proof}
By definition $\Pi(\mu(u,\xi))=\nu(u,\Pi(\xi))$ for all $u\in G$ and $\xi\in D$. 
From this, it follows easily that $D_{\nu(u,s_0)}=\mu(u,D_s)$ for all 
$u\in G$ and $s\in Z$. 

We now claim that the group $G$ is taken into $\Gc=\GL(\Hc^K)$ through the natural action 
(which we denote in this proof by $\lq\lq \cdot "$) of $G$ 
on the section space $\Gamma(Z,D)$ of the bundle $\Pi$. So in particular 
$K_{\mu(u,\xi)}=u\cdot K_\xi$, for $u\in G$ and $\xi\in D_s$.
To see this, notice that, given $u\in G$ and $s\in Z$, for any $t\in Z$ and $\eta\in D_t$,
$$ 
\begin{aligned}
(K_{\mu(u,\xi)}\mid K_\eta)
&:=(K(t^{-*},\nu(u,s))\mu(u,\xi)\mid\eta)_{t^{-*},t} 
=(\mu(u,K(\nu(u^{-1},t^{-*}),s)\xi)\mid\eta)_{t^{-*},t} \\ 
&=(\mu(u,K_\xi(\nu(u^{-1},t^{-*})))\mid\eta)_{t^{-*},t}=(u\cdot K_\xi\mid K_\eta)
\end{aligned}
$$
by~\eqref{kgroup}.
By the density of the linear combinations of sections $K_\eta$ in $\Hc^K$ 
we obtain the wished-for relation $K_{\mu(u,\xi)}=u\cdot K_\xi$, for $u\in G$ and $\xi\in D_{s}$.

Let $u\in G$ again. 
Since the action of $u$ on the space 
$\Gamma(Z,D)$ is continuous we have 
$u\cdot\overline{\widehat K(D_s)}
\subseteq\overline{u\cdot \widehat K(D_s)}$ and then
$u^{-1}\cdot\overline{u\cdot \widehat K(D_s)}\subseteq
\overline{(u^{-1}u)\cdot \widehat K(D_s)}=\overline{\widehat K(D_s)}$. 
Both inclusions give us $u\cdot\overline{\widehat K(D_s)}
=\overline{u\cdot \widehat K(D_s)}$. 
In consequence,  
$$
\zeta_K(\nu(u,s))=\overline{\widehat K(D_{\nu(u,s)})}=
\overline{\widehat K(\mu(u,D_s))}=
\overline{\widehat u\cdot K(D_s)}=u\cdot\zeta_K(s).
$$
\end{proof}

\begin{corollary} 
Let $(\Pi,G)$ be an object in the category $\normalfont {\GLH}$, and let $K$ be a continuous 
reproducing kernel 
on $\Pi$ in this category. 
Suppose, in addition to the assumption of Theorem~\ref{U6},
that for every $s\in Z$ the isotropy group $G_s:=\{u\in G\mid\nu(u,s)=s\}$ 
is a Lie subgroup of $G$ and that the mapping $K(s,s)$ is invertible. 
Then the bundle morphism $\Delta_K$ is continuous. 
\end{corollary}

\begin{proof} 
In view of previous results, we only have to verify that the mapping $\zeta$ is locally similar. 
But this is a clear consequence of the above lemma, since the fact that $G_s$ is a Lie subgroup 
entails that $\{\nu(u,s)\mid u\in G\}$ is an open subset of Z for every $s\in Z$.
\end{proof}

\begin{remark}\label{U9}
\normalfont
Let us discuss briefly the condition in Theorem~\ref{U6} on 
the existence of a suitable involutive diffeomorphism 
of $\Gr(\Hc^K)$ onto itself in the case when the base $Z$ 
of the bundle $\Pi\colon D\to Z$ is \textit{connected}. 

Assuming that
 the mapping 
$\zeta_K\colon Z\to\Gr(\Hc^K)$ is continuous, 
the image of this map will be a connected subset of 
$\Gr(\Hc^K)$. 
In particular, $\zeta_K(Z)$ will be contained in a connected component 
of $\Gr(\Hc^K)$. 
It then follows by Remark~4.2(b) in \cite{BG07} that 
there exists a closed linear subspace $\Sc_0$ of $\Hc$ 
such that $\zeta_K(Z)\subseteq\Gr_{\Sc_0}(\Hc)$. 
So condition~\eqref{compatible} of Theorem~\ref{U6} 
actually requires a suitable involutive diffeomorphism 
\textit{only on the unitary orbit} $\Gr_{\Sc_0}(\Hc)$, 
not on the whole $\Gr(\Hc^K)$.

Thus the role of universal reproducing $(-*)$-kernel  
seems to be played by the reproducing $(-*)$-kernels 
induced by $Q_{\Hc}$  
on the restrictions of 
the tautological vector bundle $\Theta_{\Hc}$ 
to the various unitary orbits $\Gr_{\Sc_0}(\Hc)$ 
endowed with \textit{various} $(-*)$-structures, 
for arbitrary Hilbert spaces $\Hc$.

Another situation when we are naturally led to deal with orbits 
$\Gr_{\Sc_0}(\Hc)$ of Grassmannians is that one considered prior to 
Lemma~\ref{zetaorbits}, in the case when 
the action of the group $G$ on $Z$ is transitive. 
Under such a condition, we manage in the next result to remove the assumption about 
the existence of the involutive morphism $\Omega$ satisfying \eqref{compatible} 
(which is a highly {\it non-canonical} assumption). 
In order to do so, we must replace the tautological bundle $\Pi_\Hc$ 
by its complexification $\Pi_\Hc^\C$.  
\qed
\end{remark}

Let $(\Pi,G)$ be an object of the category $\GLH$, as above, with 
$\Pi\colon D\to Z$, $\mu\colon G\times D\to D$, 
$\nu\colon G\times Z\to Z$, such that the action $\nu$ is \emph{transitive} on $Z$.
Assume that there exists $s_0\in Z$ with $s_0^{-*}=s_0$. 
Let $K$ be a reproducing $(-*)$-kernel on $\Pi$ satisfying the property~\eqref{kgroup}, 
and let $\Hc^K$ be its reproducing kernel Hilbert space. 
For $\widehat K$ as in Theorem~\ref{U6}, set 
$\Sc_0:=\overline{\widehat K(D_{s_0})}\subseteq\Hc^K$. 
Recall that we denote the vector bundle $\Gc\times_{\Gc(p)}\Sc_0\to\Gc/\Gc(p)$ by
$\Pi_{\Hc^K}^\C$. 
Here $p=p_{\Sc_0}$ and $\Gc=\GL(\Hc^K)$. 
Let in turn $Q_{\Hc^K}^\C$ denote the kernel associated with the bundle $\Pi_{\Hc^K}^\C$ 
and the projection $p$ as in the subsection~\ref{subsect1.3}. 

\begin{theorem}\label{U10}
In the preceding setting, there exists a vector bundle morphism 
$\widetilde\Delta_K=(\tilde\delta_K,\tilde\zeta_K)$ from $\Pi$ into 
$\Pi_{\Hc^K}^\C$ such that $K$ is equal to the pull-back of the 
reproducing $(-*)$-kernel $Q_{\Hc^K}^\C$, that is, 
$$
K=(\widetilde\Delta_K)^*Q_{\Hc^K}^\C.
$$   
In addition, if 
\begin{enumerate}
\item the isotropy group at the point $s_0=s_0^{-*}\in Z$ 
is a complex Banach-Lie subgroup of $G$ and 
\item the reproducing $(-*)$-kernel $K$ is holomorphic of the second kind,  
\end{enumerate}
 then the vector bundle morphism 
$\widetilde\Delta_K$ is holomorphic. 
\end{theorem}

\begin{proof}
As it was seen in Lemma \ref{zetaorbits}, $\zeta_K(\nu(u,s_0))=u\cdot\zeta_K(s_0)$ for every $u\in G$.
Let us proceed to constructing the morphism $\widetilde\Delta_K$. 
Define, for 
$s=\nu(u,s_0)$ in $Z$,
\begin{equation}\label{zeta}
\tilde\zeta_K(s):=(\zeta_K(\nu(u,s_0)),\zeta_K(\nu(u^{-*},s_0)))
=(u\cdot\Sc_0,u^{-*}\cdot\Sc_0)\equiv u\ \Gc(p).
\end{equation}
Then, on account of the definition of the involution $-*$ in $\Gc(p)$, 
we have
$$ 
\tilde\zeta_K(s)^{-*}=[u\cdot\Gc(p)]^{-*}:=u^{-*}\cdot\Gc(p)
=\tilde\zeta_K(\nu(u^{-*},s_0))
=\tilde\zeta_K(\nu(u,s_0)^{-*})=\tilde\zeta_K(s^{-*}).
$$
Take $\xi\in D$. 
Then $\xi\in D_s$ for some (unique) $s=\nu(u,s_0)$, where 
$u\in G$. 
Define $\tilde\delta_K(\xi)$ in $\Gc\times_{\Gc(p)}\Sc_0$  by
\begin{equation}\label{delta}
\tilde\delta_K(\xi)=[(u,K_{\mu(u^{-1},\xi)})].
\end{equation}
It is readily seen that $\tilde\delta_K$ is well defined and then that 
$(\tilde\delta_K,\tilde\zeta_K)$ is a morphism between 
the like-Hermitian vector bundles $\Pi$ and $\Pi_{\Hc_K}^\C$. 
Put for a while $\delta=\tilde\delta_K$ and 
$\zeta=\tilde\zeta_K$ for simplicity.

For 
$s=\nu(u,s_0),\ t=\nu(v,s_0)\in Z$ and $\xi\in D_{s^{-*}}$, $\eta\in D_t$, we have
$$
\begin{aligned}
&(((\delta_{s^{-*}} )^{-*}\circ Q_{\Hc^K}^\C(\zeta(s),\zeta(t))\circ\delta_t )\eta
\mid\xi )_{s,s^{-*}} \\
&=((\delta_{s^{-*}} )^{-*}
[(u,p( u^{-1}\cdot v\cdot K_{\mu(v^{-1},\eta)}))]   
\mid\xi )_{s,s^{-*}} \\
&=([(u,p(u^{-1}\cdot v\cdot K_{\mu(v^{-1},\eta)}))]
\mid[u^{-*},K_{\mu(u^{-*},\xi)})])_{u\cdot\Gc(p),u^{-*}\cdot\Gc(p)} \\
&=(p_{\Sc_0}(u^{-1}\cdot v\cdot K_{\mu(v^{-1},\eta)})
\mid K_{\mu(u^{-*},\xi)})_{\Hc^K} 
=(K(s,t)\eta\mid\xi)_{\Hc^K},
\end{aligned}
$$
This implies that $K=(\widetilde\Delta_K)^*\ Q_{\Hc^K}^\C$, that is, $K$ is the pull-back of $Q_{\Hc^K}^\C$ 
by the morphism~$\widetilde{\Delta}_K$.

To prove the second part of the statement, recall that if the isotropy 
group at $s_0\in Z$ is a complex Banach-Lie subgroup of $G$, 
then the orbit mapping 
$$
u\mapsto\nu(u,s_0),\quad G\to Z
$$
has holomorphic local cross-sections around every point $s\in Z$. 
Therefore, in the above construction of the vector bundle morphism 
$\widetilde\Delta_K=(\tilde\delta_K,\tilde\zeta_K)$ 
(see formulas \eqref{zeta}~and~\eqref{delta}), 
the element $u\in G$ can be chosen to depend holomorphically on $s\in Z$ 
around an arbitrary point in $Z$. 
This remark implies at once that 
the mapping $\tilde\zeta_K$ is holomorphic and that in order to prove that 
the component $\tilde\delta_K$ is holomorphic 
as well, it will be enough to check that the mapping 
$$
\widehat{K}\colon \xi\mapsto K_\xi,\ D\to\Hc^K
$$
is holomorphic. Since $K$ is continuous, by Lemma \ref{Kcontin} we have that 
$\widehat{K}$ is continuous.
 Then, by Corollary~A.III.3 in Appendix~III of \cite{Ne00}, 
it will be sufficient to show that for every $h\in\Hc^K_0$ the function 
$(\widehat{K}(\cdot)\mid h)_{\Hc^K}$ is holomorphic on~$D$. 
Thus take $\xi,\eta\in D$, $s=\Pi(\xi)$, $t=\Pi(\eta)$. 
Then, by \eqref{prods}~and~\eqref{exchange},
$$
(\widehat{K}(\xi)\mid K_\eta)_{\Hc^K}
=(K(t^{-*},s)\xi\mid \eta)_{t^{-*},t}
.
$$
Hence, by the hypothesis that the kernel $K$ is holomorphic of the second kind (see Definition~\ref{kernel}), 
we obtain that
$\widehat{K}\colon D\to\Hc^K$ is holomorphic, 
and this concludes the proof. 
\end{proof}

\begin{remark} 
\normalfont
There is a close relationship between Theorem~\ref{compare} and Theorem~\ref{U10}, 
and yet they are different from each other, for in Theorem~\ref{U10} it is not assumed the existence of 
a transfer mapping (relating the bundle $\Pi\colon D\to Z$ with some Hilbert space $\Hc$). 
Instead, one needs to consider in Theorem~\ref{U10} the Hilbert space of sections $\Hc^K$. 
Is there some transfer mapping, say $\Rc_K$, from $D$ into~$\Hc^K$? 
A natural candidate seems to be the mapping
$\Rc_K\colon\xi\mapsto[(u,K_{\mu(u^{-1},\xi)})]\mapsto u\cdot K_{\mu(u^{-1},\xi)}$.
Since for all $u\in G$, $\xi\in D$ we have $u\cdot K_{\mu(u^{-1},\xi)}=K_\xi$ such a map is $\widehat K$. 
However $\widehat K$ need not be an isometry, or even injective, in general. 
(Recall that 
$(\widehat K(\xi)\mid\widehat K(\eta))_{\Hc^K}=(K(s,s)\xi\mid\eta)_{s,s^{-*}}$ for 
$s\in Z$, $\xi\in D_s$, $\eta\in D_{s^{-*}}$.)
\qed 
\end{remark}

\section{Applications of the pull-back operation on reproducing $(-*)$-kernels}\label{sect6}

The first application of the pull-back operation works out 
the precise relationship between reproducing $(-*)$-kernels 
on homogeneous vector bundles constructed as in the subsection~\ref{subsect1.4}. 
The resulting result, part (d), illustrates Proposition~\ref{cat5}.

\begin{proposition}\label{cat7}
Given two $C^*$-algebras $\1\in B\subseteq A$ with  
a conditional expectation $E\colon A\to B$, 
assume that we have a tracial state 
$\varphi\colon A\to{\mathbb C}$ (i.e., $\varphi(a_1a_2)=\varphi(a_2a_1)$  
for every $a_1a_2\in A$) such that $\varphi\circ E=\varphi$.
For $X\in\{A,B\}$ denote by $G_X$ the group of invertible elements of $X$, 
by $\Hc_X$ the Hilbert space obtained 
by the GNS construction out of the state $\varphi|_X$ of $X$, 
and by $C_X\colon\Hc_X\to\Hc_X$ the antilinear isometric involutive 
operator defined by the involution of $X$. 
Also denote by $P\colon\Hc_A\to\Hc_B$ the orthogonal projection and by 
$\lambda,\rho,\pi\colon A\to\Bc(\Hc_A)$
the representations defined by 
$$
\lambda(u)[f]=[uf],\quad \rho(u)[f]=[fu^{-1}],\quad 
\text{and}\quad \pi(u)[f]=\lambda(u)\rho(u)[f]=[ufu^{-1}]
$$
whenever $u\in G_B$ and $f\in B$, 
where we denote by $f\mapsto[f]$ the natural map $B\to\Hc_B$.

Now for $\alpha\in\{\lambda,\rho,\pi\}$ 
denote by $K_\alpha$ the corresponding reproducing $(-*)$-kernel 
on the homogeneous bundle $G_A\times_\alpha\Hc_B\to G_A/G_B$. 
Then the following assertions hold: 
\begin{itemize} 
\item[{\rm(a)}] There exist well-defined diffeomorphisms  
$$
\delta_1\colon G_A\times_\lambda\Hc_B\to G_A\times_\rho\Hc_B,\quad 
\delta_2\colon  G_A\times_\rho\Hc_B\to G_A\times_\lambda\Hc_B,\quad\text{and}\quad
\delta_3\colon  G_A\times_\pi\Hc_B\to G_A\times_\pi\Hc_B,$$
given, each of them, by the correspondence 
$[(u,f)]\mapsto[(u^{-*},C_B(f))]$. 
\item[{\rm(b)}] For $j=1,2,3$ denote by $\Theta_j$ the pair consisting 
of the mappings $\delta_j$ and $uG_B\mapsto u^{-*}G_B$. 
Then $\Theta_j$ is an adjointable antimorphism of 
like-Hermitian bundles. 
\item[{\rm(c)}] We have $\Theta_1^*K_\rho=K_\lambda$, 
$\Theta_2^*K_\lambda=K_\rho$, and $\Theta_3^*K_\pi=K_\pi$. 
\item[{\rm(d)}] There exists an involutive antilinear isometry 
$\overline{\Theta}_3\colon\Hc_{\mathbb C}^{K_\pi}\to\Hc_{\mathbb C}^{K_\pi}$ 
such that 
$\overline{\Theta}_3((K_\pi)_\xi)=(K_\pi)_{\tau(\xi)}$ for all $\xi\in D$. 
\end{itemize}
\end{proposition}

\begin{proof} First of all, notice that $\varphi$ being tracial 
the representation $\rho$ is well defined, so is $\pi$.
To prove assertion~(a), 
note that for all $v\in G_B$ and $f\in B$ we have 
\begin{equation}\label{inter}
C_B(\lambda(v^{-1})[f])=[(v^{-1}f)^*]
=[f^*(v^*)^{-1}]=\rho((v^{-*})^{-1})C_B([f]).
\end{equation}
Now, if $(u_1,f_1)\sim_\lambda(u_2,f_2)$ 
then there exists $v\in G_B$ such that $u_2=u_1v$ and 
$f_2=\lambda(v^{-1})f_1$. 
Then the above calculation shows that 
$C_B(f_2)=\rho((v^{-*})^{-1})C_B(f_1)$, 
so that $(u_1^{-*},C_B(f_1))\sim_\rho(u_2^{-*},C_B(f_2))$. 
Thus the mapping 
$\delta_1\colon G_A\times_\lambda\Hc_B\to G_A\times_\rho\Hc_B$ 
is well defined.

Since the projection mappings $G_A\times\Hc_B\to G_A\times_\alpha\Hc_B$ 
are submersions for $\alpha\in\{\lambda,\rho\}$ 
and the mapping $(u,f)\mapsto(u^{-*},C_B(f))$ 
is clearly a diffeomorphism of $G_A\times\Hc_B$ onto itself, 
it then follows that $\delta_1$ is smooth 
(see Corollary~8.4 in \cite{Up85}). 
One can similarly check that 
$\delta_2\colon G_A\times_\rho\Hc_B\to G_A\times_\lambda\Hc_B$ 
is smooth. 
In addition, it is easy to see that the mappings $\delta_1$ and $\delta_2$ 
are inverse to each other, hence they are diffeomorphisms.

The fact that $\delta_3$ is also a well-defined smooth mapping 
follows by a similar reasoning 
based on the calculation
$$
C_B(\pi(v^{-1})[f])=[(v^{-1}fv)^*]=[v^*f^*v^{-*}]
=\pi((v^{-*})^{-1})C_B([f])
$$
that holds whenever $v\in G_B$ and $f\in B$. 
Since $\delta_3\circ\delta_3=\id_{G_A\times_\pi\Hc_B}$, 
it then follows that $\delta_3$ is a diffeomorphism as well.

For assertion~(b) we shall use Remark~\ref{morph2}.
It is clear that for $j=1,2,3$ the mapping 
$\delta_j$ is a fiberwise isomorphism 
of real Banach spaces, so that it will be enough 
to prove that for all $u\in G_A$ and $f\in\Hc_B$ we have
$$
\bigl([(u,f)]\mid[(u^{-*},g)]\bigr)_{uG_B,u^{-*}G_B}
=\overline{\bigl(\delta_j[(u,f)]
  \mid\delta_j[(u^{-*},g)]\bigr)}_{u^{-*}G_B,u^{-*}G_B}.
$$
This equality is equivalent to 
$(f\mid g)_{\Hc_B}=\overline{(C_B(f)\mid C_B(g))}_{\Hc_B}$, 
which is true since $C_B\colon\Hc_B\to\Hc_B$ is an antilinear isometry.

Assertion~(c) can be proved by means of~\eqref{star_anti}. 
For instance, in order to check that $\Theta_1^*K_\rho=K_\lambda$, 
we need to prove that 
for all $u_1,u_2\in G_A$ and $f_1,f_2\in\Hc_B$ we have 
$$
\begin{aligned}
\bigl(K_\lambda(u_1G_B,u_2G_B)[(u_2,f_2)]
&  \mid[(u_1^{-*},f_1)]\bigr)_{u_1G_B,u_1^{-*}G_B} \\
&=\overline{\bigl(K_\rho(u_1^{-*}G_B,u_2^{-*}G_B)\delta_1[(u_2,f_2)]
  \mid\delta_1[(u_1^{-*},f_1)]\bigr)}_{u_1^{-*}G_B,u_1G_B}.
\end{aligned}
$$
In view of the way $\delta_1$ is defined, 
the right-hand side is equal to 
$$
\overline{\bigl(K_\rho(u_1^{-*}G_B,u_2^{-*}G_B)[(u_2^{-*},C_B(f_2))]
  \mid[(u_1,C_B(f_1))]\bigr)}_{u_1^{-*}G_B,u_1G_B}.
$$ 
Now, by taking into account the definitions of $K_\lambda$ and $K_\rho$ 
(see \cite{BG08}), 
it follows that the wished-for equality is equivalent to 
$$
\begin{aligned}
\bigl([(u_1,P(\lambda(u_1^{-1})\lambda(u_2)f_2))]
&  \mid[(u_1^{-*},f_1)]\bigr)_{u_1G_B,u_1^{-*}G_B} \\
& = 
 \overline{([(u_1^{-*},P(\rho((u_1^{-*})^{-1})\rho(u_2^{-*})C_B(f_2)))]
  \mid[(u_1,C_B(f_1))]\bigr)}_{u_1^{-*}G_B,u_1G_B}.
\end{aligned}
$$
This is further equivalent to 
\begin{equation}\label{re}
(P(\lambda(u_1^{-1}u_2)f_2)\mid f_1)_{\Hc_B}
 =\overline{(P(\rho(((u_1^{-*})^{-1}u_2^{-*})C_B(f_2))
     \mid C_B(f_1)))}_{\Hc_B}. 
\end{equation}
Now we have 
$$
\begin{aligned}
(P(\rho(((u_1^{-*})^{-1}u_2^{-*})C_B(f_2))\mid C_B(f_1))_{\Hc_B}
&=(\rho(((u_1^{-*})^{-1}u_2^{-*})C_B(f_2)\mid P(C_B(f_1)))_{\Hc_B} \\
&=(\rho(((u_1^{-1}u_2)^{-*})C_B(f_2)\mid C_B(f_1))_{\Hc_B} \\
&\stackrel{\scriptstyle\eqref{inter}}{=}
 (C_B(\lambda(((u_1^{-1}u_2)^{-*})f_2)\mid C_B(f_1))_{\Hc_B} \\
&=(f_1\mid \lambda(((u_1^{-1}u_2)^{-*})f_2)_{\Hc_B} \\
&=(f_1\mid P(\lambda(((u_1^{-1}u_2)^{-*})f_2))_{\Hc_B} \\
&=\overline{(P(\lambda(((u_1^{-1}u_2)^{-*})f_2)\mid f_1)}_{\Hc_B}
\end{aligned}
$$
which shows that \eqref{re} holds. 
This completes the proof of the fact that $\Theta_1^*K_\rho=K_\lambda$. 

The equalities $\Theta_2^*K_\lambda=K_\rho$ and $\Theta_3^*K_\pi=K_\pi$ 
can be proved similarly. 

To prove assertion~(d), just note that it follows by (c) 
that \eqref{symm} in Corollary~\ref{cat6} 
is satisfied. 
\end{proof}

As regards item~(d) in the statement of Proposition~\ref{cat7}, 
it is to be noticed that the diagram 
$$
\begin{CD}
\Hc^{K_\pi} @<{\gamma}<< \Hc_A\\
@V{\overline{\Theta}_3}VV @VV{C_A}V \\
\Hc^{K_\pi} @<{\gamma}<< \Hc_A
\end{CD}
$$
is commutative, where the mapping $\gamma$ is a suitable isometry
(see \cite{BG08} and the comments after Theorem~\ref{6.1}).

\medskip
The next application is a special case of Theorem~\ref{U6} and yet 
it applies to all of the classical reproducing kernels
(see for instance \cite{Ha82} and Example~I.1.10 in~\cite{Ne00} 
for many specific examples) 
when they are thought of as living on trivial line bundles. 
Moreover, note that the classical reproducing kernel spaces of complex functions 
on domains in ${\mathbb C}^n$ 
(Bergman spaces, Hardy spaces etc.)  
are \emph{separable} Hilbert spaces 
hence are isomorphic to $\ell^2({\mathbb N})$. 
Hence the reproducing kernels of all these classical function spaces 
can be obtained as pull-backs of a \emph{unique} kernel, 
namely the universal reproducing kernel $Q_{\Hc}$ 
for $\Hc=\ell^2({\mathbb N})$. 
See \cite{Od88}~and~\cite{Od92} for the implications of this fact in quantum mechanics.

\begin{theorem}\label{6.1}
Let $\Pi\colon D\to Z$ be a Hermitian vector bundle. 
Denote by $p_1,p_2\colon Z\times Z\to Z$ the natural projections 
and let $K\colon Z\times Z\to\Hom(p_2^*\Pi,p_1^*\Pi)$ 
be a reproducing kernel. 
If $\Hc^K$ stands for 
the reproducing kernel Hilbert space associated with $K$, 
then there exists a vector bundle homomorphism 
$\Delta_K=(\delta_K,\zeta_K)$ from $\Pi$ into the tautological vector bundle  
$\Pi_{\Hc^K}$ such that 
$K$ is equal to the pull-back of $Q_{\Hc^K}$ by $\Delta_K$. 
\end{theorem}

\begin{proof}
The hypothesis of Theorem~\ref{U6} is clearly satisfied 
if we take $\Omega$ to be the identity morphism on 
the tautological vector bundle $\Pi_{\Hc}$. 
\end{proof}

Finding applications of Theorem~\ref{U6} in other concrete cases involving general involutions $-*$, 
for which condition~\eqref{compatible} is automatically fulfilled, does not seem simple. 
Instead, one has to use complexifications of bundles 
of the type $\Tc_{\Sc_0}(\Hc)\to\Gr_{\Sc_0}(\Hc)$,  $\Sc_0\in\Gr(\Hc)$, as in Theorem~\ref{U10}.
This theorem clearly applies to bundles $G_A\times_{G_B}\Hc_B\to G_A/G_B$  in the category $\HLH$ and kernels $K^\pi$ accordingly, as in Remark~\ref{kapi}. 
Let us  now assume for simplicity, in that case, that $\Hc_A=\overline{\spann}\, \pi_A(G_A)\Hc_B$. 
(This holds for instance when $A$, $B$ are $C^*$-algebras and 
$G_A\times_{G_B}\Hc_B\to G_A/G_B$ is an object in $\SLH$.)
Then $\Hc_A$ and $\Hc^{K_\pi}$ are isomorphic Hilbert spaces via the  unitary \lq\lq realization" operator 
$\gamma\colon h\mapsto F_h,\ \Hc_A\to\Hc^{K_\pi}$ where 
$F_h\colon u\G_B\mapsto[(u,P(\pi^{-1}(u)h))],\ G_A\times_{G_B}\Hc_B\to G_A/G_B$ (see \cite{BG08}). 
The translation of Theorem~\ref{U10} in the above situation is that 
$K^\pi=\widetilde\Delta_{K^\pi}^*Q_{\Hc_A}^\C$ where 
$\widetilde\Delta_{K^\pi}=(\tilde\delta_{K^\pi},\tilde\zeta_{K^\pi})$ with 
$\tilde\delta_{K^\pi}=\pi_A\times\id_{\Hc_B}$ and $\tilde\zeta_{K^\pi}=(\pi_A)_q$. 
In particular the morphism $(\pi_A\times\id_{\Hc_B},(\pi_A)_q)$ is 
a (extremal) morphism between the kernels $K^\pi$ and $Q_{\Hc_A}^\C$.

To go further in this direction, note that it has been proven in \cite{BG08} that 
the representation of $G_A$ on $\Hc_A$ can be realized as 
multiplication of the elements of $G_A$ on the Hilbert space $\Hc^K$. 
This realization is implemented by the formula
$$
\pi_A(u)=\gamma^*\circ(u\ \cdot\ )\circ\gamma\qquad (u\in G_A)
$$
In this sense, and according to what has been indicated formerly, 
one can say that $\pi_A$ is the pull-back of 
the natural multiplication (collineation) operator associated with 
the Stiefel bundle 
$\Gc\times_{\Gc(p)}\Ran p\to\Gc/{\Gc(p)}$.

Finally, as our last application here, we show next a similar result for completely positive mappings.
So let us assume once again the setting of subsection~\ref{subsect1.4}. Then
$\Phi(a)=V^*\circ\pi_A(a)\circ V$, ($a\in A$), where
$\Phi\colon A\to\Bc(\Hc_0)$ is a unital completely positive mapping,
$\pi_A\colon A\to\Bc(\Hc_A)$ is a Stinespring dilation of $\Phi$, and 
$V\colon \Hc_0\to\Hc_A$ is the corresponding isometry between the Hilbert spaces $\Hc_0$ and $\Hc_A$. 
Note that $V(\Hc_0)$ is a Hilbert subspace of $\Hc_A$.
Let $\Phi_{V(\Hc_0)}$ denote the unital, completely positive mapping
$$ 
\Phi_{V(\Hc_0)}\colon T\mapsto p_{V(\Hc_0)}\circ T\circ\iota_{V(\Hc_0)},\ 
\Bc(\Hc_A)\to\Bc(V(\Hc_0))
$$
associated with $V(\Hc_0)$ as in Proposition~\ref{cpgras}.

\begin{lemma}\label{pullcom}
In the above notation, for every $a\in A$ we have 
$\Phi(a)=V^*\circ\Phi_{V(\Hc_0)}(\pi_A(a))\circ V$.
\end{lemma}

\begin{proof} 
First note that there exists the composition $\iota_{V(\Hc_0)}\circ V$ and 
$$
V^*=(\iota_{V(\Hc_0)}\circ V)^*=V^*\iota_{V(\Hc_0)}^*=V^*p_{V(\Hc_0)}.
$$
Hence, for each $a\in A$,
$$
V^*\Phi_{V(\Hc_0)}(\pi_A(a))V
= V^*(p_{V(\Hc_0)})\pi_A(a)\iota_{V(\Hc_0)})V 
=(V^*p_{V(\Hc_0)})\pi_A(a)(\iota_{V(\Hc_0)}V)
=V^*\pi_A(a)V=\Phi(a),
$$
as we wanted to show.
\end{proof}

Looking at $\Phi$ and $\Phi_{V(\Hc_0)}\circ\pi_A$ as sections of appropriate trivial bundles, 
as in Example~\ref{triv}, Lemma~\ref{pullcom} says that 
$\Phi$ can be regarded as the pull-back of $\Phi_{V(\Hc_0)}\circ\pi_A$. 
Moreover, we have the following result.

\begin{proposition}\label{morcp}
The pair $(\pi_A,V)$ is a (extremal, with $M=1$) morphism in 
the category $\Cp$ from $\Phi$ into $\Phi_{V(\Hc_0)}$.
\end{proposition}

\begin{proof}
Let $p$ be the orthogonal projection from $\Hc_A$ onto $\Hc_B$. 
Take $b$ in $B$. 
We know that $\pi_A(b)\in\{p\}'$, see subsection~\ref{subsect1.4}. 
Thus 
$E_p \pi_A(b)=p\pi_A(b)p+(1-p)\pi_A(b)(1-p)=\pi_A(b)[p+(1-p)]=\pi_A(b)$.

By Lemma~\ref{pullcom} we get 
$V\Phi(a)=\Phi_{V(\Hc_0)}(\pi_A(a))$ for all $a\in A$. 
In fact, for every $x_0,y_0\in\Hc_0$,
$$
(V\Phi(a)x_0\mid Vy_0)_{V(\Hc_0)}
=(\Phi(a)x_0\mid y_0)_{\Hc_A} 
=(V^*\Phi_{V(\Hc_0)}Vx_0\mid y_0)_{\Hc_A}
=(\Phi_{V(\Hc_0)}Vx_0\mid Vy_0)_{V(\Hc_0)}.
$$
In particular the equality holds for every $b\in B\subseteq A$.
Consequently we have shown that $(\pi_A,V)$ is a morphism in the category $\SLH$, 
see after Lemma~\ref{two_squares} in subsection~\ref{subsect1.4}.

Now, for every integer $n$, $h_1,\dots, h_n\in\Hc_0$ and a matrix $(a_{ij})\subseteq M_n(A)$ we have 
$$
\begin{aligned}
\sum_{i,j=1}^n(\Phi_{V(\Hc_0)}(\pi_A(a_{ij}))Vh_j\mid Vh_i)_{V(\Hc_0)}
&=\sum_{i,j=1}^n((V^*\Phi_{V(\Hc_0)}(\pi_A(a_{ij})V)h_j\mid Vh_i)_{\Hc_A} \\
&=\sum_{i,j=1}^n (\Phi(a_{ij})h_j\mid h_i)_{\Hc_A}.
\end{aligned}
$$
This means that, in fact, $(\pi_A,V)$ is a morphism in $\Cp$.
\end{proof}

In summary, in what concerns completely positive mappings, it has been shown in the paper that they can be viewed as objects in a category, called here $\Cp$. 
Proposition~\ref{morcp} says furthermore that completely positive mappings in that category are in fact the pull-back of canonical objects, which are universal in this sense, and are the completely  positive mappings associated to Grassmannian tautological bundles, as given by Proposition~\ref{cpgras}.

\appendix 

\section*{Appendix: List of categories}

\noindent $\LH$ 
(end of subsection~\ref{subsect1.1}); 
$\GLH$ (Definition~\ref{relationship}); 
$\Rg$ (Definition~\ref{transrel}); 
$\HLH$ (subsection~\ref{subsect1.3}); 
$\SLH$ (after Remark~\ref{two_sq}); 
$\Kern$ (Definition~\ref{cat1}); 
$\Hilb$ (Definition~\ref{cat3}); 
$\Tran$, $\Rep$, $\Cp$ (Remark~\ref{SevCat}).

\end{document}